\documentclass[reqno, 11pt]{amsart}

\usepackage[T1]{fontenc}

\usepackage{latexsym}
\usepackage{amssymb}
\usepackage{amsmath}
\usepackage{amsthm}
\usepackage{mathrsfs}
\usepackage{stmaryrd}
\usepackage{url}
\usepackage{longtable}
\usepackage{color}
\usepackage{dsfont}
\usepackage[normalem]{ulem}

\usepackage{lmodern}
\usepackage[linktocpage]{hyperref}
\usepackage{mathtools}
\usepackage{tikz}
\usepackage{pgfplots}
\pgfplotsset{my style/.append style={axis x line=left, axis y line=  
left, axis equal }}

\usepackage{enumitem}

\usepackage{graphicx}
\usepackage{subcaption}
\usepackage[latin1]{inputenc}
\usepackage{enumitem}
\usepackage{amsopn}
\usepackage{comment}
\newcommand{\beq}{\begin{equation}}
\newcommand{\eeq}{\end{equation}}

\makeatletter
\newcommand{\labelword}[2]{%
  \phantomsection
  #1\def\@currentlabel{\unexpanded{#1}}\label{#2}%
}
\makeatother

\numberwithin{equation}{section}

\newtheorem{theorema}{Theorem}

\theoremstyle{plain}
\newtheorem{theorem}{Theorem}[section]
\newtheorem{lemma}[theorem]{Lemma}

\newtheorem{proposition}[theorem]{Proposition}
\newtheorem{cor}[theorem]{Corollary}
\theoremstyle{remark}
\newtheorem{rem}[theorem]{Remark}

\newtheorem*{notation*}{Notation}
\theoremstyle{definition}
\newtheorem{definition}[theorem]{Definition}
\newtheorem{example}[theorem]{Example}
\newtheorem{assumption}{Assumption}
\newtheorem{conj}{Conjecture}

\newcommand{\bbZ}{\mathbb{Z}}

\newcommand{\bbE}{\mathbb{E}}

\newcommand{\bone}{\mathbf{1}}

\newcommand{\cF}{\mathcal{F}}
\newcommand{\cP}{\mathcal{P}}

\newcommand{\cA}{\mathcal{A}}

\newcommand{\R}{\mathbb{R}}
\newcommand{\bbR}{\mathbb{R}}

\newcommand{\N}{\mathbb{N}}
\newcommand{\E}{\mathbb{E}}
\renewcommand{\P}{\mathbb{P}}
\newcommand{\bbP}{\mathbb{P}}

\newcommand{\cals}{{\mathcal S}}

\newcommand{\cC}{\mathcal C}
\newcommand{\cL}{\mathcal L}
\newcommand{\cY}{\mathcal Y}
\newcommand{\cG}{\mathcal G}

\newcommand{\calu}{{\mathcal U}}
\newcommand{\calf}{{\mathcal F}}
\newcommand{\cala}{{\mathcal A}}

\newcommand{\calp}{{\mathcal P}}
\newcommand{\calb}{{\mathcal B}}

\newcommand{\calt}{{\mathcal T}}

\newcommand{\caly}{{\mathcal Y}}

\newcommand{\calv}{{\mathcal V}}

\newcommand{\calz}{{\mathcal Z}}
\newcommand{\cZ}{{\mathcal Z}}

\newcommand{\frakX}{{\mathfrak{X}}}
\newcommand{\calx}{{\mathcal{X}}}
\newcommand{\X}{\mathbb{X}}

\newcommand{\al}{{\alpha}}
\newcommand{\la}{{\lambda}}
\newcommand{\gl}{{\lambda}}

\newcommand{\eps}{{\varepsilon}}

\newcommand{\gep}{{\varepsilon}}
\newcommand{\ga}{{\gamma}}
\newcommand{\Ga}{{\Gamma}}
\newcommand{\gG}{{\Gamma}}
\newcommand{\vp}{{\varphi}}

\newcommand{\om}{{\omega}}
\newcommand{\go}{{\omega}}
\newcommand{\Om}{{\Omega}}
\renewcommand{\phi}{\varphi}

\newcommand{\gb}{{\beta}}

\newcommand{\wh}{\widehat}
\newcommand{\wt}{\widetilde}

\newcommand{\dd}{\mathrm{d}}

\newcommand{\cc}{\mathrm{c}}

\newcommand{\bQ}{\mathbf{Q}}

\newcommand{\bt}{\boldsymbol{t}}
\newcommand{\bz}{\boldsymbol{z}}

\newcommand{\bw}{\boldsymbol{w}}
\newcommand{\bx}{\boldsymbol{x}}

\newcommand{\bj}{\boldsymbol{j}}

\newcommand{\btheta}{\boldsymbol{\theta}}

\newcommand{\ceq}{:=}
\newcommand{\ind}{\mathbf{1}}
\newcommand{\lint}{\llbracket}
\newcommand{\rint}{\rrbracket}

\newcommand{\var}{{\mathrm{Var}}}

\newcommand{\suptwo}[2]{\sup_{\substack{#1 \\ #2}}} 
\newcommand{\sumtwo}[2]{\sum_{\substack{#1 \\ #2}}} 

\newcommand{\bthm}{\begin{theorem}}
\newcommand{\ethm}{\end{theorem}}

\newcommand{\bcor}{\begin{cor}}
\newcommand{\ecor}{\end{cor}}

\newcommand{\blem}{\begin{lemma}}
\newcommand{\elem}{\end{lemma}}

\newcommand{\bprop}{\begin{proposition}}
\newcommand{\eprop}{\end{proposition}}

\newcommand{\bdf}{\begin{definition}}
\newcommand{\edf}{\end{definition}}

\newcommand{\bex}{\begin{example}}
\newcommand{\eex}{\end{example}}

\newcommand{\brem}{\begin{rem}}
\newcommand{\erem}{\end{rem}}

\newcommand{\bass}{\begin{assumption}}
\newcommand{\eass}{\end{assumption}}

\newcommand{\bpr}{\begin{proof}}
\newcommand{\epr}{\end{proof}}

\newcommand{\benu}{\begin{enumerate}}
\newcommand{\eenu}{\end{enumerate}}

\newcommand{\bit}{\begin{itemize}}
\newcommand{\eit}{\end{itemize}}

\newcommand{\itt}{\textit}

\allowdisplaybreaks[3]

\title[Stochastic heat equation with Lévy noise]{The stochastic heat equation with multiplicative \\ Lévy noise: Existence, Moments, and Intermittency}

\author{Quentin Berger}
\address{LPSM, Sorbonne Université\\ UMR 8001 Campus Pierre et Marie Curie\\ Boîte courrier 158\\ 4 Place Jussieu\\ 75252 Paris Cedex 05\\ France}
\email{\href{mailto:quentin.berger@sorbonne-universite.fr}{quentin.berger@sorbonne-universite.fr}}

\author{Carsten Chong}
\address{Department of Statistics\\ 
		Columbia University\\
			1255 Amsterdam Avenue\\
			New York, NY 10027\\ USA}
\email{\href{mailto:carsten.chong@columbia.edu}{carsten.chong@columbia.edu}}

\author{Hubert Lacoin}
\address{IMPA\\ Institudo de Matemática Pura e Aplicada\\ Estrada Dona Castorina 110\\ Rio de Janeiro CEP-22460-320\\ Brazil}
\email{\href{mailto:lacoin@impa.br}{lacoin@impa.br}} 


\makeatletter
\@namedef{subjclassname@2020}{%
	2020 Mathematics Subject Classification}
\makeatother

\date{}

\begin{document}

\begin{abstract}
We study the stochastic heat equation (SHE) $\partial_t u = \frac12 \Delta u + \beta  u \xi$ driven by a multiplicative Lévy noise $\xi$ with positive jumps and amplitude $\beta>0$, in arbitrary dimension $d\geq 1$.
We prove the  existence of solutions under an optimal condition if $d=1,2$ and  a close-to-optimal condition if $d\geq3$.
Under an assumption that is general enough to include  stable noises, we further prove that the solution is unique.
By establishing tight moment bounds on the multiple Lévy integrals arising in the chaos decomposition of $u$, we further show that the solution has finite $p$th moments for $p>0$ whenever the noise does. 
Finally,  for any $p>0$,
we derive
 upper and lower bounds on the moment Lyapunov exponents of   order $p$ of the solution, which are asymptotically sharp in the limit as $\beta\to0$.
One of our most striking findings is that the solution to the SHE exhibits 
a property called \emph{strong intermittency} (which implies moment intermittency of all orders $p>1$ and pathwise mass concentration of the solution), for any non-trivial Lévy measure, at any disorder intensity $\gb>0$,
in any dimension $d\geq1$. 
\end{abstract}

\maketitle

\tableofcontents

\section{Introduction}

\noindent We consider the stochastic partial differential equation 
\begin{equation}\label{eq:SHELN}
\partial_t u = \tfrac12\Delta u + \beta  u \xi\, ,\qquad u(0,\cdot)=u_0,
\end{equation}
where $\xi$ is a space-time Lévy noise, $\gb>0$ is an intensity parameter and $u_0$ is some initial condition. In most parts of the paper, we assume that $\xi$ is spectrally positive (\textit{i.e.}, only has positive jumps) without a Gaussian part; extensions to the general case will be discussed in Section~\ref{sec:genelevy} below. 
The equation \eqref{eq:SHELN} is usually referred to  as  the \emph{stochastic heat equation  (SHE)} with multiplicative noise
or
  the \emph{parabolic Anderson model (PAM)}; see  \cite{CB95,CM94} for early works on the subject and \cite{Chen15, Chen2016,HHNT2015, Khoshnevisan14,Khoshnevisan17,Khoshnevisan18} for a selection of more recent contributions in the case where $\xi$ is Gaussian.
We opt for the SHE denomination since  the expression  \textit{parabolic Anderson model} is also often used in the literature to  designate  the equation 
\begin{equation}\label{eq:PAM}
\partial_t u = \tfrac12\Delta u + \beta  u V,
\end{equation}
where the multiplicative noise  term $V$  does not depend on time (see  \cite{GW2018,HL2015} for examples of continuous models and \cite{GK00, GM90, Konig16} for  lattice models). In this case,   Equation \eqref{eq:PAM} is the real-valued analogue of the Schrödinger equation associated with the Anderson Hamiltonian $H=\frac 1 2 \Delta u+\beta V$. A lattice version of this Schrödinger equation appears in one of the original papers concerning Anderson localization \cite[Equation (1)]{Anderson58}; see also \cite[Equation (0.2)]{GM90}.
The study of the localization properties of $H$ (which bears some connection to \eqref{eq:PAM}) has been and remains an important area of research in  physics. We refer the reader to the reviews \cite{CM87,Gun11} and the references therein.

\smallskip

Compared to~\eqref{eq:PAM}, Equation \eqref{eq:SHELN} corresponds to a related but different problem: the diffusion of particles in a 
random medium that varies at small time scales.
This equation has been widely studied in the literature, firstly and mostly in the case where $\xi$ is a space-time Gaussian white noise~\cite{CB95}.
In this case, the equation can be related, via the Cole--Hopf transform, to the KPZ equation, a model for the progression of growth fronts that has attracted a lot of attention in the literature; see \cite{ACQ11, BG97,Hairer13,KPZ86} and the two review papers \cite{Corwin12,Quastel15}. Furthermore, in the case of a Dirac initial condition, the solution to \eqref{eq:SHELN} 
has been used to describe the scaling limit of a one-dimensional directed polymer in a random environment \cite{AKQ14b,AKQ14a}, called the \emph{continuum directed polymer model}.
With Lévy noise,  the existence of solutions to \eqref{eq:SHELN} has only been proved  under  some additional integrability assumptions on the Lévy measure; see \cite{BL20_cont,Chong17, Loubert98} as well as Sections~\ref{sec:mild} and~\ref{sec:soltrunk} below for a more detailed review.  Recently, \cite{BL20_disc} further showed that directed polymers in certain heavy-tailed environments have  scaling limits that, as we shall show below, agree with   solutions to \eqref{eq:SHELN} with $\al$-stable $\xi$,  providing a physical motivation for studying the SHE with Lévy noise.

\smallskip

Another central point of interest in past studies of the SHE is the phenomenon of localization of solutions, or \emph{intermittency} \cite{CM94, GM90,ZMR87}.
One striking manisfestation is the
appearance  of sharp peaks at large times $t$ in the random field
$(u(t,x))_{x\in \bbR^d}$, where $u$ is the  solution to \eqref{eq:SHELN} with initial condition $u_0\equiv 1$. 
These peaks result from  mass concentration   of $u(t,\cdot)$ on a fraction of space that decreases exponentially in $t$.
A common way to quantify this concentration property (see Section~\ref{sec:interm} for more details) is to prove that for some $p<p'$, the ratio of moments $\bbE[ u(t,x)^{p'}]/\bbE[ u(t,x)^{p}] $ grows exponentially in time, and to compute the growth rate.  
For the SHE with Gaussian noise, this type of moment intermittency  was studied in depth by, for example, \cite{CB95,Chen15,Chen2015,Das21,Khoshnevisan14}. 
We also refer to 
\cite{Khoshnevisan17,Khoshnevisan18} for a path-by-path analysis of the intermittency peaks.
For the SHE with Lévy noise, the issue of moment intermittency was investigated only much more recently in \cite{CK19}. In \cite{CK20}, it was further proved that unusually large peaks (on a logarithmic scale) already appear in the solution to the SHE with an \emph{additive} Lévy noise.

\smallskip

Before we describe our main results in Section~\ref{sec:results}, we   provide a detailed technical introduction to Equation \eqref{eq:SHELN}, which aims at being as self-contained as possible. 
 
\subsection{Lévy noise with positive jumps}
Given  a  measure $\lambda$ on $(0,\infty)$ that satisfies     $\gl([1,\infty))<\infty$ and $\int_{(0,1)} z^2\, \lambda(\dd z) <\infty$, we provide a constructive definition of a pure-jump Lévy noise with  intensity~$\gl$. General Lévy noises, with a Gaussian part or with negative jumps, will be considered in Section~\ref{sec:genelevy}. 
Let $\go$ be a Poisson point process on $\R \times \bbR^d \times (0,\infty)$ with intensity $\nu:=\dd t \otimes \dd x \otimes \lambda (\dd  z)$.
The law of $\go$ is denoted by $\bbP$.
For convenience, we sometimes split $\go$ into two processes $\go_{<}$ and $\go_{\geq}$ that correspond to the restriction of $\go$ to  $\bbR \times \bbR^d \times(0,1)$ and $\bbR \times \bbR^d \times[1,\infty)$, respectively. We let $\bbP_{<}$ and $\bbP_\geq$  denote the associated probabilities, so that $\bbP=\bbP_<\otimes \bbP_\geq$.
For $a\in(0,1]$, consider the  measure
\begin{equation}
\label{def:xia}
\xi_{\go}^{a} := \sum_{(t,x,z)\in \go} z \ind_{\{z\geq a\}} \delta_{(t,x)}  - \kappa_a \cL \, ,
\end{equation}
where $\cL$ is the Lebesgue measure on $\bbR \times \bbR^d$
and $\kappa_a  =  \int_{[a,1)} z \,\lambda (\dd z)$.
We also set 
\begin{equation}
\label{def:xia12}
\xi_{\go_<}^{a} := \sum_{(t,x,z)\in \go} z \ind_{\{z\in [a,1)\}} \delta_{(t,x)}  - \kappa_a \cL  \qquad \text{and} \qquad  \xi_{\go_\geq} := \sum_{(t,x,z)\in \go} z \ind_{\{z\ge 1\}} \delta_{(t,x)}  \,.
\end{equation}
Under the assumption $\int_{(0,1)} z^2\,\gl(\dd z) <\infty$,
the measure $\xi_{\go}^{a}$ converges a.s., in
the local Sobolev space $H_{\mathrm{loc}}^{-s}(\bbR \times \bbR^d)$ for any $s>(1+d)/2$, towards a limit $\xi_{\go} \in H_{\mathrm{loc}}^{-s}(\bbR \times \bbR^d)$ (this is a standard result, we refer to \cite[Appendix A]{BL20_cont} for a proof of this exact statement).
For $k\ge 1$ and $u\in \bbR$,  the Sobolev space  $H^{u}(\bbR^k)$  is the Hilbert space of all Schwartz distributions $\varphi\in\cals'(\R^k)$   for which
\begin{equation*}
\lVert\varphi \rVert_{H^u(\bbR^k)}:=  \bigg(\int_{\bbR^k} (1+|\xi|^2)^u |\hat \varphi(\xi)|^2 \,\dd \xi \bigg)^{\frac12} < \infty,
\end{equation*}
where $\hat \varphi$ denotes the Fourier transform of $\varphi$. When $\varphi\in C^{\infty}_c(\bbR^k)$ (\textit{i.e.}, $\varphi$ is smooth and compactly supported), then $\hat \varphi$ is defined by
$ \hat \varphi(\xi):= \int_{\bbR^k} e^{i \xi\cdot x} \varphi(x) \,\dd x.$
The local Sobolev space $H^{u}_{\mathrm{loc}}(\bbR^k)$ is then defined by 
\begin{equation*}
H^{u}_{\mathrm{loc}}(\bbR^k):= \Big\{  \varphi \in\cals'(\R^k)\ : \ \forall \rho\in C^{\infty}_c(\bbR^k), \  \rho\varphi\in
 H^{u}(\bbR^k)  \Big\} \,,
\end{equation*}
equipped with the topology induced by the family of seminorms $\{\varphi \mapsto \| \rho \varphi\|_{H^u(\bbR^k)}:\rho\in C^{\infty}_c(\bbR^k)\}$.
The random distribution $\xi_{\go}$ is called a \emph{L\'evy (space-time white) noise} with intensity or Lévy measure~$\lambda$. In the
the case where $\lambda(\dd x) = \alpha x^{-(1+\alpha)}\, \dd x$ for some $\alpha\in (0,2)$, $\xi_\om$ is referred to as an \emph{$\alpha$-stable noise}. 
We refer the reader to \cite{Aziznejad20} and \cite{Dalang17} for more background on L\'evy noises.
In the remainder, we let $\cF:=(\cF_t)_{t\ge 0}$ be the 
 completed  natural filtration associated with $\xi_{\go}$, that is,
	\begin{equation}\label{filtraF}
	\cF_t := \overline{ \sigma\Big( \go \cap ([0,t]\times \bbR^d\times(0,\infty))\Big)}.
	\end{equation}

	\begin{rem} 
	Note that the Poisson process $\go$ is measurable with respect to $\sigma(\xi_{\go})$. Hence the solution $u$ to \eqref{eq:SHELN} that we construct in this paper (which is defined in terms of  $\go$) is a function of the noise $\xi$.  
	Our choice to index by $\om$ instead of $\xi$ is only for convenience because
many quantities   in our proofs are easier to express in terms of $\om$ than in terms of $\xi$; in particular, this is not because we need an extension of the probability space in our construction. We mention this explicitly because there are martingale constructions which are similar to the one used in this paper and that do require an extension of the probability space: this is, for instance, the case in     Kahane's construction of  Gaussian multiplicative chaos associated with a kernel of $\sigma$-positive type  (see \cite{Kahane85} for the seminal paper or \cite[Section 2.1]{RV14} for a review).
\end{rem}

\subsection{Mild solutions to SHE}\label{sec:mild}

A random field $(u(t,x))_{t>0,\ x\in\R^d}$ is called a \emph{mild solution} to
the SHE~\eqref{eq:SHELN} with initial condition $u_0$ if for all $t>0$ and $x\in\R^d$,
\begin{equation}\label{eq:integz}
	u(t,x)=\int_{\bbR^d} \rho(t,x-y) u_0(\dd y)+ \beta \int^t_0 \int_{\bbR^d}\rho(t-s,x-y)u(s,y) \,\xi(\dd s, \dd y)\qquad\text{a.s.}
\end{equation}
The statement includes the requirement that the integrals be well-defined and finite. In \eqref{eq:integz}, 
$$ \rho(t,x) :=  (2\pi t)^{-\frac d2} e^{- \frac{\|x\|^2}{2t}} $$
is the $d$-dimensional  heat kernel and $\lVert\cdot\rVert$ denotes the Euclidean norm on $\bbR^d$.

\brem\label{rem:Ito}
If $\int_{(0,1)} z \,\la(\dd z)=\infty$, the noise $\xi$ does not have a locally finite total variation. In this case, the stochastic integral on the right-hand side of~\eqref{eq:integz} cannot be defined as a path-by-path Lebesgue integral (\textit{i.e.}, an integral with respect to a random measure defined on $[0,t]\times \bbR^d$) and must be interpreted in It\^o's sense.
\erem

\noindent In \cite{Loubert98}, it was shown that \eqref{eq:SHELN} has a unique mild solution if   the Lévy measure $\la$ satisfies
\beq\label{eq:SLB} 
\exists \, q\in (1,1+\tfrac2d)\ : \ \int_{(0,\infty)} z^q\,\la(\dd z)<\infty \,.
\eeq
 While this condition covers many  examples, it is not fully satisfying, in particular because it rules out  $\al$-stable noises for all $\al\in(0,2)$. Subsequent papers that considered the case of $\alpha$-stable noise (\textit{e.g.}, \cite{Balan14, BL20_cont,Chong17,Chong19}) neither established uniqueness nor the finiteness of moments of the solution to~\eqref{eq:SHELN}.
At this point, let us also mention \cite{Mueller98,Mytnik02}, who investigated an SHE with stable noise but with a non-Lipschitz nonlinearity.

\subsection{The truncation approach to  SHE}\label{sec:soltrunk}

Prior to this article, 
the most general existence condition for the solution to the SHE \eqref{eq:SHELN} was obtained, for a Lévy noise with positive jumps, in~\cite{BL20_cont}, but using a (possibly) weaker notion of solution than \eqref{eq:integz}.
The approach in  \cite{BL20_cont} consists in solving the equation 
 \begin{equation}\label{eq:SHELNa}
\partial_t u^a = \tfrac12\Delta u^a + \beta  u^a \xi^a_\go \, ,\qquad u^a(0,\cdot)=u_0,
\end{equation}
with $\xi^a_{\go}$ from \eqref{def:xia} and then taking the limit of $u^a$ as $a\downarrow 0$.

\begin{rem}\label{rem:dirpol}
In~\cite{BL20_cont}, it is not proved that the obtained limit satisfies \eqref{eq:integz}. In fact, the main  focus of \cite{BL20_cont} is not the SHE, but rather constructing a 
  continuum path measure on~$\bbR^d$: \emph{the continuum directed  polymer in Lévy noise}. This is further discussed in Section~\ref{sec:dirpol} below.
\end{rem}

In order to review the results of \cite{BL20_cont}, let us start by introducing some important notations and  quantities.
For $0<s<t<\infty$, $x,y \in \bbR^d$, 
 let 
 $\frakX_k(s,t) := \{ (t_1, \ldots, t_k) \in \bbR^k:  s<t_1 < \cdots < t_k <t \}$ denote the $k$-dimensional simplex delimited by $s$ and $t$ and define, for $\bt\in\frakX_k(s,t)$ and 
 $\bx\in (\bbR^d)^k$, 
\begin{equation}
\label{def:rho}
\rho_{s,x;t,y}(\bt, \bx) = \prod_{i=1}^{k+1} \rho( \Delta t_i, \Delta x_i) \,,
\end{equation}
where
\begin{equation}\label{deltadeff}
\Delta t_i := t_i-t_{i-1}\qquad \text{and}\qquad \Delta x_i := x_{i}-x_{i-1} \,,
\end{equation}
with the convention that $t_0:=s$, $x_0:=x$ and $t_{k+1}:=t$, $x_{k+1}:=y$.
Given $a>0$ we let $|\xi_{\go}^a|$ and $\xi_{\go}^{a,+}$ denote the total variation and the positive part of  $\xi_{\go}^a$ (considered as a measure),
that is, 
\begin{equation}\label{def:posandtv}
 |\xi_{\go}^a|=\xi_{\go}^a+2\kappa_a\mathcal L \qquad \text{and}   \qquad \xi_{\go}^{a,+}=\xi_{\go}^a+\kappa_a\mathcal L.
\end{equation}
We define the \textit{point-to-point partition function} associated with \eqref{eq:SHELNa}
and truncated L\'evy noise $\xi_{\go}^a$ as
\begin{equation}\label{eq:trucatedp} 
\cZ_{\gb}^{\go, a}(s,x;t,y) :=   \rho(t-s,y-x)+ \sum_{k=1}^{\infty}\gb^k
\int_{ \frakX_k(s,t) \times (\bbR^d)^k}   \rho_{s,x;t,y}( \bt , \bx)   \prod_{i=1}^k  \xi_{\go}^a  (\dd t_i , \dd x_i ),
\end{equation}
under the assumption that the sum of integrals is absolutely convergent, that is,  
\begin{equation}\label{absol}
 \sum_{k=0}^{\infty}\gb^k
\int_{ \frakX_k(s,t) \times (\bbR^d)^k}   \rho_{s,x;t,y}( \bt , \bx)   \prod_{i=1}^k  |\xi_{\go}^a|  (\dd t_i , \dd x_i )<\infty.
\end{equation}
It is proved in~\cite[Prop.~2.5 and 2.6]{BL20_cont} that \eqref{absol} holds if and only if
\begin{equation}\label{eq:log}
	\int_{[1,\infty)} (\log z)^{\frac d2} \,\lambda(\dd z) <\infty.
\end{equation}
The case where $x=0$ and $s=0$ is of particular interest to us, hence we introduce the notational convention
\begin{equation}\label{eq:simplif}
\cZ^{\go,a}_{\beta}(x;t,y):=\cZ^{\go,a}_{\beta}(0,x;t,y), \qquad  
\cZ^{\go,a}_{\beta}(t,x):=\cZ^{\go,a}_{\beta}(0,0;t,x), 
\end{equation}
which applies similarly to other quantities such as $\frakX_k(t):= \frakX_k(0,t)$ and $\rho_{t,x}(\bt,\bx):=\rho_{0,0;t,x}(\bt,\bx)$.
When the condition \eqref{eq:log} holds, $\cZ^{\go,a}_{\beta}(\cdot,\cdot)$ is a mild solution solution to the SHE with noise  $\xi^a_\go$ and initial condition $\delta_0$. Indeed, from \eqref{absol}, the integrals are absolutely convergent, so~\eqref{eq:trucatedp} can be rewritten as
\begin{equation}\label{eq:pppf-2}
	\cZ_{\gb}^{\go,a}(t,x) = \rho(t,x)  + \gb \int_{0}^t \int_{\bbR^d} \rho(t-s,x-y) \cZ_{\gb}^{\go,a}(s,y)\, \xi_{\go}^a(\dd s, \dd y) \,.
\end{equation}
In the same manner, $\cZ_{\gb}^{\go,a}(y;\cdot,\cdot)$ is a mild solution to the SHE with initial condition $\delta_y$  and noise~$\xi^a_\om$, and  given some uniformly bounded $u_0: \bbR^d \to [0,\infty)$, the random field $u^a$ defined by 
\begin{equation}
	\label{eq:integz-2}
	u^{a} (t,x):= \int_{\bbR^d} u_0(y)\cZ_{\gb}^{\go,a}(y;t,x)\, \dd y
\end{equation}
is a mild
solution to the SHE with initial condition $u_0$ and noise $\xi_\om^a$; see \cite[Prop.~2.19]{BL20_cont}.

\begin{rem}
\label{rem:invariance}
Note that by translation invariance, we have
\begin{equation}\label{eq:theinlaws}
\begin{split}
\Big(\cZ^{\go,a}_{\beta}(s,x;t,y)\Big)_{y\in \bbR^d} 
& \stackrel{(d)}{=}\Big(\cZ^{\go,a}_{\beta}(t-s,y-x)\Big)_{y\in \bbR^d} \stackrel{(d)}{=} \bigg(e^{-\frac{\| x-y\|^2+\|y\|^2}{2(t-s)}}\cZ^{\go,a}_{\beta}(t-s,y)\bigg)_{y\in \bbR^d}.
\end{split}
\end{equation}
For this reason, we present most results only for the case $s=0$ and $x=0$, without loss of generality.
\end{rem}

\noindent

The next step is to investigate existence and ``relevance'' of the limit $\cZ^{\go,a}_{\beta}(t,x)$ as $a\downarrow 0$.
If  \eqref{eq:log} holds,
integrating the ``Lebesgue part'' of $\xi_{\go}^a$ (see  \cite[Prop.\ 2.15]{BL20_cont} for details),   we obtain the  alternative expression 
\begin{equation}
\label{eq:help} \cZ_{\gb}^{\go, a}(t,x) \\
= e^{-\beta\kappa_a t } \left(\rho(t,x)+ \sum_{k=1}^{\infty}\gb^k
\int_{ \frakX_k(t) \times (\bbR^d)^k}   \rho_{t,x}( \bt , \bx)   \prod_{i=1}^k  \xi_{\go}^{a ,+}  (\dd t_i , \dd x_i )\right),
\end{equation}
which entails the positivity of $\cZ_{\gb}^{\go, a}(t,x)$.
 Still under \eqref{eq:log},    considering the   reverse filtration $\cG:=(\cG_a)_{a\in(0,1)}$ defined by
\begin{equation}\label{filtraG}
	\cG_a:= \sigma(\xi^{a}_\go)= \sigma( \go \cap (\bbR\times \bbR^d\times [a,\infty))),
\end{equation}
one can observe  that for almost every realization of $\go_\geq$,
$(\cZ_{\beta}^{\go, a}(t,x))_{a\in(0,1]}$ is a non-negative càdlàg time-reversed $\P_<$-martingale with respect to $\cG$ (cf.\ \cite[Lemma~3.5]{BL20_cont})
and thus admits a limit as $a\downarrow0$.
To determine whether $\lim_{a\to 0}\cZ_{\beta}^{\go, a}(t,x)$  is a good candidate for being a solution to \eqref{eq:integz}, a first step  is to determine whether this limit is degenerate or not. 
The answer depends on the intensity measure $\gl$ and  the following result summarizes the main findings of \cite[Thm.~2.7 \& Prop.~2.10--2.15]{BL20_cont}.

\begin{theorema}\label{th:previousex}
	Assume that $\gl$ satisfies the condition in \eqref{eq:log}.
	\begin{itemize}
		\item [(i)] If, in addition,
		\begin{equation}
			\label{assump1}
			\begin{cases}
				\displaystyle\int_{(0,1)} z^2\, \lambda(\dd z) <\infty \  \quad \ &\text{if } d=1,\\
				\exists\, p\in \big(1,1+\frac{2}{d}\big) \ : \ \displaystyle\int_{(0,1)} z^p \,\lambda(\dd z) <\infty &\text{if } d\geq 2,
			\end{cases}
		\end{equation}
		then for every $0<s<t<\infty$ and $x,y\in\R^d$, we have
		\begin{equation}\label{eq:conv}
			\cZ_{\gb}^{\go}(s,x;t,y):= \lim_{a\to 0} \cZ_{\gb}^{\go, a}(s,x;t,y) >0  \qquad \text{$\bbP$-a.s.}
		\end{equation}
		Moreover, the convergence holds in $L^1(\bbP_{<})$ for $\P_\geq$-a.e.\ realization of $\go_{\geq}$, that is, 
		\begin{equation*}
			\lim_{a\to 0}  \bbE_{<} \Big[  \Big| \cZ_{\beta}^{\go}(s,x;t,y)-\cZ_{\beta}^{\go, a}(s,x;t,y) \Big| \Big]=0    \qquad \text{$\bbP_{\geq}$-a.s.}
		\end{equation*}
		\item [(ii)] If, on the other hand, we have
		\begin{equation}
			\label{assump2}
			\begin{cases}
				\displaystyle\int_{(0,1)} z^2\, \lambda(\dd z) =\infty \  \quad \ &\text{if } d=1,\\
				\displaystyle\int_{(0,1)} z^2 \lvert \log z\rvert \,\lambda(\dd z) =\infty \ \quad \  &\text{if } d=2,\\
				\displaystyle   \int_{(0,1)} z^{1+\frac 2d} \,\lambda(\dd z) = \infty &\text{if } d\geq 3,
			\end{cases}
		\end{equation}
		then $\bbP$-a.s., we have $\lim_{a\to 0} \cZ_{\beta}^{\go, a}(s,x;t,y) =0$.
	\end{itemize}
\end{theorema}

  One can also consider a \textit{free-end} (or \textit{point-to-line}) version of the partition function
with truncated L\'evy noise, given by 
\begin{equation}\label{eq:fepf}
\cZ_{\gb}^{\go, a} (s,x;t,\ast) :=\int_{\bbR^d}\cZ^{\go,a}_{\beta} (s,x;t,y)\,\dd y.
\end{equation}
Theorem~\ref{th:previousex} applies to $\cZ_{\gb}^{\go, a} (s,x;t,\ast)$ \itt{mutatis mutandis}. Following the convention \eqref{eq:simplif}, we use the notations $\cZ^{\go}_{\gb}(s,x;t,\ast)$, $\cZ^{\go}_{\gb}(x;t,\ast)$ and $\cZ^{\go}_{\gb}(t,\ast)$ for the limits we obtain as~$a$ tends to zero (as a consequence of the reverse martingale property).
 Furthermore, under the assumptions of Theorem~\ref{thm:local},   \cite[Prop.~2.20]{BL20_cont} showed 
that  
$\lim_{a\to 0} u^{a} (t,x)$ exists a.s.\ for $u^a$ defined in \eqref{eq:integz-2}, without proving that the limit is a solution to \eqref{eq:integz}.
Since moments of the measure $\gl$ play an important role in our assumptions, we introduce a notation for partial moments of the measure $\gl$ by setting  
\begin{equation}\label{defmuab}
	\mu_{a,b}(p) :=\int_{[a,b)} z^p \,\lambda(\dd z) \,,
\end{equation}
for $0\leq a\leq b\leq\infty$.
 We   simply write  $\mu:=\mu_{1,\infty}(1)$ for the first moment restricted to~$[1,\infty)$. 

\begin{rem}\label{rem:1}
In the case of a finite mean $\mu$, given $t>0$ and $x\in \mathbb R$,
  the convergence of both $\cZ_{\beta}^{\go, a}(t,x)$ and
$\cZ_{\beta}^{\go, a}(t,\ast)$ as $a\to 0$  also holds in $L^1(\bbP)$.
\end{rem}

If $\mu <\infty$, it is convenient to consider the normalized partition functions, defined as 
\begin{equation}\label{barparti}
	\bar \calz^{\om,a}_{\beta}(t,x):= e^{-\beta \mu t}  \calz^{\om,a}_{\beta}(t,x),\qquad \bar \calz^{\om,a}_{\beta}(t,\ast):= e^{-\beta \mu t}  \calz^{\om,a}_{\beta}(t,\ast).
\end{equation}
In the same manner, we can center the noise by setting 
\begin{equation}\label{barnoise}
	\bar \xi^a_{\go}:= \xi^a_{\go}-\mu \cL,
\end{equation}
which has the effect that  integrals with respect to  $\bar \xi^a_{\go}$ have mean zero. Note that we have
\begin{equation}\label{eq:chaos}
	\bar \calz^{\om,a}_{\beta}(t,x)= \rho(t,x)+ \sum_{k=1}^{\infty}  \gb^k \int_{ \frakX_k(t) \times (\bbR^d)^k}  \rho_{t,x}( \bt , \bx)   \prod_{i=1}^k \bar \xi_{\go}^a (\dd t_i , \dd x_i )
\end{equation}
and a similar identity for $\bar \calz^{\om,a}_{\beta}(t,\ast)$.

\subsection{Intermittency and related notions}
\label{sec:interm} Consider a random field $(v(t,x))_{t>0, x\in\R^d}$ where $d\geq0$; when $d=0$ this means that $(v(t))_{t>0}$ is a stochastic process only indexed by $t$. Assuming that the \emph{moment Lyapunov exponents}
 \begin{equation}\label{eq:Lya}
  \ga(p)= \gamma(v,p):=\lim_{t\to\infty}  \frac1t \log \E\Big[ \lvert v(t,x)\rvert^p \Big]  
\end{equation}
exist on the extended real line for all $p\in(0,\infty)$ and are independent of $x\in\R^d$, we let
\beq\label{defig2} I = I(v):=\{p\geq0 \ : \ \ga(v,p)<\infty\}, \eeq
with the convention $\ga(0)=0$.
Further assuming $\lvert \ga(1)\rvert <\infty$ in the following, we  define
the \emph{normalized moment Lyapunov exponents} of $v$ by
\begin{equation}\label{eq:gabar}
 \bar \gamma(p):= \gamma(p)-p\,\gamma(1). 
\end{equation}
Clearly, $\bar \ga(1)=0$.
Following the terminology of \cite{CM94,GM90,ZMR87}, if $p>1$ and $\lvert \ga(1)\rvert <\infty$,  we say that~$v$ exhibits
 \emph{(moment) intermittency of order $p$} if 
$\bar \gamma(p)\in (0,\infty)$.
We say that  $v$ exhibits
 \emph{full intermittency} if $v$ is intermittent of all orders   $p\in (1,\infty)\cap I$. 
 
Let us also introduce some new terminology when we consider moments of order $p\in (0,1)$:
we say that $v$ exhibits 
 \emph{strong intermittency}   (in analogy with ``very strong disorder'' used in the directed polymer context, see the discussion below) if $\bar \gamma(p)\in(-\infty,0)$ for some $p\in (0,1)$. As $p\mapsto \bar\ga(p)$ is convex on $I$ (cf.\ Proposition~\ref{prop:qual} \textit{(i)} below),   strong intermittency implies  $\bar \gamma(p)<0$ for all $p\in (0,1)$ but also that   $\bar \gamma(p)> 0$ for all $p>1$. Thus,  strong intermittency implies full intermittency. As   Section~\ref{sec:geoloca} below reveals,   strong intermittency, plus some ergodic properties in $x$, yields a    \emph{geometric characterization} of intermittency: we have a mass concentration  of the paths of $v$ at large times, characterized by the appearance of exponentially large peaks on islands  covering only an exponentially small fraction of space (a more quantitative study of this phenomenon was undertaken in \cite{Khoshnevisan17,Khoshnevisan18},  and
 we also refer to \cite{GKM07} for a similar work concerning the parabolic Anderson model on $\bbZ^d$).


\subsubsection*{Intermittency for the SHE}
Let us first discuss the known intermittency result in a semi-discrete setting (this is the setup where the most is known). 
For the SHE on $\mathbb Z^d$ with either a Gaussian or a finite-variance Lévy noise, the results of \cite{Ahn92b,Ahn92,CM94} and \cite{Ber15} show that
\bit
\item
if $d=1,2$,  then strong intermittency holds for every $\beta>0$; 
\item
if $d\geq3$, then strong intermittency holds  if and only if $\beta> \bar \beta_c$ for some $\bar \beta_c>0$.
\eit

For the SHE on $\R$ driven by a multiplicative Gaussian white noise $\xi$, the analysis of intermittency has a long history: if $\xi$ has variance $\beta^2$ and $u_0\equiv 1$,   the authors of \cite{CB95} derived the formula
\beq\label{eq:fml} \ga(p)=\ga(u,p)=\frac{p(p^2-1)}{24}\beta^4 \eeq
for $p\in\N$ and  proved it for $p=2$. Later, a proof  of \eqref{eq:fml} was given in \cite{Chen15} and \cite{Le16} for all integers $p\in\N$ and all real numbers $p\geq2$, respectively. The formula \eqref{eq:fml} for all $p>0$ (and in particular, strong intermittency) was only established recently in \cite{Das21,Ghosal20} using integrable probability methods. For $d\geq2$, there is no notion of solution to the SHE with a Gaussian space-time white noise (cf.\ Section~\ref{sec:genelevy}). For the SHE on $\R^d$ driven by a Gaussian noise that is white in time and colored in space with, say, a compactly supported correlation function, one has a similar picture to the discrete-space setting: strong intermittency always holds if $d=1,2$ and only for large $\beta$ if $d\geq3$; see \cite{Chen19, Lac11}. The situation may be different if the noise has long-range spatial correlation, as shown in \cite{Chen19,Foondun17, Lac11}.

For the solution to the Lévy-driven SHE,
under the assumption \eqref{eq:SLB},  \cite{CK19} established intermittency of order $p$  in the following cases: for all $p\in(1,3)$ if $d=1$; for $p$ close enough to (but smaller than) $1+\frac2d$ if $d\geq2$; for any fixed $p\in(1,1+\frac2d)$ in any dimension $d\geq2$ if $\beta>\beta_{p,d}$ for some $\beta_{p,d}>0$. Whether strong intermittency (or full intermittency in dimensions $d\geq2$) holds in the Lévy case or not, has been an open problem so far.



\subsubsection*{Very strong disorder for the directed polymer model}

To complete the former discussion, let us mention some results that have been proved for the directed polymer in a random environment (DPRE), whose partition function formally corresponds to the solution to an SHE in discrete space and time.
In the context of the directed polymer model, the notion equivalent to strong intermittency is that of \textit{very strong disorder} (see for instance \cite{AKQ14a,CH06,Lac10pol}). For cultural reasons (DPRE is a statistical mechanics model), very strong disorder is a property of the
\emph{free energy}, that is, of the asymptotic behavior of 
$\bbE[\log Z_{\beta}(N,\ast)]$, rather than  a property of the moments of order $p\in(0,1)$ of the partition function;
but this is not relevant for the present discussion.
Very strong disorder   has been proved to hold for directed polymers in a various settings:
\begin{itemize}
\item When $d=1$ \cite{CV06} and $d=2$ \cite{Lac10pol} for any $\beta>0$;
\item When $d\ge 3$, if the environment has a power-law distribution with  exponent $\alpha\in (1, 1+\frac{2}{d}]$, for any $\beta>0$, in \cite{Vi19} (this roughly corresponds  to $\alpha$-stable noise in the SHE context).
\end{itemize}
The method used in \cite{Lac10pol} does not rely much on the discrete nature on the model and has been adapted to  prove analogous results when either space or time are  continuous:  Let us mention the case of the SHE on $\mathbb Z^d$ with Gaussian white noise \cite{Ber15}, 
 the SHE on $\bbR^d$ with a Gaussian noise which is white in time but colored in space \cite[Theorem 1.2]{Lac11}, or directed polymers in a Poisson environment \cite[Remark 3.4.3]{CCrev} (the partition function  of which corresponds to that of the SHE with a spatially convoluted Lévy noise). 
 In fact, the techniques can also be adapted to the SHE with Lévy noise, and the content of Section \ref{sec:onebody} in the present paper is strongly inspired by the proofs in \cite{Lac10pol} (for $d=1$) and \cite{Vi19} (for the case of heavy-tailed noise).
 
 \smallskip
 
 On the other hand there are other situations where
the  directed polymer model \textit{does not} display very strong disorder:
 \begin{itemize}
  \item  When $d\ge 3$, if the environment has bounded second moment and $\beta$ is small \cite{Bol89,IS88};
  \item  When $d\ge 3$, if the environment  has a power-law distribution with  exponent $\alpha >1+\frac{2}{d}$ and $\beta$ is sufficiently small  \cite{Vi19}.
 \end{itemize}
 In both cases, it has been shown that the directed polymer displays \textit{weak disorder}, a property which implies but goes beyond the absence of very strong disorder. Weak disorder implies, for instance, that the sample paths drawn from the polymer measure have diffusive behavior \cite{CY06}.

 \smallskip



For this reason, the question of whether \textit{strong intermittency} holds for the SHE with Lévy space-time white noise is particularly intriguing when $d\ge 3$ and $\beta$ is small.

\section{Main results}\label{sec:results}

Our main results can be summarized as follows:

\smallskip
$\bullet$ Firstly, we considerably reduce the gap 
 conditions~\eqref{assump1} and \eqref{assump2} leave
regarding the non-degeneracy of $\cZ^{\go}_{\beta} (t,x)$ if $d\geq2$. We obtain, in Theorem \ref{thm:local}  and Proposition \ref{prop:waitingforbetter}, a necessary and sufficient condition in dimension $d=2$ and a close-to-optimal condition in dimensions $d\geq 3$.
We further prove the $L^p$-convergence of $\cZ^{\go,a}_{\beta} (t,x)$ to $\cZ^{\go}_{\beta} (t,x)$ for all $p\in(1,1+\frac 2d)$  as soon as the noise~$\xi_\om$ has a finite $p$th moment (\textit{i.e.}, $\mu_{1,\infty}(p)<\infty$).

\smallskip
$\bullet$ Secondly, in Theorem \ref{thm:SHE}, we show that the limit $\cZ^{\go}_{\beta} (t,x)$, when non-degenerate, is indeed a mild solution to the SHE with  initial condition $\delta_0$. Theorem \ref{thm:SHEunique} further establishes  the uniqueness of solutions  under  conditions that are  general enough to  include the case of $\al$-stable noise.

\smallskip
$\bullet$ Thirdly,
after establishing the existence of the moment Lyapunov exponents of $\calz^{\om}_\beta(t,\ast)$ and $\calz^{\om}_\beta(t,x)$ in Proposition \ref{prop:finitemom}, we show our most striking result in Theorem \ref{thm:int}: the solution to the SHE with a multiplicative Lévy noise and $\delta_0$-initial condition exhibits  strong intermittency---and thus, in particular,  full intermittency---for any non-trivial environment with finite expectation, that is, for any $\gb>0$, in any dimension $d\geq1$, for any non-trivial intensity measure $\la$ with $\mu<\infty$. 
In Theorems \ref{thm:thinup} and \ref{thm:heavyup}, we further complement this result by deriving sharp asymptotic estimates on the moment Lyapunov exponents as $\beta$ tends to $0$.

\medskip
The proof of the aforementioned results relies on two main methodological achievements: 

\smallskip
$\bullet$ By combining decoupling inequalities with an iterative partition of  $\frakX_k(t)\times(0,\infty)^k$ (the integration domain of  $(\bt,\bz)$), we establish, in Propositions~ \ref{prop:localbisdone} and~\ref{prop:localbis}, sharp moment bounds of order $p>1$ on the multiple Lévy integrals that arise in the series representation \eqref{eq:chaos} of $\bar\calz^{\om, a}_{\beta}(t,x)$. From these, we will then derive the upper bounds for $p>1$ and, by convexity, the lower bounds for $p<1$ in Theorems~\ref{thm:thinup} and~\ref{thm:heavyup}. The proof of Theorems~\ref{thm:local},~\ref{thm:SHE} and~\ref{thm:SHEunique} will also rely on (variants of) these moment bounds.

\smallskip
$\bullet$ By combining a change-of-measure technique with a coarse-graining approach,  summarized  in Lemma~\ref{lem:changemeas}, we obtain  moment upper bounds of order $p\in(0,1)$ for  $\cZ^{\go}_{\beta}(t,\ast)$
which are then  used to show Theorem~\ref{thm:int} as well as the lower bounds in Theorems~\ref{thm:thinup} and~\ref{thm:heavyup}. Moreover, we apply a simpler version of this method to prove Proposition \ref{prop:waitingforbetter}.
 While the coarse-graining and change-of-measure approach  is derived from that used in \cite{Lac10pol}, which itself was inspired from earlier work on disordered pinning \cite{DGLT09, GLT10hier,  Ton09}, its implementation to prove Proposition \ref{prop:dgeq2} (which is the most important part of the proof of Theorem \ref{thm:int}) and Proposition \ref{prop:waitingforbetter}  relies on important novel ideas, which we discuss in more details at the beginning  of Section \ref{sec:theotherbound}).

\subsection{Optimal conditions for the non-degeneracy of $\cZ^{\go}_{\beta}(t,x)$}

When $d\geq 2$, Theorem~\ref{th:previousex} displays a gap between the sufficient condition~\eqref{assump1} and the necessary condition~\eqref{assump2} for the non-degeneracy of $\cZ^{\go}_{\beta}(s,x;t,y)$.
Our first result reduces this gap when $d\ge 3$ and identifies the necessary and sufficient condition in dimension $d=2$. 
Our new sufficient condition  reads as follows:
 \begin{equation}
\label{eq:log-2}
\begin{cases}
\displaystyle\int_{(0,1)} z^2\, \lambda(\dd z) <\infty \  \quad \ &\text{if } d=1 \,,\\
 \displaystyle\int_{(0,1)} z^{1+\frac 2d}\lvert \log z\rvert \,\lambda(\dd z) <\infty &\text{if } d\geq 2\,.
\end{cases}
\end{equation}
For later reference, we also included the case $d=1$, which is identical to~\eqref{assump1}.
\begin{theorem}[Non-degeneracy]
 \label{thm:local}  Let $d\geq1$ and assume  that \eqref{eq:log} and \eqref{eq:log-2} hold.
Then, for any $t$ and $x$,  we have 
\beq\label{eq:conv-2} \cZ_{\beta}^{\go}(t,x):= \lim_{a\to 0} \cZ_{\beta}^{\go, a}(t,x) >0\qquad\text{$\bbP$-a.s.} \eeq
In addition, for any $p\in  [1, 1+\frac{2}{d} )$, we have
\begin{equation}\label{eq:condconv}
 \lim_{a\to 0} \bbE_{<}\Big[\Big|\cZ_{\beta}^{\go,a}(t,x)-\cZ_{\beta}^{\go}(t,x) \Big|^p \Big]=0 
 \qquad \text{$\bbP_{\geq}$-a.s.}
\end{equation}
If furthermore $\mu_{1,\infty}(p)<\infty$, then
\begin{equation}\label{eq:lpconv}
 \lim_{a\to 0} \bbE\Big[\Big|\cZ_{\beta}^{\go,a}(t,x)-\cZ_{\beta}^{\go}(t,x) \Big|^p \Big]=0.
\end{equation}
The same convergence results hold for the free-end partition function $\calz^\om_\beta(t,\ast)$.
\end{theorem}

\noindent Together with  \eqref{assump2}, this shows that \eqref{eq:log-2} is necessary and sufficient for the non-degeneracy of $\cZ_{\beta}^{\go}(t,x)$ and $\calz^\om_\beta(t,\ast)$ in dimension $d=2$. If $d\geq3$, the following result improves upon the necessary condition in \eqref{assump2}.
\begin{proposition}[Degeneracy]
\label{prop:waitingforbetter} Suppose that \eqref{eq:log} holds and that $d\geq 3$. If for some $\eps>0$, we have that 
\begin{equation}\label{eq:log1-eps} 
\int_{(0, e^{-e})}  \,\frac{ z^{1+\frac 2d} \lvert \log z\rvert}{ (\log\lvert\log z\rvert )^{5+\frac 4d +\eps}}\,\la(\dd z)=\infty \,,
\end{equation}
 then for all $t>0$ and $x\in\R^d$, we have  
 $\lim_{a\to 0} \cZ_{\beta}^{\go,a}(t,x)=\lim_{a\to 0} \cZ_{\beta}^{\go,a}(t,\ast)=0$ $\bbP$-a.s.
\end{proposition}
\noindent
While there is still a small gap between \eqref{eq:log-2} and \eqref{eq:log1-eps} if $d\geq3$, we believe that \eqref{eq:log-2} is necessary and sufficient for non-degeneracy. 


\begin{rem}
There is a small abuse of language when we say  that we establish an almost necessary and sufficient condition for existence of solutions to the SHE, since Proposition \ref{prop:waitingforbetter} does not \textit{a~priori} exclude the existence of solutions that are not given by limits of solutions with truncated noise. 
\end{rem}

\begin{rem}
 Besides replacing  \eqref{assump1} by the weaker condition \eqref{assump2} and the inclusion of $L^p$-convergence, Theorem \ref{thm:local} contains a third important improvement. Contrary to Theorem \ref{th:previousex}, the proof of~\eqref{eq:condconv},
 which is sufficient to establish non-triviality of the limit, does not rely on the positivity of $\cZ_{\beta}^{\go,a}(t,x)$ and  can thus be directly adapted to the case of signed noise, an observation that we will further elaborate on in Section \ref{sec:genelevy}. On the other-hand, our proof of Proposition \ref{prop:waitingforbetter}  strongly relies on having only positive jumps.
\end{rem}

\subsection{Existence and uniqueness  of solutions to SHE}

Our next result shows that the point-to-point partition function $\calz^\om_\beta(t,x)$ solves the SHE with Lévy noise $\xi_{\go}$ and initial condition $\delta_0$.

\begin{theorem}[Existence]
\label{thm:SHE} Suppose that
 \eqref{eq:log} and \eqref{eq:log-2} hold.
 \benu
 \item[(i)] 
The point-to-point partition function $\cZ_{\gb}^{\go}(t,x)$
is a mild solution to the SHE \eqref{eq:integz} with initial condition $\delta_0$,
that is, for every $(t,x)\in(0,\infty)\times\R^d$, the stochastic integral below is well defined and  the following identity holds a.s.:
\begin{equation}
\label{eq:SHE}
\cZ_{\gb}^{\go}(t,x) = \rho(t,x)  + \gb\int_{0}^t \int_{\bbR^d} \rho(t-s,x-y) \cZ^\om_{\gb}(y,s) \,\xi_{\go}(\dd s, \dd y) \,.
\end{equation}
\item[(ii)] If $u_0$ is a locally finite signed measure on $\R^d$ such that 
\beq\label{eq:u0growth} \limsup_{r\to\infty} r^{-2}\log\Big(\lvert u_0\rvert([-r,r]^d)\Big)<\frac{1}{2T}\eeq
for some $T>0$, then
\begin{equation}\label{eq:v}
 v(t,x):= \int_{\bbR^d}\cZ_{\gb}^{\go}(y;t,x)\,u_0(\dd y)
\end{equation}
is well defined and finite for all $(t,x)\in[0,T]\times\R^d$ and 
is a solution to the SHE~\eqref{eq:integz} with initial condition $u_0$, 
on $[0,T]\times\R^d$.
\eenu
\end{theorem}

\begin{rem}\label{rem:lexpress}
The proof presented in this paper actually implies that
	\begin{equation}\label{iteratz}
	\calz^\om_\beta(t,x)=	\rho(t,x)+ \sum_{k=1}^{\infty}\gb^k
		\int_{ \frakX_k(t) \times (\bbR^d)^k}   \rho_{t,x}( \bt , \bx)   \prod_{i=1}^k  \xi_{\go} (\dd t_i , \dd x_i ),
	\end{equation}
where the right-hand side is well defined as a convergent sum of iterated stochastic integrals (see Section \ref{sec:sol} for the corresponding framework).
\end{rem}

\noindent
 Under the stronger condition \eqref{assump1}, we prove that the solution found above is unique.

\begin{theorem}[Uniqueness]
\label{thm:SHEunique}
 We assume that $\gl$ satisfies \eqref{eq:log} and \eqref{eq:log-2}  and that $u_0$ and $T$ satisfy~\eqref{eq:u0growth}.
We let $(v(t,x))_{t\in [0,T],\ x \in  \bbR^d}$ be the solution to the SHE  defined in \eqref{eq:v}. Then there exists  $\theta>1$ such that
for all $p\in (1,1+\tfrac{2}{d})$, $t\in(0,T]$  and $x\in \bbR^d$, we have
\beq\label{eq:cond}  \begin{cases}
\E_{<}\Big[\lvert v(t,x)\rvert^p\Big]<\infty \qquad &\text{$\P_\geq$-a.s.,}\\
\displaystyle \int_{(0,t)\times\R^d} \rho(\theta(t-s),x)^p \E_{<}\Big[\lvert v(s,x)\rvert^p\Big] \,\dd s\,\dd x <\infty\qquad&\text{$\P_\geq$-a.s.}
\end{cases} \eeq
Furthermore, if $\gl$ satisfies \eqref{assump1}, $v$ is the unique---up to modifications---predictable random field that  is a mild solution to the SHE \eqref{eq:integz} and  satisfies  \eqref{eq:cond}.
\end{theorem}

\subsection{Lyapunov exponents and intermittency}
\label{sec:Lyap}

Having established the finiteness of moments in Theorem~\ref{thm:local}, we now investigate the growth rate of these moments as $t\to\infty$
 and their dependence on $\beta$.
Our first task is to establish the existence of the
\textit{(moment) Lyapunov exponents}.

\begin{proposition}
\label{prop:finitemom}
Assume that \eqref{eq:log} and \eqref{eq:log-2}  hold.
\benu
\item[(i)]
If $p\in (0,1+\tfrac{2}{d} )$ and $\mu_{1,\infty}(p)<\infty$, then the moment Lyapunov exponent
 \begin{equation}\label{eq:ga-p}
  \gamma_{\beta}(p):=\ga(\cZ_{\gb}^{\go} (t,\ast),p)=\lim_{t\to\infty}  \frac1t \log \E\Big[  \cZ_{\gb}^{\go} (t,\ast)^p \Big]  
\end{equation}
exists and is finite. 
\item[(ii)] The exponents   $\gamma_{\beta}(p)$ also capture the growth of the point-to-point partition function, in the sense that 
\begin{equation}\label{eq:free-end}
 \lim_{t\to\infty}  \frac1t \log \E\Big[  \cZ_{\gb}^{\go} (t,0)^p \Big]=   \gamma_{\beta}(p) \,.
\end{equation}
 \eenu
\end{proposition}

By \cite[Theorem~3.1]{CK19}, we have $\E [  \cZ_{\gb}^{\go} (t,\ast)^{1+2/d}  ]  =\E[ \cZ_{\beta}^{\go} (t,0)^{1+2/d} ]=\infty$ for any $t>0$ and non-trivial $\la$. We also have $\E [  \cZ_{\gb}^{\go} (t,\ast)^{p}  ] =\infty$ if  $\mu_{1,\infty}(p)=\infty$.
Hence, 
$\gamma_{\beta}(p)$ is well defined (possibly infinite) for every $p\ge 0$. 
Letting
\begin{equation}\label{def:pmax}
  p_{\max}=p_{\max}(\gl):=\sup\Big\{ p\in [0,1+\tfrac{2}{d}   ) \ : \ \mu_{1,\infty}(p)<\infty\Big\} \,,
\end{equation} 
we obtain from the above observation
 that the set
\begin{equation}\label{defig}
I_{\gl} := I(  \cZ_{\gb}^{\go} (t,\ast))= \{ p\ge 0 \ : \ \gamma_{\beta}(p)<\infty \}
\end{equation}
 is either $ [0,p_{\max})$ or $[0,p_{\max}]$, 
the latter occurring if $p_{\max}<1+\tfrac 2d$ and $\mu_{1,\infty}(p_{\max})<\infty$.

The next result lists a few (partly classical) qualitative properties of $\ga_\beta(p)$. 
If $\mu<\infty$ (in which case $1\in I_{\gl}$), we write 
\[
\bar \gamma_{\beta}(p):=\bar \ga( \cZ_{\gb}^{\go} (t,\ast),p)=\gamma_{\gb}(p)-p\gamma_{\beta}(1)  = \gamma_{\beta}(p)- p \mu = \ga(\bar\cZ_{\gb}^{\go} (t,\ast),p ).
\] 
Note that 
by Jensen's inequality, we have $\bar \gamma_{\beta}(p)\ge 0$ for $p>1$ and $\bar \gamma_{\beta}(p)\le 0$ for $p<1$.
\begin{proposition}\label{prop:qual} Under the assumptions of Proposition~\ref{prop:finitemom}, we have the following:
	\begin{enumerate}
	\item[(i)] The map $p\mapsto \ga_\beta(p)$ is convex on $I_{\gl}$.
	\item[(ii)] If $\mu<\infty$ and  $p\in I_{\gl}$, the map  $\gb \mapsto \bar \gamma_{\beta}(p)$
	is non-decreasing if $p\geq 1$ and non-increasing if $p\in (0,1]$.
	\item[(iii)] If $p_{\max}>1$, then  $p\mapsto \ga_{\beta}(p)/p$ is strictly increasing on $\{p\in I_{\gl} : \gamma_{\gb}(p)-p\gamma_{\beta}(1)>0\}$. 
\end{enumerate}
\end{proposition}


\smallskip
Our next goal is to understand under what conditions we have strong intermittency. Based on the  results reviewed in Section~\ref{sec:interm}, one may conjecture that the solution to the SHE on $\bbR^d$ driven by a Lévy noise (that is white in time \emph{and} space) exhibits a similar behavior: strong intermittency in dimensions $d=1,2$ and strong intermittency only for large $\beta$ if $d\geq3$.
This turns out to be a fallacy:  the SHE with a non-trivial Lévy noise \emph{always} exhibits strong intermittency, irrespective of $\beta$, $\lambda$, $d$ or  $p$.


\begin{theorem}[Strong intermittency]
\label{thm:int}
Let $d\geq1$. If $\la\not\equiv0$, $\mu<\infty$  and $\beta>0$, we have $\bar \gamma_{\beta}(p)< 0$ for $p\in (0,1)$ and $\bar \gamma_{\beta}(p)> 0$ for $p\in  I_{\gl}\cap (1,\infty)$. 
\end{theorem}

To shed more light on this peculiar result, we further investigate the  detailed behavior
of $\bar \gamma_{\beta}(p)$ in the small $\beta$ limit. Let us start with the case of light-tailed noise at infinity.

\begin{theorem}[Light-tailed noise]
\label{thm:thinup}
Assume that \eqref{eq:log-2} holds.
 \begin{itemize}
  \item [(i)] If $d=1$ and $\mu_{1,\infty}(2)<\infty$, then for every $p\in I_{\gl}\setminus \{1\}$, there exists $C_p\in(0,\infty)$ such that for every $\beta \in (0,1]$, we have
  \beq\label{eq:d1}(C_p)^{-1}\beta^4 \le  \lvert\bar \gamma_{\beta}(p)\rvert \le C_p \beta^4.\eeq
  In the special case $p=2$, we have $\bar \gamma_{\beta}(2) = \frac14 \mu_{0,\infty}(2)^2 \gb^4$. 
  \item [(ii)] If $d\ge 2$ and $\mu_{1,\infty}(1+\frac{2}{d})<\infty$, then for every $p\in \big(0,1+\frac{2}{d}\big)\setminus \{1\}$, we have 
  \begin{equation}\label{eq:1+2d}
   \lim_{\beta \to 0}   \frac{\log |\log |\bar \gamma_{\beta}(p)||}{|\log \beta|}= 1+\tfrac{2}{d}.
  \end{equation}
 \end{itemize}
\end{theorem}

\noindent
If $d=1$ and $d=2$, these results match the asymptotics  
for the free energy of the directed polymer in a random environment (see \cite{Nak19} for $d=1$ and \cite{BL17} for $d=2$).
In dimensions $d\geq 3$, however, one has in the directed polymer setting that $\bar\gamma_{\gb}(p) =0$ for sufficiently small $\gb$, see~\cite{Bol89} and~\cite[Prop.~4.1]{Vi19}.
  Note that \eqref{eq:1+2d} implies that $|\bar \gamma_{\beta}(p)|$ vanishes very fast as $\beta$ goes to zero. As a consequence, there is not a lot of margin to play with to prove that  $\lvert\bar\gamma_{\gb}(p)\rvert>0$ when $\beta$ is small, and the proof indeed relies on a fine tuning of parameters.
The behavior of $\bar \gamma_{\beta}(p)$ is different if $\gl$ displays an heavier tail at infinity. In this case, $\bar \gamma_{\beta}(p)$ rather has a power-law behavior close to $\beta=0$.
Let us introduce 
\begin{equation}\label{defnup}
 \nu_p:=1-\tfrac d2 (p-1).
\end{equation}
\begin{theorem}[Heavy-tailed noise]
\label{thm:heavyup}
Assume   that \eqref{eq:log-2} holds and that $\mu<\infty$.
\bit
  \item [(i)] If $\mu_{1,\infty}(q)<\infty$ for some $q\in ( 1, \min(2,1+\frac{2}{d}) )$,
 then for every $p\in (0,q]\setminus \{1\}$, there exists $C_{p,q}\in(0,\infty)$ such that for every $\beta \in (0,1]$,
\begin{equation}\label{sioum}
 \lvert\bar \gamma_{\beta}(p)\rvert \le C_{p,q}  \beta^{\frac{q}{\nu_q}}\,,
 \end{equation}
which implies
\beq\label{eq:liminf} \liminf_{\beta \to 0}   \frac{\log  \lvert\bar \gamma_{\beta}(p)\rvert}{\log \beta}\ge \frac{q}{\nu_q}. \eeq

    \item [(ii)] If $\mu_{1,\infty}(q)=\infty$ for some $q\in ( 1,  \min(2, 1+\frac{2}{d} ) )$, then 
   for every $p\in (0,q)\setminus \{1\}$,
    \begin{equation}\label{eq:lastone} 
\limsup_{\beta \to 0}   \frac{\log  \lvert\bar \gamma_{\beta}(p)\rvert}{\log \beta}\le \frac{q}{\nu_q}. 
    \end{equation}
       \eit
   In particular, if  
   \beq\label{eq:alpha}\lim_{z\to \infty} \frac{\log \gl([z,\infty))}{\log z}=-\alpha\eeq 
   for some $1<\alpha < \min (2,1+\frac2d )$ (which includes the $\alpha$-stable case), then for all $p\in(0,\al)\setminus\{1\}$
   \begin{equation}\label{eq:stable}
   	\lim_{\beta \to 0}   \frac{\log \lvert\bar \gamma_{\beta}(p)\rvert}{\log \beta}= \frac{\alpha}{\nu_\al} \,.
   \end{equation}
\end{theorem}

\subsection{Overview of the rest of the paper}

The remainder of the paper is organized as follows.
In Section  
\ref{sec:comments}, we discuss some  extensions of our results and some related conjectures concerning geometric localization.
Section~\ref{sec:prelim} introduces a few important tools
that are needed throughout the paper.
 In Section~\ref{sec:properties}, we prove all basic properties concerning the moment Lyapunov exponents, that is, Propositions~\ref{prop:finitemom} and \ref{prop:qual}.
Section~\ref{sec:upper} is devoted to proving our first main technical achievement: Propositions~ \ref{prop:localbisdone} and~\ref{prop:localbis}, which contain upper estimates on moments of order $p>1$ of $\cZ^{\go}_{\beta}(t,x)$. 
In Section \ref{sec:deducing}, we then use
these moment bounds to derive
Theorem \ref{thm:local} and the first halves of Theorems \ref{thm:thinup} and \ref{thm:heavyup}. The existence and uniqueness of solutions (\textit{i.e.}, Theorems~\ref{thm:SHE} and~\ref{thm:SHEunique}) are addressed in Section~\ref{sec:SHE}. 
Finally, 
 Section~\ref{sec:theotherbound} gathers the proofs that rely on a coarse-graining/change-of-measure approach summarized in Lemma~\ref{lem:changemeas}.
In particular, we show the remaining halves of Theorems  \ref{thm:thinup} and \ref{thm:heavyup} as well as Proposition \ref{prop:waitingforbetter}.
 The Appendix contains proofs of some of the results introduced in Section~\ref{sec:prelim}.

\section{Extensions and consequences  of our  results}\label{sec:comments}

\subsection{Continuum directed polymer model}\label{sec:dirpol}
As mentioned in Remark \ref{rem:dirpol}, the main purpose of   introducing $\mathcal Z^{\go,a}_\gb(t,\ast)$ in \cite{BL20_cont} is not to study   the SHE but rather to define a random probability measure on the space of continuous functions
$$C_0([0,T]):=\{ \theta: [0,T]\to \bbR^d \ : \ \theta \text{ continuous and } \theta(0)=0 \}, $$ called the continuum directed polymer in a Lévy environment.
For $T>0$ the directed polymer in the  truncated environment $\xi^a_{\go}$ is the probability measure $\bQ^{\go,a}_{T,\beta}$ on $C_0([0,T])$ defined via the integral of bounded Borel measurable functions as follows:
\begin{equation}
 \bQ^{\go,a}_{T,\beta}(f)
= \frac{e^{-\beta\kappa_a t } }{\cZ^{\go,a}_{\beta}(T,\ast)}\left(\bQ\Big[ f((B_t)_{t\in[0,T]})\Big]+ \sum_{k=1}^{\infty}\gb^k
\int_{ \frakX_k(T) \times (\bbR^d)^k}   \rho( \bt , \bx,f)   \prod_{i=1}^k  \xi_{\go}^{a ,+}  (\dd t_i , \dd x_i )\right),
\end{equation}
where $\rho(\bt,\bx,f):= \bQ [ f((B_t)_{t\in[0,T]})  \ | \ \forall i\in \lint k\rint: B_{t_i}=x_i ] $ and $\lint k\rint:=\{1,\dots,k\}$. Here,  $\bQ$ denotes the distribution of a standard $d$-dimensional Brownian motion  and,  with some light abuse of notation,  $\bQ [ \cdot  \ | \ \forall i\in \lint k\rint: B_{t_i}=x_i ]$ is the law of the concatenation of Brownian bridges obtained by conditioning a Brownian motion on the null event  $\{\forall i\in \lint k\rint: B_{t_i}=x_i\}$.
The main result in \cite{BL20_cont} is to prove that  $\bQ^{\go,a}_{T,\beta}$ converges in distribution to a limit 
$\bQ^{\go}_{T,\beta}$ under the assumptions \eqref{eq:log} and \eqref{assump1}. 
By the findings of the present paper, we can  replace the assumption \eqref{assump1} by \eqref{eq:log-2} in the results of \cite{BL20_cont}.

\begin{theorem}
If the measure $\gl$ satisfies \eqref{eq:log} and \eqref{eq:log-2},
then there exists a measure $\bQ^{\go}_{T,\beta}$ on $C_0([0,T])$ such that almost surely, for every bounded continuous function $f$ in $C_0([0,T])$, we have 
\begin{equation}
 \lim_{a\to 0} \bQ^{\go,a}_{T,\beta}(f)=  \bQ^{\go}_{T,\beta}(f).
\end{equation}
\end{theorem}
\begin{proof}
The convergence of $\bQ^{\go,a}_{T,\beta}(f)$ for a fixed $f$ is a consequence of Theorem \ref{thm:local}. Indeed, $\bQ^{\go,a}_{T,\beta}(f)$ can be written as a quotient $\cZ^{\go,a}_{\beta}(T,f)/\cZ^{\go,a}_{\beta}(T,\ast)$. 
Theorem \ref{thm:local} asserts that the denominator converges to a positive limit and the martingale argument below \eqref{filtraG} entails the almost-sure convergence of $\cZ^{\go,a}_{\beta}(T,f)$.
The fact that  convergence holds simultaneously for all continuous bounded functions is a consequence of tightness.
The proof given in \cite[Section 4.8]{BL20_cont} shows that the family of non-normalized measures $\{\cZ^{\go,a}_{\beta}(T,\ast)\times \bQ^{\go,a}_{T,\beta}:a\in(0,1]\}$ is tight, with no other assumption than \eqref{eq:log}.
The tightness of $\{\bQ^{\go,a}_{T,\beta}:a\in(0,1]\}$ then follows from the almost-sure positivity of $\cZ^{\go,a}_{\beta}(T,\ast)$; cf.~\eqref{eq:conv-2}.
\end{proof}

\subsection{Geometric localization of the solution to the SHE}

\label{sec:geoloca}

\subsubsection*{The case $u_0\equiv 1$.}
The intermittency result  given in Theorem \ref{thm:int} (especially the part concerning  strong intermittency) has direct implications on the distribution of mass for the solution to \eqref{eq:integz} with initial condition $u_0\equiv 1$. To illustrate this, let us present an argument from \cite[\S 2.4]{CB95}, adapted to the case of non-integer moments---the argument in \cite[\S 2.4]{CB95} is about mass concentration for the square of the solution.

If $\mu<\infty$ and \eqref{eq:log-2} holds, then the solution  to \eqref{eq:integz} with $u_0\equiv 1$ is given by
\begin{equation*}
 U(t,x):= \int_{\bbR^d} \cZ^{\go}_{\beta}(y;t,x)\,\dd y.
\end{equation*}
Let us define  the normalized solution $\bar U(t,x) := e^{-\mu t} U(t,x)$.
By symmetry, we have that  $\bar U(t,x)\stackrel{(d)}{=}\bar \cZ^{\go}_\beta(t,\ast)$  for any fixed $x$.
For a fixed value of $t$, let us assume that the field $\bar U(t,x)_{x\in \bbR^d}$ is ergodic (in the case of a Gaussian noise, this was proved in~\cite{Chen19b}): we then get
\begin{equation}\label{mass1}
\lim_{R\to \infty} \frac{\int_{\|x\|\le R} \bar U(t,x)\, \dd x}{ \sigma_d R^d}= \bbE\Big[ \bar  \cZ^{\go}_\beta(t,\ast) \Big] = 1 \,,
\end{equation}
where $\sigma_d= \pi^{d/2}\gG(\tfrac d 2 +1)$ is the volume of the unit ball in $\bbR^d$.
On the other hand, if $p\in(0,1)$ and $\bar \gamma_{\beta}(p)<0$ (which is ensured by Theorem \ref{thm:int}), then for any $\alpha\in  (0, \frac{\bar \gamma_{\beta}(p)}{p-1} )$ we have
\begin{equation}\label{mass2}
\begin{split}
\lim_{R\to \infty} \frac{\int_{\|x\|\le R} \bar U(t,x)\ind_{\{ \bar U(t,x)\le e^{ \alpha t}\}}\, \dd x}{ \sigma_d R^d}
&= \bbE\Big[ \bar\cZ^{\go}_\beta(t,\ast)\ind_{\{ \bar\cZ^{\go}_\beta(t,\ast)\le e^{\alpha t } \}} \Big] 
\\ &\le e^{(1-p)\alpha t} \bbE[\bar\cZ^{\go}_\beta(t,\ast)^p]
\le e^{[(1-p)\alpha+\bar \gamma_{\beta}(p)] t},
\end{split}
\end{equation}
where the last inequality relies on super-multiplicativity, see Lemma \ref{lem:submult} below.
Setting $\delta:=\alpha(p-1)-\bar \gamma_{\beta}(p)>0$ and
combining \eqref{mass1} and \eqref{mass2}, we obtain
\begin{equation}
\label{mass3}
 \lim_{R\to \infty} \frac{\int_{\|x\|\le R} \bar U(t,x)\ind_{\{\bar U(t,x)\ge e^{\alpha t}\}} \,\dd x}{\int_{\|x\|\le R} \bar U(t,x)\, \dd x}\ge 1- e^{-\delta  t}  \,.
\end{equation}
Note that we also have 
\begin{equation}
\label{mass4}
\lim_{R\to \infty} \frac{\int_{\|x\|\le R} \ind_{\{\bar U(t,x)\geq  e^{\alpha t}\}} \,\dd x}{ \sigma_d R^d}
= \bbP\Big( \bar\cZ^{\go}_\beta(t,\ast)\ge e^{\alpha t }   \Big) \leq  e^{-\alpha t} \,,
\end{equation}
thanks to Markov's inequality.
Hence, the two identities~\eqref{mass3} and \eqref{mass4} show that in the large $t$ limit, the mass of the solution concentrates on a very small portion of space. 


\subsubsection*{The case $u_0= \delta_0$.}
Although this is much more difficult to prove rigorously, we 
believe that intermittency in Theorem \ref{thm:int}  also has implications on the localization of the solution $\mathcal Z^{\go,a}_{\beta}(t,x)$ to the SHE with $\delta_0$ initial condition.
Let us consider   the probability measure on $\bbR^d$ given by
$$ P^{\go}_{\beta,t}(\dd x)=  \frac{\mathcal Z^{\go,a}_{\beta}(t,x)}{\mathcal Z^{\go,a}_{\beta}(t,\ast)}\, \dd x.$$
When $\beta=0$, $P^{\go}_{\beta,t}$ is simply $\rho(t,x)\,\dd x$.
In this case, for large $t$, its mass is roughly homogeneously spread out on the centered Euclidean ball of radius $t^{1/2}$.
By contrast, if $\beta>0$, the large time behavior of 
$P^{\go}_{\beta,t}(\dd x)$ is conjectured to be different:
the mass of $P^{\go}_{\beta,t}$ should typically be concentrated on a set of bounded volume (the volume here should not depend on $t$). 

\begin{conj}\label{conjdificil}
 Given $\gep>0$ and $\beta>0$, there exist two constants $k=k(\gep,\beta)\in \mathbb N$ and $R=R(\gep,\beta)>0$ such that for all $t\geq0$,
 \begin{equation*}
 \bbP\Bigg( \exists z_1, z_2, \dots, z_k \in \bbR^d \ : \  P^{\go}_{\beta,t}\bigg(\bigcup_{i=1}^k B(z_i,R)\bigg) \ge 1-\gep \Bigg) \ge 1-\gep,
 \end{equation*}
where $B(z,R):=\{x\in\R^d: \lVert x-z\rVert < R\}$.
\end{conj}

A weaker version of the above conjecture is that   localization holds in Ces\`aro mean, that is,  for $\gep>0$ and $\beta>0$ fixed, there exist two constants $k=k(\gep,\beta)\in \mathbb N$ and $R=R(\gep,\beta)>0$ such that 
\begin{equation}\label{cesarostuff}
 \liminf_{T\to \infty} \frac{1}{T} \int^T_0  \max_{z_1,\dots, z_k\in \bbR^d}P^{\go}_{\beta,t}\bigg(\bigcup_{i=1}^k B(z_i,R)\bigg)\, \dd t \ge  1-\gep.
\end{equation}
In fact, such statements have been rigorously proved for a discrete analogue of the SHE: the partition function of the discrete polymer in $\bbZ^d$.
In this setup, the link between    \textit{very strong disorder} and localization in the Ces\`aro sense was first explored in \cite{CH02,CSY03}. More recently, in \cite{BC20}, it was shown that for directed polymers, 
very strong disorder  implies \eqref{cesarostuff}; the result was extended to continuous space and discrete time in \cite{BaSe20}.
The approach in \cite{BaSe20,BC20} is to prove that the measure $\frac{1}{T}\int^T_0 P^{\go}_{\beta,t}\,\dd t$ converges to a limiting object, using a specific  topology (introduced in \cite{MuVa16}) on the space of finite measures on~$\bbR^d$.
The corresponding localization result has also been shown for a continuum directed polymer model (and SHE) with spatially convoluted Gaussian white noise in~\cite{BM19}, using a Gaussian multiplicative chaos approach.

\medskip

While many details of the proofs presented in \cite{BaSe20,BC20} use the discrete nature of the polymer model, the general ideas the proof is based on do not seem to rely on it. For this reason, we believe that while there are technical challenges to be overcome, these proofs could in principle be adapted to the continuum case, to show that \eqref{cesarostuff} holds whenever $\bar \gamma_\beta(p)<0$ for $p\in (0,1)$ and hence, by Theorem \ref{thm:int}, for any $\beta>0$ as soon as $\mu<\infty$.

\begin{rem}
Note that the localization result in \cite{BC20} holds for every $\beta$ in dimensions~$1$ and $2$ but only for $\beta$ above a certain threshold in dimensions $d\geq3$. This is because for the directed polymer, the analogue of our Theorem \ref{thm:int} does not hold. Quite the contrary, if $d\ge 3$  and $\beta$ is small, the end-point distribution of the directed polymer satisfies the central limit theorem \cite{Bol89,CY06,IS88}.
\end{rem}

\begin{rem}
 The above discussion gives a justification of Conjecture \ref{conjdificil} in the case when $\mu<\infty$ (which implies that $\mathcal Z^{\go}_{\beta}(t,\ast)$ has a finite expectation).
 As a general rule, an   environment with heavier tail is expected to only increase geometric localization (cf.\  \cite{Va07}).
\end{rem}

\begin{rem}
 When $d=1$, known results concerning the distribution of $\cZ^\om_\beta(t,x)$ in the case of Gaussian white noise \cite{ACQ11} strongly suggest that, in fact, only one big ball (that is, $k=1$) is sufficient to capture most of the mass of $P^{\go}_{\beta,t}$. For a proof of a result of this kind in the discrete polymer  setup (with a special boundary condition), we refer to \cite{CoVu14}.
 In dimension $d\ge 2$, whether one or finitely many balls are needed is a challenging open problem, even at the heuristic level, see~\cite[Section 9, (1)]{BC20}.
 \end{rem}



\subsection{The case of more general noise}
\label{sec:genelevy}


In general, a Lévy white noise $\xi$ (on $\bbR\times\bbR^{d}$) is a random distribution in $\bigcap_{s> ({1+d})/{2}} H^{-s}_{\mathrm{loc}}(\R\times\bbR^d)$  that can be decomposed as follows:
\begin{equation}\label{noisedeco}
 \xi= \xi_1-\xi_2+ \alpha_1 \xi_3+ \alpha_2\mathcal{L}
\end{equation}
where $\xi_1$ and $\xi_2$ are two independent Lévy noises with respective jump intensities $\gl_1$ and $\gl_2$, $\xi_3$ is a standard Gaussian space-time white noise that is independent of $\xi_1$ and $\xi_2$, and $\alpha_1\ge 0$ and $\alpha_2\in \bbR$ are  constants. In our exposition so far, we have chosen to focus on the case $\xi_2\equiv 0$ and $\alpha_1=\alpha_2=0$. 
The constant $\alpha_2$ is irrelevant for any discussion of \eqref{eq:SHELN} as it only adds a multiplicative exponential in $t$ to the solution. 

Adding a Gaussian part in dimension $1$ does not pose a problem for solving \eqref{eq:SHELN}; see \cite{Chong16}. Given that $\ga(p)$ in \eqref{eq:fml} exhibits the same behavior in $\beta$ as $\ga_\beta(p)$ in  \eqref{eq:d1}, we believe that both Theorem~\ref{thm:thinup} and \ref{thm:heavyup}  remain valid when Gaussian noise is added. 
If  $d\geq2$, the SHE with multiplicative Gaussian white noise is known to be degenerate in the following sense. If $\xi^{(\gep)}$ denotes the noise obtained by space convolution of $\xi$ with a smooth kernel $\theta_{\gep}:=\gep^{-d} \theta( \gep^{-1} \cdot)$, then 
the solution of 
$$ \partial_t u^{(\gep)}= \frac 1 2 \Delta u^{(\gep)} + \beta \xi^{(\gep)} u^{(\gep)}, \quad    u^{\gep}(0,\cdot)=\delta_0$$ 
satisfies $\lim_{\gep\to 0} u^{(\gep)}(t,\cdot)= 0$ (not only pointwise but also in $L^1(\bbR^d)$ for any fixed $t>0$).
In the past years, important progress has been made to obtain  convergence of the solutions to non-trivial (non-Gaussian) limits   for $d=2$ when $\beta$  is sent to zero (in a specific critical window) jointly with  $\gep$, 
see~\cite{CSZ19,GQT21} and the recent breakthrough~\cite{CSZ21}.

One can also consider a noise $\xi=\xi_1-\xi_2$ that includes negative jumps. 
In this case, one can approximate $\xi$ by $\xi^a=\xi^a_1-\xi^a_2$.
Then
 the criterion \eqref{eq:log} applied to both jump intensities $\gl_1$ and $\gl_2$
 still implies the  convergence of the integrals in \eqref{absol}, and
 many of the proofs presented in this paper can be applied verbatim.
 Note that when negative jumps are present in the noise, the solution is not necessarily positive.
 In particular, in this case, $\mathcal Z^{\go}_{\beta}(t,x)$ does not correspond to the partition function of a polymer model.

Applying the proofs presented below,  Theorems \ref{thm:local}, \ref{thm:SHE} and \ref{thm:SHEunique} remain valid when signed noise is considered, and the same holds true for the  upper bounds of 
$\bbE [  |\cZ^{\go}_{\beta}(t,0)|^p ]$ for $p>1$. In fact, this translates into half of the inequalities (the upper bounds) proved in Theorems \ref{thm:thinup} and \ref{thm:heavyup}.
On the other hand, our proof of Proposition \ref{prop:finitemom} and the upper bounds on $\bbE\big[  |\cZ^{\go}_{\beta}(t,0)|^p\big]$ for $p\in(0,1)$ rely in a crucial manner on the positivity of $\cZ^{\go}_{\beta}(t,0)$, and thus break down when allowing for negative jumps. 
Note that the upper bounds on fractional moments of the partition function are required for the other half of the inequalities (the lower bounds) in
Theorems \ref{thm:thinup} and \ref{thm:heavyup}, and in particular, for Theorem \ref{thm:int}.
We believe nonetheless that   intermittency occurs in the following sense.

\begin{conj}
If $\xi =\xi_1-\xi_2$ is a signed Lévy white noise as above, with jump intensities that satisfy \eqref{eq:log-2} and $\int^{\infty}_0 z^{p_0} \,(\gl_1+\gl_2)(\dd z)$ for some
$p_0< \min(2,1+\tfrac2d)$, then
\[
\gamma_\beta(p):=\lim_{t\to\infty} \frac{1}{t}\log \bbE [   | Z^{\go}_{\beta}(t,0) |^p ]
\]
is well defined for $p\in[0,p_0]$ and $p\mapsto   \gamma_\beta(p)/p$ is an increasing function of $p$ on the interval $[1,p_0]$.
\end{conj}

\section{Preliminary tools}
\label{sec:prelim}

\subsection{Size-biased measure}\label{sec:sizebias}

If $\mu<\infty$, note that
$\bar \cZ_{\beta}^{\go,a}(t,\ast)$ is non-negative by \eqref{eq:help} and satisfies $
\bbE  [\bar \cZ_{\beta}^{\go,a}(t,\ast) ]=1$.
We can therefore define a new measure $\wt \bbP^a_{\beta,t}$  for the environment by
\begin{equation}
\label{eq:sizebiased}
\wt \bbP^a_{\beta,t}(\go\in A):= \bbE\Big[ \bar \cZ_{\beta}^{\go,a} (t,\ast) \ind_{A}\Big] \, ,
\end{equation}
which is referred to as the \emph{size-biased measure}.
Note that when  
$(\bar \cZ_{\beta}^{\go,a}(t,\ast))_{a\in(0,1)}$ is uniformly integrable, then using martingale convergence, we can also consider
$\wt \bbP^0_{\beta,t}$ whose density with respect to $\bbP$ is given by $\bar \calz^{\om}_{\beta}(t,\ast)$.

Lemma~3.7 in~\cite{BL20_cont} gives a useful representation of 
the size-biased measure $\wt \bbP^a_{\beta,t}$:
the distribution of $\go$ under $\wt \bbP^a_{\beta,t}$ 
is obtained by adding to $\go$ an independent Poisson process on the trajectory of a Brownian motion.  
More precisely, we let $\go'$ be a Poisson point process on $\R\times (0,\infty)$ with intensity $\dd t \otimes \gb z\, \lambda(\dd z)$ and we denote its law by $\bbP'_{\beta}$. 
Furthermore, let $(B_t)_{t\in [0,\infty)}$ be a standard $d$-dimensional Brownian motion starting from $0$ that is independent of both $\om$ and $\om'$ and denote its distribution by $\bQ$.
For $a\in[0,1]$, we then define
\begin{align*}
 \wt\go(\go',B)&:=\Big\{ (\tau, B_\tau, \zeta) :  (\tau,\zeta)\in \go' \Big\},\\
  \wt\go_{a,t}(\go',B)&:=\Big\{ (\tau, B_\tau, \zeta) :  (\tau,\zeta)\in \go', \tau\in[0,t]\,, \zeta\geq a  \Big\} \,,
\end{align*}
which corresponds to putting the Poisson process $\go'$ on the trajectory of $B$, with some restrictions on $\tau$ and $\zeta$ in the second case. When $a=0$, the dependence in $a$ is sometimes omitted.

\begin{lemma}\label{lem:sizebias}
If $\mu <\infty$, then for any $\beta,t>0$ and $a\in (0,1]$ and for any measurable bounded function $g$, we have
\begin{equation}\label{sides}
 \wt \bbE^a_{\beta,t} [ g(\go) ]= \bbE\otimes \bbE'_{\beta}\otimes \bQ [ \wh g_{a,t}(\go, \go', B) ],
\end{equation}
where $\wh g_{a,t}( \go, \go', B):= g( \go \cup \wt \go_{a,t}(\go',B))$.
 If
$(\bar \cZ_{\beta}^{\go,a}(t,\ast))_{a\in(0,1)}$ is uniformly integrable, then \eqref{sides} is also satisfied for $a=0$.
\end{lemma}

If $a>0$, this was shown in~\cite[Lemma~3.7]{BL20_cont}. For $a=0$, the result  can be deduced from the  case $a>0$ by observing that $\bbE[ \bar \calz^{\om}_{\beta}(t,\ast) \mid \cG_a] = \bar \calz^{\om,a}_{\beta}(t,\ast)$ for any $a\in(0,1]$, which is why
\eqref{sides} with $a=0$ is true for any bounded $\mathcal G_a$-measurable $g$. The monotone class theorem yields \eqref{sides} for general $g$. Note that by Theorem \ref{thm:local}, $(\bar \cZ_{\beta}^{\go,a}(t,\ast))_{a\in(0,1)}$ is uniformly integrable as soon as \eqref{eq:log-2} holds.

\begin{rem}
Let us make a quick comment about notation. In \eqref{sides} as well as in the rest of the paper, we  denote the expectation with respect to a measure $\P$, possibly with decorations, by $\E$ with the same decorations. 
For example, we write $\wt \bbE^a_{\beta,t}$ and  $\E'_{\beta}$ for expectation with respect to $\wt\P^a_{\beta,t}$ and $\P'_{\beta}$, respectively.
\end{rem}

\subsection{Decoupling inequalities}\label{sec:dec}

In the proofs below, we shall occasionally make use of \emph{moment decoupling inequalities} for multiple Poisson integrals. The results that we review in this section are based on \cite{Kallenberg17}, where discrete-time decoupling techniques from \cite{delaPena99, Kwapien92} are extended to continuous time. We also refer the reader to these references for further literature (and history) on the subject. 
Let $N\geq1$ be an integer, let $\X\ceq  (0,\infty)\times\R^d\times (0,\infty)$, with generic point $w\ceq (t,x,z)$, and for concreteness let $M=M^\om$ be the compensated Poisson random measure 
$M^\om(\dd w)=(\delta_\om-\nu)(\dd w)$ (recall that $\nu:=\dd t\otimes\dd x\otimes\la(\dd z)$). Furthermore, suppose that $f: \X^N\to \R$ is measurable and \emph{tetrahedral}, that is,
\[
f((t_1,x_1,z_1),\dots,(t_N,x_N,z_N))=0\quad \text{unless} \  t_1<\dots<t_N \,. 
\]
In that case, subject to integrability conditions on $f$, the multiple integral
\[
\int_{ \X^N} f(\bw) \,M(\dd w_1)\cdots M(\dd w_N) \ceq \int_\X \Biggl(\cdots\Biggl(\int_\X f(\bw)\,M(\dd w_1)\Biggr)\cdots\Biggr)\, M(\dd w_N)\,,
\]
with $\bw=(w_1,\dots,w_N)$,
can be defined as an iterated It\^o integral (recall Remark \ref{rem:Ito}). The general definition of It\^o integrals used in this paper is detailed in Section~\ref{sec:sol}. 
We let $(\om_i)_{i=1}^N$ denote i.i.d.\ copies of our environment $\om$ and we use the notations $\bbP^{\otimes N}$ and $\bbE^{\otimes N}$ for the associated probability and expectation, respectively. The following theorem is proved in Appendix \ref{app:A}.

\bthm\label{thm:dec}  
There exists a universal constant $C>0$ such that for every  $1<p\leq 2$,  for any integer $N\geq1$ and any measurable and tetrahedral function $f:  \X^N\to\R$ satisfying
\begin{equation}
\label{eq:fcond} \int_{ \X^N} \lvert f(\bw)\rvert^p\,\nu(\dd w_1)\cdots\nu(\dd w_N)<\infty\,,
\end{equation}
the stochastic integrals below are well defined and have finite moments of order $p$ satisfying
\begin{equation}\label{eq:dec-2} \begin{split}
&\bigg(\frac{p-1}{C}\bigg)^N \E^{\otimes N}\Biggl[ \Biggl\lvert\int_{ \X^N} f(\bw)\,M^{\om_1}(\dd w_1)\cdots\,M^{\om_N}(\dd w_N)\Biggr\rvert^p\Biggr]\\
&\qquad\qquad\qquad\qquad\leq\E\Biggl[ \Biggl\lvert\int_{ \X^N} f(\bw)\,M^\om(\dd w_1)\cdots\,M^\om(\dd w_N)\Biggr\rvert^p\Biggr]\\
&\qquad\qquad\qquad\qquad\qquad\qquad\leq \bigg(\frac{C}{p-1}\bigg)^N \E^{\otimes N}\Biggl[ \Biggl\lvert\int_{ \X^N} f(\bw)\,M^{\om_1}(\dd w_1)\cdots\,M^{\om_N}(\dd w_N)\Biggr\rvert^p\Biggr].
\end{split}
\end{equation}

\ethm

\subsection{Stochastic integration in the absence of moments}\label{sec:sol}

Because $\calz^\om_\beta(t,x)$ does not possess any finite moments in general, the proof of Theorem~\ref{thm:SHE} has to make use of a stochastic integration theory that does not assume existence of any moments a priori. For the reader's convenience, let us give a brief review of this $L^0$-theory, which was developed by \cite{Bichteler83} in its most general form (see also \cite[Appendix~A]{Chong19} for a summary). Let $M=M^\om$ denote either the Poisson measure $\delta_\om$ or its compensated version $\delta_\om-\nu$.

A \emph{predictable step process} $H$ is of the form $H=\sum_{i=1}^r a_i \bone_{A_i}$, where $r\in\N$, $a_i\in\R$, and for each $i$ we have $A_i\in\calp\otimes\calb(\R^d\times(0,\infty))$ (where $\cP$ is the usual predictable $\sigma$-field) and $A_i\subseteq \Om\times (0,T)\times [-N,N]^d \times(a,\infty)$ for some $T,N,a>0$.
Then, the integral of the predictable step process $H=\sum_{i=1}^r a_i \bone_{A_i}$
 is canonically defined as
\[
 \int_\X  H(\bw)\,M(\dd \bw) := \sum_{i=1}^r a_i M(A_i)\,. 
 \]
Denoting by $\cals$ the collection of predictable step processes, we can extend the integral to a larger subset of predictable processes by using the metric induced by
\begin{equation}
\label{eq:semivar} 
\lVert K\rVert_{M,p} \ceq \sup_{H\in\cals, \lvert H\rvert\leq \lvert K\rvert } \Biggl\lVert \int_\X  H(\bw)\,M(\dd \bw) \Biggr\rVert_{L^p}, 
 \end{equation}
defined for $\calp\otimes\calb(\R^d\times(0,\infty))$-measurable processes $K$,
where we have used the notation $\lVert X\rVert_{L^p}\ceq\E[\lvert X\rvert^p]^{1/p}$ and $\lVert X\rVert_{L^0}\ceq\E[1\wedge \lvert X\rvert]$. Such a process $K$ is
called \emph{$L^p$-integrable with respect to $M$} if there exists a sequence $(H_n)_{n\in\N}\subseteq\cals$ such that $\lim_{n\to \infty}\lVert K-H_n \rVert_{M,p} =0$. 
The stochastic integral of $K$ with respect to $M$ is then defined as the $L^p$-limit of $\int_\X H_n(\bw)\,M(\dd \bw)$,
which exists and does not depend on the choice of $(H_n)_{n\in\N}$. If $p=0$, we simply say that $K$ is \emph{integrable with respect to $M$}. According to \cite[(2.8)]{Bichteler83}, $K$ is $L^p$-integrable with respect to $M$ if and only if
\beq\label{eq:integ}
 \begin{cases}
  \lVert K\rVert_{M,p}<\infty&\text{if } p>0,
 \\ \displaystyle\lim_{u\to0} \lVert uK\rVert_{M,0}=0&\text{if } p=0.
 \end{cases}
  \eeq
Furthermore, by \cite[Lemma~A.2]{Chong19}, 
\bit
\item if $1\leq p<\infty$ and $M^\om=\delta_\om-\nu$, there are $c=c_p>0$ and $C=C_p>0$ such that
\begin{equation}\label{eq:condizione}
c\lVert K\rVert_{M,p}\leq \E\Biggl[\biggl(\int_\X K(\bw)^2\,\delta_\om(\dd \bw)\biggr)^{\frac p2} \Biggr]^{\frac 1p}\leq C\lVert K\rVert_{M,p};
\end{equation}
\item if $0<p\leq 1$ and $M^\om=\delta_\om$, then
$$ \lVert K\rVert_{M,p}^p\leq \int_\X \E[\lvert K(\bw)\rvert^p]\,\nu(\dd \bw).$$
\eit
In particular, the $L^p$-theory, $p\geq1$, for $\delta_\om-\nu$ encompasses the $L^p$- and $L^2$-integrals considered in \cite{Loubert98} and \cite{Walsh86}. Even in the case when $\delta_\om-\nu$ has a finite first moment (\textit{i.e.}, is a martingale measure), the $L^0$-theory is more general since, for example, it does not require the integrands  have finite moments.

Let us end this section with a simple integrability criterion that we need in the proof of Theorem~\ref{thm:SHE}. Its proof can be found in the appendix.

\blem\label{lem:integrab} 
Suppose that for $\P_\geq$-a.e.\ realization of $\om_\geq$, the process $(\om_<, t,x)\mapsto K(\om_< \cup\om_\geq,t,x)$ is $L^p(\P_<)$-integrable with respect to $\xi_{\om_<}$ for some $p>0$. Then $K$ is $L^0(\P)$-integrable with respect to $\xi_{\om_<}$.
\elem

\subsection{A technical lemma}

The following technical result will be used repeatedly in the paper.
Its proof comes from a straightforward calculation and can be found in~\cite[Lemma~A.3]{BL20_cont}.
 
\begin{lemma}
\label{lemGamma}
For any $t >0$, $k\ge 0$ and $\zeta_1,\dots ,\zeta_{k+1} >0$, 
\[
\int_{ \mathfrak{X}_k(t)}  \prod_{i=1}^{k+1} (\Delta t_i)^{\zeta_i-1} \,\dd t_i    = t^{\sum_{i=1}^{k+1}\zeta_i-1} \frac{\prod_{i=1}^{k+1}\Gamma(\zeta_i)}{\Gamma(\sum_{i=1}^{k+1}\zeta_i)} \,,
\]
where $\Delta t_i := t_i-t_{i-1}$ as defined in~\eqref{deltadeff}, with the convention $t_0=0$ and $t_{k+1}=t$.
\end{lemma}

\section{Qualitative properties of Lyapunov exponents}
\label{sec:properties}
In this section, we prove the statements made in Section \ref{sec:Lyap} concerning existence and basic properties of moment Lyapunov exponents (Propositions~\ref{prop:finitemom} and \ref{prop:qual}), assuming that Theorem \ref{thm:local} as well as the bounds in Propositions~\ref{prop:localbisdone} and \ref{prop:localbis} below are true. These are proved independently in Sections~\ref{sec:upper} and \ref{sec:deducing}, respectively.

\subsection{Existence of Lyapunov exponents}\label{subsec:exiss}

The existence of $\ga_\beta(p)$  readily  follows from sub-/ supermultiplicative properties of the moments of the partition function. 
\begin{lemma}
 \label{lem:submult}
Under assumption \eqref{eq:log}, for every $a\in(0,1]$, $s,t>0$ and $p\in[0,1+\frac 2d)$,  we have
\begin{gather}\label{eq:supersub1}
   \E[\calz^{\go,a}_{\beta}(t+s,\ast)^p]\le   \E[\calz^{\go,a}_{\beta}(s,\ast)^p]    \E[\calz^{\go,a}_{\beta}(t,\ast)^p]  \qquad \text{ if } p\ge 1,\\
    \E[\calz^{\go,a}_{\beta}(t+s,\ast)^p]\ge \E[\calz^{\go,a}_{\beta}(s,\ast)^p]   \E[\calz^{\go,a}_{\beta}(t,\ast)^p] \qquad  \text{ if } p\le1.\label{eq:supersub2}
\end{gather}
\end{lemma}
\bpr  The statement being trivial if $\mu_{1,\infty}(p)=\infty$, we  can assume $\mu_{1,\infty}(p)<\infty$ for the remainder of the  proof.
Let $P_{\gb,t}^{\go,a}$ denote the (random) probability measure on $\R^d$ whose density with respect to the Lebesgue measure is given by $\cZ_{\beta}^{\go,a}(t,x)/\cZ^{\om,a}_{\beta}(t,\ast)$ and $X$ denote an $\R^d$-valued random variable with distribution $P_{\gb,t}^{\go,a}$.
Recalling the definition~\eqref{eq:fepf} of $\cZ_{\beta}^{\go, a} (s,x;t,\ast)$
and using the Markov property for Brownian motion,
we have
\begin{align*}
	\cZ_{\beta}^{\go,a} (s+t,\ast)  &= \int_{\bbR^d}  \cZ_{\beta}^{\go,a} (s,x)  \cZ_{\beta}^{\go, a} (s,x;s+t,\ast)  \,\dd x \\
	&= \cZ_{\gb}^{\go,a}(s,\ast) \int_{\bbR^d}  \cZ_{\beta}^{\go, a} (s,x;s+t,\ast)  \,P_{\gb,s}^{\go,a} ( \dd x )\\
	& =  \cZ_{\gb}^{\go,a}(s,\ast) \, E_{\gb,s}^{\go,a} \Big[ \cZ_{\beta}^{\go,a}  (s,X;s+t,\ast)\Big] \,.
\end{align*}
For $p\geq 1$, we get by Jensen's inequality 
that 
\begin{equation*}
	\cZ_{\beta}^{\go,a} (s+t,\ast)^p  \leq  \cZ_{\gb}^{\go,a}(s,\ast)^p \, E_{\gb,s}^{\go,a} \Big[ \cZ_{\beta}^{\go,a}  (s,X;s+t,\ast)^p\Big] \,,
\end{equation*}
and the inequality goes in the other direction if $p\in (0,1)$.
Now, recalling \eqref{filtraF} and \eqref{eq:theinlaws}, we have 
\begin{equation*}
	\begin{split}
		\bbE\Big[  \cZ_{\gb}^{\go,a}(s,\ast)^p \, E_{\gb,s}^{\go,a} \Big[ \cZ_{\beta}^{\go,a}  (s,X;s+t,\ast)^p\Big] \mathrel{\Big|} \cF_s \Big]
		&= \cZ_{\gb}^{\go,a}(s,\ast)^p \, 
		E_{\gb,s}^{\go,a}\Big[ \bbE\Big[ \cZ_{\beta}^{\go,a}  (s,X;s+t,\ast)^p \mathrel{\Big|} \cF_s \Big]\Big] \\
		&=  \cZ_{\gb}^{\go,a}(s,\ast)^p \, 
		\bbE\Big[ \cZ_{\beta}^{\go,a}  (t,\ast)^p \Big] \,.
	\end{split}
\end{equation*}
Altogether, we have proven that for $p\geq 1$,
\[
\bbE\Big[ \cZ_{\beta}^{\go,a}(s+t,\ast)^p \mathrel{\Big|} \cF_s \Big] \leq    \cZ_{\gb}^{\go,a}(s,\ast)^p \, 
\bbE\Big[ \cZ_{\beta}^{\go,a}  (t,\ast)^p \Big]  \,,
\]
which yields \eqref{eq:supersub1}.
Repeating the same computation with reversed inequalities  when $p\in (0,1)$, we obtain   \eqref{eq:supersub2}.
\epr

\begin{proof}[Proof of Proposition~\ref{prop:finitemom}]
As a consequence of \eqref{eq:lpconv} in Theorem~\ref{thm:local},  we have
$$\lim_{a\to 0}\bbE\Big[ \calz^{\go,a}_{\beta}(t,\ast)^p\Big]=\bbE\Big[\calz^{\go}_{\beta} (t,\ast)^p\Big].$$ 
Therefore, the previous lemma remains valid for $\calz^{\go}_{\beta}(t,\ast)$ instead of $\calz^{\go,a}_{\beta} (t,\ast)$. 
An immediate consequence is that the limit in~\eqref{eq:ga-p} exists by the continuous version of Fekete's subadditive (or superadditive) lemma. It remains to check that $\gamma_{\beta}(p)$ is finite when \eqref{eq:log-2} is satisfied and $\mu_{1,\infty}(p)<\infty$.
When $p\in [1,1+ \tfrac 2 d)$, $\ga_\beta(p)<\infty$ follows from 
\begin{equation}
\label{Lyapunovfinite}
	\ga_\beta(p)=\lim_{n\to\infty} \frac1n \log \E\Big[\calz^\om_{\beta}(n,\ast)^p\Big]\leq \lim_{n\to\infty} \frac1n \log \E\Big[ \calz^\om_{\beta}(1,\ast)^p\Big]^n = \log \E\Big[\calz^\om_{\beta}(1,\ast)^p\Big],
\end{equation}
and the finiteness of $\E [ \calz^\om_{\beta}(1,\ast)^p ]$, see Theorem \ref{thm:local}.

For $p\in(0,1)$, we prove  $\gamma_{\beta}(p)>-\infty$ and  $\gamma_{\beta}(p)<\infty$ separately.
For the first part, 
 let $\gamma^{<}_{\beta}(p)$ be the moment Lyapunov exponents that one obtains after replacing $\om$ by $\om_{<}$.
Thus, by an obvious comparison and convexity (see Proposition~\ref{prop:qual}, applied to $\lambda_{<}(\dd z) = \ind_{(0,1)}(z)\, \lambda(\dd z) $), it follows that
$$
\gamma_{\beta}(p)\ge \gamma^{<}_{\beta}(p)\ge \frac{\ga_\beta^{<}(1)-(1-\theta)\ga_\beta^{<}(1+\frac1d)}{\theta}>-\infty, 
$$
where $\theta\in(0,1)$ is such that $1=\theta p+(1-\theta)(1+\frac 1d)$.
Note that  $\ga_\beta^{<}(1)=0$ and  that $\ga_\beta^{<}(1+\frac 1d)<\infty$ thanks to~\eqref{Lyapunovfinite}. 

In order to show that $\ga_\beta(p)<\infty$ 
when $p\in (0,\infty)$ and $\mu_{1,\infty}(p)<\infty$, we use Jensen's inequality for $\bbE_{<}$ so that for every $a\in [0,1)$,
\begin{equation*}
 \bbE\Big[ \cZ^{\go,a}_{\beta}(t,\ast)^{p} \Big]\le \bbE \Big[\cZ^{\go_\geq}_{\beta}(t,\ast)^p\Big].
\end{equation*}
 To bound the right-hand side, we use the following simple inequality, which   will be used extensively in the remainder of the paper:
given $\kappa\in(0,1)$ and any countable collection of non-negative numbers $(a_i)_{i\in I}$,
 we have
\begin{equation}\label{eq:subbaditiv}
	\bigg(\sum_{i\in I} a_i\bigg)^{\kappa}\le \sum_{i \in I} a^{\kappa}_i.
\end{equation}
We will refer to~\eqref{eq:subbaditiv}, which can be proved by induction (using the fact that $(a+b)^{\kappa}\le a^{\kappa}+b^{\kappa}$), as the \textit{subadditivity property}; note that it obviously extends to stochastic integrals with respect to discrete measures.

Moreover, note that we have
\beq\label{help3}
\rho(t,x)^p=t^{\nu_p-1}\vartheta(p)\rho (\tfrac tp,x ),
\eeq 
where $\nu_p = 1-\frac d2 (p-1)$ was defined in \eqref{defnup} and $\vartheta(p):= (2\pi)^{\nu_p-1} p^{-\frac d2}$.
Thus, using  \eqref{eq:subbaditiv} for the first step, \eqref{help3} and the fact that $\rho$ is a density for the second step, and Lemma~\ref{lemGamma} for the last step, we derive
\begin{align*}
  \bbE \Big[\cZ^{\go_\geq}_{\beta}(t,\ast)^p\Big]&\le \sum_{k=0}^\infty (\beta^p \mu_{1,\infty}(p))^k \int_{\frakX_k(t)\times(\bbR^d)^k}
   \prod_{i=1}^k  \rho(\Delta t_i, \Delta x_i)^p \,\dd t_i\, \dd x_i \\
   &=\sum_{k=0}^\infty (\beta^p \vartheta(p)\mu_{1,\infty}(p))^k \int_{\frakX_k(t) }
   \prod_{i=1}^k  (\Delta t_i)^{\nu_p-1}\,\dd t_i   = \sum_{k=0}^\infty  \frac{(\beta^p\vartheta(p)\mu_{1,\infty}(p)\Gamma(\nu_p)t^{\nu_p})^k}{\Gamma(\nu_pk+1)} \,.
\end{align*}
Using Stirling's formula for the gamma function (or  the estimate \eqref{eq:6} below), one can check that the right-hand side 
grows exponentially in $t$: as a consequence, we have $\gamma_{\beta}(p)<\infty$. This ends the proof of part \textit{(i)} of the proposition.  

\medskip
For part \textit{(ii)}, we are going to prove the upper and lower bounds in \eqref{eq:free-end} separately.
Using Jensen's inequality for the first inequality and the translation invariance \eqref{eq:theinlaws} for the last identity, we have, for $p>1$,
\begin{equation}\label{supersub}
 \bbE \Big[\cZ^{\go}_{\beta}(t,\ast)^p\Big]\le 
 \int_{\bbR^d}\rho(t,x)  \bbE\Big[\Big(\rho(t,x)^{-1}\cZ^{\go}_{\beta}(t,x)\Big)^p\Big]\,\dd x=\bbE\Big[\Big(\rho(t,0)^{-1}\cZ^{\go}_{\beta}(t,0)\Big)^p\Big].
\end{equation}
Since $\rho(t,0)^{-1}$ is of order $t^{d/2}$, after taking logarithm, dividing by $t$ and taking the limit as $t\to\infty$, we obtain 
\begin{equation}\label{firhal}
  \liminf_{t\to \infty}\frac{1}{t}\log\bbE\Big[\cZ^{\go}_{\beta}(t,0)^p\Big]\ge \gamma_{\beta}(p).
\end{equation}
The same proof yields   for $p\in (0,1)$ that
\begin{equation}\label{firhal2}
 \limsup_{t\to \infty}\frac{1}{t}\log\bbE\Big[\cZ^{\go}_{\beta}(t,0)^p\Big]\le \gamma_{\beta}(p).
\end{equation}
Let us now prove the complementary bound, first in the case $p>1$. We introduce the probability measure  $\bar P_{\gb,t}^{\go}$ with Lebesgue density $\cZ^{\go}_{\beta}(t,x)\rho(1,x)/\int_{\R^d} \cZ^{\go}_{\beta}(t,x)\rho(1,x)\,\dd x$.
Then we have
\begin{align*}
\cZ^{\go}_{\beta}(t+1,0)&=\int_{\R^d}\cZ^{\go}_{\beta}(t,x)\rho(1,x)\cZ^{\go}_{\beta}(t,x;t+1,0) \rho(1,x)^{-1}\,\dd x\\
&= \Bigg(\int_{\bbR^d}\cZ^{\go}_{\beta}(t,x)\rho(1,x)\,\dd x \Bigg) \bar E_{\gb,t}^{\go}\Big[ \cZ^{\go}_{\beta}(t,X;t+1,0) \rho(1,X)^{-1}\Big].
\end{align*}
Proceeding as in the proof of Lemma \ref{lem:submult} and using the translation invariance~\eqref{eq:theinlaws}, we obtain 
for $p>1$
\begin{equation}\label{indasamemanner}
 \bbE\Big[\cZ^{\go}_{\beta}(t+1,0)^p\Big]\le \bbE\bigg[\bigg(
 \int_{\bbR^d}\cZ^{\go}_{\beta}(t,x)\rho(1,x)\,\dd x\bigg)^p \bigg] \bbE \Big[ \cZ^{\go}_{\beta}(1,0)^p \rho(1,0)^{-p} \Big]\,.
\end{equation}
In particular, we have
\begin{equation*}
 \bbE\Big[\cZ^{\go}_{\beta}(t+1,0)^p\Big]\le (2\pi)^{-\frac{pd}{2}}
 \bbE\Big[\cZ^{\go}_{\beta}(t,\ast)^p \Big] \bbE \Big[ \cZ^{\go}_{\beta}(1,0)^p \rho(1,0)^{-p} \Big],
\end{equation*}
which implies, for $p>1$,
\begin{equation*}
 \limsup_{t\to \infty}\frac{1}{t}\bbE\Big[\cZ^{\go}_{\beta}(t+1,0)^p\Big]\le \gamma_\beta(p) \,.
\end{equation*}
When $p\in(0,1)$, repeating the proof of \eqref{indasamemanner}, but using concavity instead of convexity, we obtain
\begin{equation}\label{eq:p}
   \bbE\Big[ \cZ^{\go}_{\beta}(t+1,0)^p \Big]\ge  \bbE\bigg[\bigg(\int_{\bbR^d}\cZ^{\go}_{\beta}(t,x)\rho(1,x)\,\dd x\bigg)^p\bigg]\bbE \Big[ \cZ^{\go}_{\beta}(1,0)^p\rho(1,0)^{-p}\Big].
\end{equation}
Now setting $\cC_z:=z+[0,1)^d$, we have
\begin{equation}\label{forsmallp}
 \cZ^{\go}_{\beta}(t,\ast)= \sum_{z\in \bbZ^d} \int_{\cC_z}\cZ^{\go}_{\beta}(t,x)\,\dd x \le \sum_{z\in \bbZ^d} \bigg(\max_{x\in \cC_z} \rho(t,x) \bigg) Y_z(t) 
\end{equation}
with 
$ Y_z(t):=\int_{\cC_z}\cZ^{\go}_{\beta}(t,x) \rho(t,x)^{-1}\,\dd x.$

Note that the variables $Y_z$ are identically distributed by \eqref{eq:theinlaws}, which together with \eqref{eq:subbaditiv} implies
\begin{equation}\label{asdf}
  \bbE\Big[ \cZ^{\go}_{\beta}(t,\ast)^p\Big] \le \bigg( \sum_{z\in \bbZ^d} \max_{x\in \cC_z} \rho(t,x)^{p} \bigg)\bbE[Y_0(t)^p]
  \le C t^{\frac d2 {(1-p)}}\bbE[Y_0(t)^p] 
\end{equation}
for $t>1$. Moreover, also for $t>1$,
\begin{equation}
\label{asdf-bis}
 \int_{\bbR^d}\cZ^{\go}_{\beta}(t,x)\rho(1,x)\,\dd x\ge \bigg(\min_{x\in \cC_0}\rho(1,x)\rho(t,x)\bigg)Y_0(t)\ge (2\pi)^{-d} t^{-\frac{d}{2}}Y_0(t) \,.
\end{equation}
Therefore, combining \eqref{eq:p} with \eqref{asdf} and~\eqref{asdf-bis}, we obtain that 
\begin{equation*}
 \bbE\Big[\cZ^{\go}_{\beta}(t+1,0)^p\Big]\ge C' {t^{-\frac d2}} \bbE\Big[ \cZ^{\go}_{\beta}(t,\ast)^p\Big]\,.
\end{equation*}
This allows us to conclude that, for $p\in(0,1)$,
\begin{equation*}
  \liminf_{t\to \infty}\frac{1}{t}\log\bbE\Big[\cZ^{\go}_{\beta}(t+1,0)^p\Big]\ge \gamma_{\beta}(p).\qedhere
  \end{equation*}
\end{proof}

\subsection{Monotonicity and convexity properties}\label{sec:quall}

To show the monotonicity in $\beta$ of the Lyapunov exponents, we prove a more general result linking monotonicity and convexity. Its proof is inspired by an analogous result proved in the discrete setup \cite[Lemma 3.3]{CY06}. Recall the definition of the renormalized partition function \eqref{barparti}.
\begin{lemma}\label{lem:mon}
Assume that $\mu_{1,\infty}(q)<\infty$  for some  $q\in 
[1,1+\frac{2}{d})$
and let $\phi: [0,\infty) \to \bbR$ be a convex function that satisfies 
\[
	\sup_{u\ge 1}\   |\phi(u)| u^{-q} <\infty \,.
\]
Then $\beta \mapsto \bbE [ \varphi (\bar \cZ^{\go,a}_{\beta}(t,0) ) ]$ is a non-decreasing function, for any $a>0$.
\end{lemma}

\bpr[Proof of Proposition~\ref{prop:qual}]
As a consequence of Lemma~\ref{lem:mon},  if 
$\mu_{1,\infty}(p)<\infty$ for some $p\in [1,1+\frac{2}{d})$,
then
$\beta \mapsto  \bbE[ \bar \cZ^{\go,a}_{\beta}(t,0)^p]$
is non-decreasing if $p>1$ and non-increasing if $p\in (0,1)$  (apply Lemma~\ref{lem:mon} to the convex function $x\mapsto - x^p$). Letting $a\to0$, we obtain the same results for $ \bbE[ \bar \cZ^{\go}_{\beta}(t,0)^p]$ because of \eqref{eq:lpconv}. 
Since we have 
$$
\bar \gamma_{\beta}(p)=\lim_{t\to \infty} \frac{1}{t}\log \bbE[ \bar \cZ^{\go}_{\beta}(t,0)^p]\,,
$$
this completes the proof of \textit{(ii)} in Proposition~\ref{prop:qual}.
Both \textit{(i)} and \textit{(iii)} are standard: The convexity of $p\mapsto \ga_\beta(p)$ follows from that 
of $p\mapsto \log \E [ X^p ]$, valid for an arbitrary non-negative random variable~$X$ by  H\"older's inequality; 
\textit{(iii)} is a direct consequence of convexity since for $1<p<p'$  we have
$\bar \gamma_{\beta}(p) \le \frac{p-1}{p'-1} \bar \gamma_{\beta}(p')$, recalling that $\bar \gamma_{\beta}(1)=0$ by definition.
\epr

\begin{proof}[Proof of Lemma~\ref{lem:mon}]
Since $\varphi$ is convex, there exists $a$ and $b$ such that 
$\varphi(u)+au+b\ge 0$ for every $u\ge 0$.
Hence,  replacing $\varphi(u)$ by $\varphi(u)+au+b$, we can assume that $\varphi$ is non-negative, since $\bbE[ \bar \cZ^{\go,a}_{\beta}(t,0) ]=\rho(t,0)$ does not depend on $\beta$. 
We may also assume without loss of generality that~$\varphi$ is differentiable on $[0,\infty)$ and that $\varphi'$ is bounded. 
Indeed, if this is not the case, we can find a sequence $\varphi_n\uparrow \varphi$ that has these properties and then use monotone convergence.

Recalling the notation \eqref{def:posandtv} and setting $\bar \kappa_a=\kappa_a+\mu$, we have, similarly to \eqref{eq:help},
\[
\bar \cZ^{\go,a}_{\beta}(t,0) =e^{-\beta\bar \kappa_a t}\sum_{k=0}^{\infty}\beta^k\int_{\frakX_k(t)\times(\bbR^d)^k} \rho_{t,0}(\bt,\bx) \prod_{i=1}^k\xi^{a,+}_{\go}(\dd t_i,\dd x_i).
\]
Taking  the derivative with respect to $\beta$, we obtain, after an index shift for the first term,
	\begin{equation}\label{twoterms}\mathtoolsset{multlined-width=0.9\displaywidth}\begin{multlined}
		\partial_\beta \bar \cZ^{\go,a}_{\beta}(t,0)= e^{-\beta \bar \kappa_a t}\sum_{k=0}^{\infty}\beta^k \bigg(   (k+1)\int_{\frakX_{k+1}(t)\times(\bbR^d)^{k+1}} \rho_{t,0}(\bt,\bx) \prod_{i=1}^{k+1}\xi^{a,+}_{\go}(\dd t_i,\dd x_i)\\
		-\bar \kappa_a t\int_{\frakX_k(t)\times(\bbR^d)^k} \rho_{t,0}(\bt,\bx) \prod_{i=1}^k\xi^{a,+}_{\go}(\dd t_i,\dd x_i)\bigg) \, .
\end{multlined}	\end{equation}
The first term in the sum can be viewed as the sum of $k+1$ integrals 
where the variables of integration are $(t_1,\dots,t_i,s,t_{i+1},\dots,t_k)\in \frakX_{k+1}(t)$ and $(x_1,\dots,x_i,y,x_{i+1},\dots,x_k)$ for $i= 0,\dots,k$.
Recombining the terms, we arrive at the identity (recall the convention $t_0:=0$ and $t_{k+1}:=t$)
	\begin{multline*}
		(k+1)\int_{\frakX_{k+1}(t)\times (\bbR^d)^{k+1}}  \rho_{t,0}(\bt,\bx) \prod_{i=1}^k\xi^{a,+}_{\go}(\dd t_i,\dd x_i) \\
	 = \int_{\frakX_{k}(t)\times (\bbR^d)^{k}} \Bigg(\int_{(0,t)\times \bbR^d}\rho_{t,0}(\bt,\bx) \rho(s,y\,|\,\bt,\bx)\,\xi^{a,+}_{\go}(\dd s,\dd y)\Bigg) \prod_{i=1}^{k}\xi^{a,+}_{\go}(\dd t_i,\dd x_i),
	\end{multline*}
where 
	\begin{equation*}
		\rho(s,y\,|\,\bt,\bx):=
			\frac{\rho(s-t_{i-1},y-x_{i-1})\rho(t_i-s,x_i-y)}{\rho(t_{i}-t_{i-1},x_i-x_{i-1})}   \quad  \text{ if } s \in (t_{i-1},t_i) 
	\end{equation*}
and $\rho(t_i,y\,|\,\bt,\bx):=0$ for all $i=1,\dots,k$.
	As $\int_{\bbR^d}\rho(s,y\,|\,\bt,\bx)\,\dd y=1$, the second term in \eqref{twoterms} corresponds to a centering of the noise $\xi^{a,+}_{\go}(\dd s,\dd y)$, and we have
	\begin{equation}\label{nicessum} 
	\begin{split}
		&\partial_\beta \bar \cZ^{\go,a}_{\beta}(t,0)\\
		&= e^{-\beta \bar \kappa_a t}\sum_{k=0}^{\infty}\beta^k \int_{\frakX_k(t)\times (\bbR^d)^{k}}\Bigg(\int_{(0,t)\times \bbR^d}\rho_{t,0}(\bt,\bx)  \rho(s,y\,|\,\bt,\bx) \, \bar \xi^{a}_{\go}(\dd s,\dd y)\Bigg) \prod_{i=1}^{k}\xi^{a,+}_{\go}(\dd t_i,\dd x_i).  
		\end{split}
	\end{equation}
Note that we have 
\begin{multline*}
  |\partial_\beta \bar \cZ^{\go,a}_{\beta}(t,0)|\\
		\le e^{-\beta \bar \kappa_a t}\sum_{k=0}^{\infty}\beta^k \int_{\frakX_k(t)\times (\bbR^d)^{k}}\Bigg(\int_{(0,t)\times \bbR^d}\rho_{t,0}(\bt,\bx)  \rho(s,y\,|\,\bt,\bx) \, |\bar \xi^{a}_{\go}|(\dd s,\dd y)\Bigg) \prod_{i=1}^{k}\xi^{a,+}_{\go}(\dd t_i,\dd x_i) \,,
\end{multline*}
where  $\lvert\bar \xi^{a}_{\go}\rvert= \bar \xi^{a}_{\go}+2\bar \kappa_a  \cL$. 
Since we have reduced to the case where $\varphi'$ is bounded,
the expression above implies (cf.\ \cite[Prop.~2.5]{BL20_cont}) that for any $\beta_0>0$, 
\beq\label{eq:int}
 \bbE\left[  \sup_{\beta\in[0,\beta_0]} \Big| \partial_\beta \bar \cZ^{\go,a}_{\beta}(t,0)\phi' (\bar \cZ^{\go,a}_{\beta}(t,0)) \Big| \right]
 <\infty.
 \eeq 
 This allows to interchange derivative and expectation and by \eqref{nicessum}, we obtain that
	\begin{equation}
	\label{compopo}
	\begin{split}
		\partial_{\beta}\bbE\Big[ \varphi\big(\bar \cZ^{\go,a}_{\beta}(t,0)\big) \Big]&= \bbE\Big[  \partial_\beta \bar \cZ^{\go,a}_{\beta}(t,0) \, \phi' \big(\bar \cZ^{\go,a}_{\beta}(t,0) \big)\Big]\\
		&= e^{-\beta\bar\kappa_a t}\sum_{k=0}^{\infty}\beta^k \bbE \Bigg[ \int_{\frakX_k(t)\times (\bbR^d)^{k}} U^\go_k(\bt,\bx)\rho_{t,0}(\bt,\bx)\prod_{i=1}^k\xi^{a,+}_{\go}(\dd t_i,\dd x_i) \Bigg]\,,
		\end{split}
	\end{equation}
where we have set
\begin{equation*}
 U_k^\go(\bt,\bx):= \Bigg(\int_{[(0,t)\setminus\{t_i\}^k_{i=1}]\times \bbR^d} \rho(s,y\,|\,\bt,\bx)\,\bar \xi^{a}_{\go}(\dd s,\dd y)\Bigg)  \varphi'\big(\bar \cZ^{\go,a}_{\beta}(t,0)\big).
\end{equation*}
Given $(\bt,\bx,\bz)\in \frakX_k(t)\times (\bbR^d)^{k}\times (0,\infty)^k$, we let $\go\cup (\bt,\bx,\bz)$ be the point process obtained by adding the~$k$ points  $(t_i,x_i,z_i)_{1\leq i \leq k}$ to $\go$.
By Mecke's multivariate equation (see \cite[Thm.~4.4]{PoiBook}),
we have 
\beq\label{supermecke}\begin{split}
& \bbE \Bigg[ \int_{\frakX_k(t)\times (\bbR^d)^{k}} U_k^\go(\bt,\bx) \rho_{t,0}(\bt,\bx) \prod_{i=1}^k  \xi^{a,+}_{\go}(\dd t_i,\dd x_i) \Bigg]\\
&\qquad=
 \bbE \Bigg[ \int_{\frakX_k(t)\times (\bbR^d)^{k}\times (0,\infty)^k}  U_k^\go(\bt,\bx) \rho_{t,0}(\bt,\bx)\prod_{i=1}^k z_i\ind_{[a,\infty)}(z_i)\, \delta_{\go}(\dd t_i,\dd x_i,\dd z_i) \Bigg]\\
&\qquad =
 \int_{\frakX_k(t)\times (\bbR^d)^{k}\times (0,\infty)^k}  \bbE\Big[ U_k^{\go\cup (\bt,\bx,\bz)}(\bt,\bx)\Big] \rho_{t,0}(\bt,\bx) \prod_{i=1}^k z_i\ind_{[a,\infty)}(z_i)\,\dd t_i\, \dd x_i\,\gl( \dd z_i).
\end{split}\eeq
If $k$, $\bt$, $\bx$ and $\bz$ are fixed,
the  functionals
	\[
	\go\mapsto \phi' \big( \bar \cZ^{\go\cup(\bt,\bx,\bz),a}_{\beta}(t,0) \big) \quad \text{ and } \quad \go \mapsto \int_{[(0,t)\setminus\{t_i\}^k_{i=1}]\times \bbR^d}\rho(s,y\,|\,\bt,\bx)\, \bar \xi^{a}_{\go}(\dd s,\dd y)
	\]
are non-decreasing for the inclusion order. Thus we can apply the FKG inequality for Poisson point processes (see~\cite[Lemma~2.1]{Jan84}) and obtain that
	\begin{equation*}
		\bbE\Big[  U^{\go\cup (\bt,\bx,\bz)}_k(\bt,\bx)\Big] \ge 
		\bbE\Bigg[ \int_{[(0,t)\setminus\{t_i\}^k_{i=1}]\times\bbR^d}\rho(s,y\,|\,\bt,\bx)\, \bar\xi^{a}_{\go}(\dd s,\dd y)\Bigg]   \bbE\Big[  \phi' (\bar \cZ^{\go\cup (\bt,\bx,\bz),a}_{\beta}(t,0))\Big]=0.
	\end{equation*}
Combining this with \eqref{compopo} and \eqref{supermecke}, we conclude that $\partial_{\beta}\bbE\big[ \varphi\big(\bar \cZ^{\go,a}_{\beta}(t,0)\big) \big] \geq 0$.
\end{proof}

\section{Moments  of order $p>1$}\label{sec:upper}

The goal of this section is to formulate and prove Propositions~\ref{prop:localbisdone} and \ref{prop:localbis} below, 
which form the core of all moment upper bounds for $p>1$ in this paper.

\subsection{The statements}

Recall that $\nu_p:=1-\frac d2 (p-1)$ and let $\bar \calz^{\om,0}_{\beta}(t,x):= \bar \calz^{\om}_{\beta}(t,x)$. 

\subsubsection{The case of dimension $d=1$}
We start with $d=1$, where the statement is  easier to state (and easier to prove).

\begin{proposition}\label{prop:localbisdone}
	Assume that $d=1$ and that \eqref{eq:log-2} holds. 
	\benu
	\item[(i)] If  $\mu_{1,\infty}(2)<\infty$, then for any $a\in[0,1)$,
	\begin{equation}\label{explicitesecondmoment}\begin{split}
	\E\Big[\rho(t,x)^{-2}\bar \calz^{\om,a}_{\beta}(t,x)^2\Big]
	&=   \sum_{k=0}^{\infty}\Bigg(\frac{\beta^2 \mu_{a,\infty}(2) \sqrt{t}}{2}\Bigg)^k \frac{\sqrt{\pi}}{\gG((k+1)/2)}\\
	&=1+\beta^2\mu_{a,\infty}(2)\sqrt{\pi t} \exp\Big(\tfrac14 \beta^4\mu_{a,\infty}(2)^2 t\Big)\Phi\Big(\beta^2\mu_{a,\infty}(2)\sqrt{\tfrac t2}\Big),
	\end{split}
	\end{equation}
where $\Phi$ is the standard normal distribution function.
 \item[(ii)] There exists a constant $C\in(0,\infty)$ such that if 
	$\mu_{1,\infty}(p)<\infty$ for some $p\in (1,2)$, then for any value of  $\eta\in(0,1]$, $a\in[0,1)$ and $t,\beta>0$, 
	\begin{equation}\label{mobdd1}
	\E\Big[\rho(t,x)^{-p}\bar \calz^{\om,a}_{\beta}(t,x) ^p\Big]^{\frac1p}
	\le 
	    \sqrt{2} \Gamma(\nu_p)^{\frac1p}  \sum_{k_1,k_2=0}^\infty \bigg(\frac{C\beta}{p-1}\bigg)^{k_1+k_2}  \Big( \mu_{0,\eta}(2)t^{\frac 12} \Big)^{\frac{k_1}2} 
	   \frac{ (\mu_{\eta,\infty}(p) t^{\nu_p} )^{\frac{k_2}p} }{\Gamma(\nu_p(k_2+1))^{\frac1p}}.
	 \end{equation}
 \item[(iii)] There exists a constant $C\in(0,\infty)$ such that if 
 $\mu_{1,\infty}(p)<\infty$ for some $p\in (2,3)$, then for any value of $a\in[0,1)$ and $t,\beta>0$, 
 	\begin{equation}\label{mobdd2}
 	\begin{split}
 	\E\Big[\rho(t&,x)^{-p}\bar \calz^{\om,a}_{\beta}(t,x) ^p\Big]^{\frac1p} \\
 	& \le \pi^{\frac14} \Gamma(\nu_p)^{\frac1p}   \sum_{k_2=0}^{\infty} \frac{ \big( C \beta  \mu_{0,\infty}(2)^{\frac12} \Gamma(\nu_p)^{\frac1p} t^{\frac14} \big)^{k_2}  }{  \Gamma(\frac{k_2+1}{2})^{\frac12}} \Bigg( \sum_{\ell =0}^{\infty}   \frac{ \big( C'\beta \mu_{0,\infty}(p)^{\frac1p} t^{\frac{\nu_p}{p}} \Gamma(\nu_p)^{\frac1p} \big)^{\ell} }{ \Ga(  \nu_p (\ell+1) )^{\frac1p}}  \Bigg)^{k_2+1} \,.
 		\end{split}
 \end{equation}
 \eenu
\end{proposition}

\begin{rem}\label{rem:secmom}
The formula~\eqref{explicitesecondmoment} is rather straightforward to prove. Starting from~\eqref{eq:chaos} and using that $\bar \xi_{\go}^a$ is centered, the second moment is given by
\[
\begin{split}
\E\Big[\rho(t,x)^{-2}\bar \calz^{\om,a}_{\beta}(t,x)^2\Big]
= 1+\sum_{k=1}^{\infty} (\gb^{2} \mu_{a,\infty}(2))^k \int_{ \mathfrak{X}_k(t)\times (\bbR^d)^k} \rho(t,x)^{-2} \prod_{i=1}^{k+1} \rho( \Delta t_i,  \Delta x_i) ^2\, \dd t_i \, \dd x_i \,,
\end{split}
\]
 with the notation \eqref{deltadeff} for $\Delta t_i$ and $\Delta x_i$.
Integrating over $x_1, \ldots, x_k \in \bbR^d$, we get
\[
\begin{split}
\E\Big[\rho(t,x)^{-2}\bar \calz^{\om,a}_{\beta}(t,x)^2\Big]
= 1+\sum_{k=1}^{\infty} (\gb^{2} \mu_{a,\infty}(2))^k (2\sqrt{\pi t}) \int_{\mathfrak{X}_k(t)} \prod_{i=1}^{k+1} \frac{1}{2\sqrt{\pi \Delta t_i}}\, \dd t_i \,,
\end{split}
\]
which gives the first equality in \eqref{explicitesecondmoment} by Lemma~\ref{lemGamma}. The second equality follows from the formula $\sum_{k=0}^\infty x^k/\Gamma((k+1)/2)=\pi^{-1/2}+2xe^{x^2}\Phi(\sqrt{2} x)$. Let us also mention that this formula coincides with the one we obtain if $\bar \calz^{\om,a}_\beta(t,x)$ in \eqref{explicitesecondmoment} is replaced by the solution to the SHE \eqref{eq:SHELN} with a space-time \emph{Gaussian} noise with variance $\mu_{a,\infty}(2)$ (and $u_0=\delta_0$); cf.\ \cite[Eq.\ (2.31)]{Chen15}.
\end{rem}

\subsubsection{The case of dimension $d\geq 2$}
For the statement with $d\ge 2$, we need to introduce a few auxiliary quantities.
With the usual convention  $t_0:=0$ and $t_{k+1}:=t$ and $\Delta t_i:= t_i-t_{i-1}$, we define $\Lambda(0,t,p):=1$ and 
\begin{equation}\label{eq:definingl}
 \Lambda(k,t,p):=t^{1-\nu_p}\int_{\mathfrak{X}_{k}(t)} (\Delta t_{k+1})^{\nu_p-1}\prod_{i=1}^{k} G_{p}(\Delta t_i)  \, \dd t_i
\end{equation}
for $k\ge 1$, where
\begin{equation}\label{eq:Gp}
G_{p}(s):=\begin{cases} s^{\frac13\nu_p-1} \quad & \text{ if } t\le 1 ,\\
s^{\nu_p-1} &\text{ if } t\ge 1. 
\end{cases}
\end{equation}
Also, we let
\begin{equation}  \label{eq:defzezeta}
 \zeta_1(\eta,p,t):= \Big(8\,  \mu^{\log}_{0,\eta}\Big(1+\tfrac{2}{d}\Big)(1+\log_+ t) \Big)^{\frac{p}{1+ 2/d}},\qquad
\zeta_2(\eta,p):= \mu_{0,\eta}^{\log}\Big(1+\tfrac 2d\Big)+ \mu_{\eta,\infty}(p),
\end{equation}
where we have set $\log_+ t:=\log(t\vee 1)$ and
\[
\mu_{0,\eta}^{\log}\Big(1+\tfrac2d\Big):=\int_{(0,\eta)} z^{1+\frac 2d} (3\lvert \log z\rvert+1  )\la(\dd z)
\]
Note that both $\zeta_1(\eta,p,t)$ and $\zeta_2(\eta,p)$ are finite if \eqref{eq:log-2} holds and $\mu_{1,\infty}(p)<\infty$.

\begin{proposition}\label{prop:localbis}
	Assume that $d\ge 2$ and that \eqref{eq:log-2} holds.
	There exists a  constant $C\in(0,\infty)$  that only depends on $d$ such that if 
	$\mu_{1,\infty}(p)<\infty$ for  $p\in(1,1+\frac2 d )$, then for any value of  $\eta\in(0,1]$, $a\in[0,1)$ and $t,\beta>0$,   we have
\begin{equation}\label{compinek}
	\E\Big[\rho(t,x)^{-p}\bar \calz^{\om,a}_{\beta}(t,x)^p\Big]^{\frac1p}
	\le 
	   \sum_{k_1,k_2=0}^\infty  \bigg(\frac{C\beta}{p-1}\bigg)^{k_1+k_2}\zeta_1(\eta,p,t)^{\frac{k_1}p}  \zeta_2(\eta,p)^{\frac{k_2}p}  \Lambda(k_2,t,p)^{\frac1p}.
	 \end{equation}
\end{proposition}

\begin{rem}
 In our proofs of Propositions \ref{prop:localbisdone} and \ref{prop:localbis}, we keep track of the dependence in $p$. The reason for this is that in the proof of Theorem  \ref{thm:thinup}---more precisely in the proof of \eqref{eq:1+2d}---we need to apply \eqref{compinek} for $p$ that depends of $\beta$. The only important point  we need to make sure is that our estimates remain uniform for $p$ in an interval around $1+\frac{2}{d}$.
\end{rem}

\subsubsection{Finiteness of the moments}
An almost immediate consequence of the proposition is the following uniform moment bound.
\begin{cor}\label{labornedesmoments}
If \eqref{eq:log-2} holds and $\mu_{1,\infty}(p)<\infty$ for some $p\in(1,1+\frac{2}{d})$, then for any $\beta$ and $T$ there exists $C(\beta,p,T)$ such that
\begin{equation*}
			\sup_{a\in[0,1)}\sup_{(t,x)\in(0,T]\times\R^d}\E\Big[\Big(\rho(t,x)^{-1}\bar \calz^{\om,a}_{\beta}(t,x) \Big)^p\Big]^{\frac1p}
<C(\beta,p,T).
	\end{equation*} 
\end{cor}


\begin{proof}
First let us note that the bounds in Propositions \ref{prop:localbisdone} and \ref{prop:localbis} already are uniform in $a$  and~$x$, so we only need to check uniformity in $t$.

\smallskip
Let us start with the case $d=1$. The statement for $p=2$ is obvious from \eqref{explicitesecondmoment}. If $p\in(1,2)$, by reorganizing \eqref{mobdd1}, 
we obtain that for every $t\in[0,T]$,
\begin{equation*}
\E\Big[\rho(t,x)^{-p}\bar \calz^{\om,a}_{\beta}(t,x)^p\Big]^{\frac1p}
\le  C'\Bigg(\sum_{k_1= 0}^\infty \Big(C_p\beta \mu_{0,\eta}(2)^{\frac12} T^{\frac{1}{4}}\Big)^{k_1}\Bigg)\Bigg(\sum_{k_2= 0}^\infty\frac{\Big(C_p\beta\mu_{\eta,\infty}(p)^{\frac1p} T^{\frac{\nu_p}p}\Big)^{k_2}\Gamma(\nu_p)^{\frac1p}}{\Gamma(\nu_p(k_2+1))^{\frac1p}}\Bigg),
\end{equation*}
where $C_p=C/(p-1)$.
The first sum is finite if one chooses $\eta$ such that  $C\beta \mu_{0,\eta}(2)^{ 1/ 2 } T^{1/4}<\frac12$ and the second one is always finite since $\Gamma(\nu_p(k_2+1))$ grows super-exponentially.  If $p\in(2,3)$, 
we have from~\eqref{mobdd2} that for every $t\in [0,T]$,
\begin{equation*}
\E\Big[\rho(t,x)^{-p}\bar \calz^{\om,a}_{\beta}(t,x)^p\Big]^{\frac1p}
\leq C_p \sum_{k_2=0}^{\infty} \frac{\big(  C_p \beta  T^{\frac14}\big)^{k_2}}{\Gamma(\frac{k_2+1}{2})^{\frac12}} \Bigg(  \sum_{\ell=0}^{\infty} \frac{ \big(C'_p \beta T^{\frac{\nu_p}{p}}\big)^{\ell} }{\Gamma(\nu_p (\ell+1)) }\Bigg)^{k_2+1}\,.
\end{equation*}
Since $\Gamma(\nu_p(\ell+1))$ and $\Gamma(\frac{k_2+1}{2})$  both grow super-exponentially, the two sums are finite.

\smallskip

When $d\ge 2$, since $\zeta_1$ is monotone in $t$, we have, in the same manner, 
\begin{equation}\label{eq:help2}
 \E\Bigl[\rho(t,x)^{-p}\bar \calz^{\om,a}_{\beta}(t,x)^p\Bigr]^{\frac1p}
\le \Biggl(\sum_{k_1= 0}^\infty \Bigl(C_p\beta \zeta_1(\eta,p,T)^{\frac1p}\Bigr)^{k_1}\Biggr) \Biggl(\sum_{k_2= 0}^\infty\Bigl(C_p\beta\zeta_2(\eta,p)^{\frac1p}\Bigr)^{k_2}\Lambda(k_2,t,p)^{\frac1p}\Biggr)
\end{equation}
for $t\le T$.
The first sum is finite provided that $\eta$ is chosen sufficiently small.
Considering the second term, we have
\begin{align*}
		 \Lambda(k_2,t,p)  &\le t^{1-\nu_p} (t\vee 1)^{\frac23\nu_pk_2}\int_{\mathfrak{X}_{k_2}(t)} (\Delta t_{k_2+1})^{\nu_p-1}\prod_{i=1}^{k_2}(\Delta t_i)^{\frac13\nu_p-1} \, \dd t_i \\
		 		   &=  (t\vee 1)^{\frac23\nu_pk_2} t^{\frac{1}{3}\nu_pk_2}  \frac{\gG(\frac13\nu_p)^{k_2} \gG(\nu_p)}{ \gG
			(\frac13\nu_p(k_2+3))}\,,
	\end{align*}
 where we   used Lemma~\ref{lemGamma} for the last identity. 
As a result,  assuming that $T\ge 1$, we obtain that for every $t\in[0,T]$,
\begin{equation*}
 \Big(C_p \beta \zeta_2(\eta,p)^{\frac1p}\Big)^{k_2} \Lambda(k_2,t,p)^{\frac1p}\le   \Big(C_p\beta \zeta_2(\eta,p)^{\frac1p} \gG( {\textstyle\frac13} \nu_p)^{\frac1p}  T^{\frac{\nu_p}p}\Big)^{k_2} \frac{ \gG(\nu_p)^{\frac1p}}{ \gG
 	(\frac13\nu_p(k_2+3))^{\frac1p}},
\end{equation*}
and since $\gG(\nu_p[(k_2/3)+1])$ grows super-exponentially, the sum over $k_2$ in \eqref{eq:help2} is finite.
\end{proof}

\subsection{Bounding moments in the chaos expansion: the first term}\label{sec:k1}

From now on, we focus on the case $x=0$, which yields no loss of generality by \eqref{eq:theinlaws}. 
Also,  in both Propositions~\ref{prop:localbisdone} and~\ref{prop:localbis} above, the case $a=0$ can be deduced from the case $a>0$ using Fatou's lemma: in the following, we can always assume that $a>0$.
Starting from the chaos decomposition \eqref{eq:chaos}, we can use Minkowski's inequality to get
\begin{equation}\label{eq:triangle}
	\E\Big[ \rho(t,0)^{-p} \bar \calz^{\om,a}_{\beta}(t,0)^p\Big]^{\frac1p}
	\le   \sum_{k=0}^{\infty} \gb^k \bbE\left[  | W_{a,k}(t) |^p \right]^{\frac1p} \,,
\end{equation}
where 
\begin{equation}\label{def:wtk}
  W_{a,0}(t):=1,\qquad W_{a,k}(t):=\int_{ \frakX_k(t) \times (\bbR^d)^k}  \frac{ \rho_{t,0}( \bt , \bx)}{ \rho(t,0)}   \prod_{i=1}^k \bar \xi^a_{\go} (\dd t_i , \dd x_i )\quad  \text{ for } k\geq1.
\end{equation}

The estimates for  $\bbE [  | W_{a,k}(t) |^p ]$ are intricate, so let us spend some time on the case $k=1$ to illustrate the intuition behind our proof. 
In what follows, we write 
\begin{equation}\label{indexnotation}
\calx^{(k)}_t:=\frakX_k(t) \times (0,\infty)^k,\quad X^{(k)}_t:=\frakX_k(t)\times(\R^d)^k,\quad \X^{(k)}_t:=\frakX_k(t)\times(\R^d)^k\times(0,\infty)^k.
\end{equation}
We drop the superscript $k$ when $k=1$.
For simplicity, let us consider  the expansion of the free-end partition function.
Because the integrals in \eqref{eq:triangle} are martingales in $t$ for the filtration \eqref{filtraF}, we can apply the Burkholder--Davis--Gundy (BDG) inequality and obtain 
\begin{equation}\label{eq:BDG-2}
 \bbE\Biggl[ \Biggl(\int_{ (0,t) \times \bbR^d}   \rho( s , x) \, \bar \xi^a_{\go} (\dd s , \dd x )\Biggr)^p\Biggr]\\
 \le C_p\,
 \bbE\Biggl[\Biggl(\int_{ (0,t) \times \bbR^d\times [a,\infty)}  (\rho( s , x) z)^2\,
 \delta_{\go}(\dd s, \dd x, \dd z)\Biggr)^{\frac p2}\Biggr].
\end{equation}
Since we are tracking the dependence in $p$,
it is worth noting that it is possible to take $C_p=(4p)^p\le 64$ if $p\leq 2$; see \cite[Chapter VII, Theorem 92]{Del82}.
In order to bound the right-hand side uniformly in $a$, we replace $[a,\infty)$ by $(0,\infty)$. 
Let us further restrict ourselves to the case $d\ge 3$ for simplicity (note that in particular $p<2$). 
By Jensen's inequality and the  subadditivity property~\eqref{eq:subbaditiv}, we have,
for $\theta \in [p,2]$,
\begin{equation}
\label{eq:jensubad}
\begin{split}
 \bbE\Bigg[\bigg(\int_{\X_t} \rho( s , x)^2  z^2\,
 \delta_{\go}(\dd s, \dd x, \dd z)\bigg)^{\frac p2}\Bigg] &\le \bbE\Bigg[\bigg(\int_{ \X_t}  \rho( s , x)^2  z^2\,
 \delta_{\go}(\dd s, \dd x, \dd z)\bigg)^{\frac \theta 2}\Bigg]^{\frac{p}{\theta}}
 \\ &\le \bigg(\int_{ \X_t}   \rho( s , x)^{\theta}z^{\theta }  \,\dd s \,\dd x\,\gl(\dd z) \bigg)^{\frac{p}{\theta}}
\\  &=  \vartheta(\theta)^{\frac p\theta} \bigg(\int_{ \calx_t}   s^{\nu_{\theta}-1 }z^\theta  \,\dd s \,\gl(\dd z) \bigg)^{\frac{p}{\theta}} 
\,,
\end{split}
\end{equation}
recalling also~\eqref{help3} and $\nu_{\theta} = 1-\frac d2(\theta-1)$ for the last line.
On the right-hand side, we see that when~$\theta$ increases, the integrability in $z$ around $0$ improves but the one in $s$ worsens, forcing us to chose $\nu_{\theta}>0$ (\textit{i.e.}, $\theta<1+\frac{2}{d}$) to obtain a finite integral. As a consequence, we need to assume $\mu_{0,\infty}(\theta)<\infty$ for some $\theta\in[p,1+\frac 2d)$. If this holds, then, in fact, the same estimate can be applied iteratively in order to obtain bounds for any of the multiple integrals in \eqref{eq:triangle}. Let us remark that this is essentially the assumption (and the method) used in \cite{Loubert98} to obtain existence, uniqueness and moments for the solution to~\eqref{eq:SHE}. 

Clearly, $\mu_{0,\infty}(\theta)<\infty$ does not hold, for any $\theta>0$, if $\xi_\om$ is an $\alpha$-stable noise with $\al\in(1,2)$; recall that we assume $\mu_{1,\infty}(p)<\infty$ for some $p>1$. In that case, at least when $k=1$, it is easy to do better than \eqref{eq:jensubad}. The 
  key point is to first separate $z\geq1$ and $z< 1$ and then, for $z< 1$, further whether  $z\le s^{d/2}$ or $z>s^{d/2}$. For $z\geq1$, we can simply apply \eqref{eq:jensubad} with $\theta=p$, that is, 
  $$ 
  \bbE\Bigg[\bigg(\int_{\X_t} \rho( s , x)^2  z^2\bone_{\{z\geq1\}}\,
  \delta_{\go}(\dd s, \dd x, \dd z)\bigg)^{\frac p2}\Bigg]\leq \vartheta(\theta)  \mu_{1,\infty}(p) \nu_p^{-1} t^{\nu_p} \,.
  $$
  For $z\leq1$, we first consider the contribution coming from   $z\le s^{d/2}$. By
 Jensen's inequality, we can take the exponent $\frac p2$ outside the expectation on the right-hand side of~\eqref{eq:BDG-2}, which leads to the bound
\begin{equation*}
\bbE\Bigg[ \int_{\X_t} \ind_{\{ z\le s^{d/2}\wedge1\}}  \rho( s , x)^2  z^2\,
 \delta_{\go}(\dd s, \dd x, \dd z)\Bigg]^{\frac p2}  =  \Bigg( \int_{ \X_t} \ind_{\{ z\le s^{d/2}\wedge1\}}  \rho( s , x)^2z^2\, \dd s\, \dd x \,\gl( \dd z)\Bigg)^{\frac p2}.
\end{equation*}
Since $d\ge 3$, we have, integrating first with respect to $x$ (recall \eqref{help3}), then with respect to $s$ and finally with respect to $z$,
\begin{multline*}
 \int_{ \X_t} \ind_{\{ z\le s^{d/2}\wedge1\}}  \rho( s , x)^2z^2\,\dd s\, \dd x\,\gl( \dd z)= \int_{ \calx_t} \frac{z^2\ind_{\{ z\le s^{d/2}\wedge1\}} }{(4\pi s)^{\frac d2}} \, \dd s\, \gl( \dd z) \le  \frac{1}{(\frac d2-1)(4\pi)^{\frac d2}}\mu_{0,1} \Big (1+\tfrac{2}{d} \Big).
\end{multline*}
For the contribution to the integral coming from  $z>s^{d/2}$,
applying~\eqref{eq:jensubad} with $\theta=p$,
we get that it is smaller than 
\begin{align*}
\vartheta(p)  \int_{ \calx_t} \ind_{\{ z> s^{d/2},z<1\}} s^{\nu_p-1} z^p \,\dd s \, \gl( \dd z) =
\frac{\vartheta(p)}{\nu_p}   \int_{ (0,1)}  z^{1+\frac 2d}  \, \gl( \dd z)  =  \frac{\vartheta(p)}{\nu_p}  \mu_{0,1}\Big (1+\tfrac{2}{d} \Big )\,,
\end{align*}
where we have  first integrated with respect to $s$, using that $\nu_p =1-\frac{d}{2} (p-1) >0$ for $p<1+\frac2d$, and then with respect to $z$.
Altogether, we have shown the following bound.
\begin{lemma}\label{lem:lematoy}
In dimension $d\geq 3$, there exists a constant $C$ (which may depend on $d$ but not on~$p$) such that for all $p\in(1,1+\frac 2d)$,
	\begin{equation*}
		\bbE\Bigg[ \bigg(\int_{ (0,t) \times \bbR^d}   \rho( s , x)\,  \bar \xi^a_{\go} (\dd s , \dd x )\bigg)^p\Bigg]
		\le  \frac{C}{\nu_p}  \Big[ \mu_{0,1}\Big(1+\tfrac{2}{d}\Big)^{\frac p2} + \mu_{0,1}\Big(1+\tfrac{2}{d}\Big)+\mu_{1,\infty}(p) t^{\nu_p} \Big].
	\end{equation*}
\end{lemma}
In other words, the  single integral has a finite $p$th moment for some $p\in(1,1+\frac 2d)$ if the intensity measure satisfies $\mu_{0,1}(1+\frac 2d)+\mu_{1,\infty}(p)<\infty$. In fact, this condition is necessary and sufficient by \cite[Theorem~3.3]{Rajput89}, so Lemma~\ref{lem:lematoy} is optimal for single integrals.

\subsection{Decoupling and partitioning: Key tools in proving Propositions \ref{prop:localbisdone} and \ref{prop:localbis}}\label{sec:start}

\subsubsection{Decoupling}

When $k\ge 2$, there is no direct analogue of \eqref{eq:BDG-2} since we cannot apply the BDG inequality for the $k$-fold iterated integral.  
The first step of our proof is to use the decoupling inequalities from Section~\ref{sec:dec} in order to obtain, instead of $W_{a,k}(t)$,  multiple Poisson integrals with respect to $k$ independent copies of the original noise. An important advantage of the decoupled integral is that we can change the order of integration without losing the martingale property.
Using Theorem \ref{thm:dec} 
with the tetrahedral function
$f(\bt,\bx,\bz):=\rho(t,0)^{-1}\rho_{t,0}( \bt , \bx)\prod_{i=1}^kz_i\ind_{\{z_i\ge a\}}$ (the reader can check that~\eqref{eq:fcond} is satisfied whenever $\mu_{1,\infty}(p)<\infty$), we have
\begin{equation}\label{eq:jedecouple}
 \bbE [  \lvert W_{a,k}(t)\rvert^p  ] \le \bigg(\frac{C}{p-1}\bigg)^k\, \bbE^{\otimes k}  [\lvert V_{a,k}(t)\rvert^p  ] \,,
\end{equation}
where 
\begin{equation}\label{eq:jeBDG}\begin{split}
V_{a,k}(t):=&~ \int_{ X_t^{(k)}}   \frac{\rho_{t,0}( \bt , \bx)}{{\rho(t,0)}}   \prod_{i=1}^k \bar \xi^a_{\go_i} (\dd t_i , \dd x_i ) \\
=&~ \int_{ \X_t^{(k)}}   \frac{\rho_{t,0}( \bt , \bx)}{{\rho(t,0)}}   \prod_{i=1}^k z_i\bone_{\{z_i\geq a\}}\, (\delta_{\go_i}-\nu)(\dd t_i , \dd x_i, \dd z_i) 
\end{split}
\end{equation}
and $\go_i$, for $i= 1,\dots,k$, are i.i.d.\ copies of $\go$. For simplicity, and with a small abuse of notation, we write $\bbP$, $\bbP^{\otimes k}$, $\E$ and $\E^{\otimes k}$ in the remainder of the proof.

\subsubsection{Partitioning and integrating over space}
As in Section~\ref{sec:k1}, we  want to estimate   the $p$th moment of the right-hand side of \eqref{eq:jeBDG} 
using a combination of the BDG inequality, Jensen's inequality and subadditivity and then integrate over space.
Because we want to use an intermediate exponent as in \eqref{eq:jensubad} that depends on the value of $\bt$ and~$\bz$,   a first task is to decompose 
$V_{a,k}(t)$ by considering   a (non-random) partition $\mathfrak{P}_k$ of the parameter space~$\calx^{(k)}_t$.
For each element $\cP\in \mathfrak{P}_k$ of our partition, we will  further determine a partition $J_1(\cP)\cup J_2(\cP)$ of $\lint k\rint$ and first integrate with respect to $(t_i,z_i)$ with $i\in J_1(\cP)$ and then with respect to $(t_i,z_i)$ with $i\in J_2(\cP)$. 
 We let $k_1$ and $k_2$ denote the respective cardinalities of~$J_1=J_1(\calp)$ and $J_2=J_2(\calp)$.

When $J\subseteq \lint k\rint$, we use the following notation:
	\begin{equation}\label{deltajnota}
		\Delta_{J} t_i:= t_i- t_{i(J,-1)}  \qquad\text{and}\qquad \Delta_J x_i:= x_i- x_{i(J,-1)}
	\end{equation}
where $x_0=x_{k+1}=0$, $t_0=0$ and $t_{k+1}=t$ as usual
and where for $i\in  J\cup\{k+1\}$ we denote the predecessor of $i$ in $J\cup\{0,k+1\}$ by  $i(J,-1)$ . 
As a result of our partitioning procedure, we obtain the following estimate:  

\begin{lemma}\label{lemmaparti}
Given  $p\in(1,2\wedge (1+\frac{2}{d}))$ and $\theta\in(p,2]$, we have for all $a\in(0,1]$,
 \begin{equation}\label{eq:dolemaparti}
 \begin{split}
		\bbE\Big[ \big\lvert V_{a,k}(t)\big\rvert^{p} \Big]^{\frac1p}  \le    C^k \sum_{\cP\in \mathfrak P_k} t^{\frac{1-\nu_p}p} \Bigg( &\int_{\calx_t^{(k_2)}}\Biggl(\int_{ \calx_t^{(k_1)}}     \ind_{\cP}(\bt,\bz ) \prod_{i=1}^{k+1} (\Delta t_i)^{\nu_{\theta}-1} \prod_{i\in J_1} z^{\theta}_i  \,\dd t_i  \,\gl(\dd z_i) \Biggr)^{\frac{p}{\theta}}  \\ 
		&\times    \prod_{i\in J_2\cup\{k+1\}} (\Delta_{J_2} t_i)^{(\nu_p-1)+\frac{p}{\theta}(1-\nu_{\theta})}\prod_{i\in J_2} z^p_i\,\dd t_i   \,\gl(\dd z_i)\Bigg)^{\frac1p}.
\end{split}
\end{equation} 
\end{lemma}

\begin{proof} 
For $\cP\in \mathfrak{P}_k$, let us define 
\begin{equation}\label{eq:VcP}
V_{a,k}(t,\cP):= \int_{ \X_t^{(k)}}  \bone_\calp(\bt,\bz) \frac{\rho_{t,0}( \bt , \bx)}{{\rho(t,0)}}   \prod_{i=1}^k z_i\bone_{\{z_i\geq a\}}\, (\delta_{\go_i}-\nu)(\dd t_i , \dd x_i, \dd z_i)\,.
\end{equation}
Then, by \eqref{eq:jeBDG} and    Minkowski's inequality, we have that
\begin{align*}
		\bbE\Big[ \big\lvert V_{a,k}(t)\big\rvert^{p} \Big]^{\frac1p} &\le  \sum_{\cP\in \mathfrak P_k}  \bbE\Big[ \big\lvert V_{a,k}(t,\cP)\big\rvert^{p}\Big]^{\frac1p}.
	\end{align*}
Because the $\om_i$'s are independent, we can, similarly to \eqref{eq:fubini} in the appendix, permute the integrals in \eqref{eq:jeBDG} and  integrate with respect to the indices in $J_1$ first and $J_2$ afterwards. In conjunction with the  BDG inequality, subadditivity (recall $p\leq 2$) and Jensen's inequality (applied in the fashion as in~\eqref{eq:jensubad}), we obtain
	\beq
	\label{eq:Vktp}
	\begin{split}
		 &\bbE[ \lvert V_{a,k}(t,\cP)\rvert^{p}] \\
&\qquad  \le \mathtoolsset{multlined-width=0.8\displaywidth}\begin{multlined}[t] C^{k_2}
	 \int_{ \X_t^{(k_2)}} \bbE\Bigg[ \bigg\lvert \int_{ \X_t^{(k_1)}}  \bone_\calp(\bt,\bz) \frac{\rho_{t,0}( \bt , \bx)}{{\rho(t,0)}}   
\prod_{i\in J_1} z_i\bone_{\{z_i\geq a\}}\, (\delta_{\go_i}-\nu)(\dd t_i , \dd x_i, \dd z_i)  \bigg\rvert^p\Bigg]
\\ 
\times \prod_{i\in J_2} z_i^p\bone_{\{z_i\geq a\}}\, \dd t_i \, \dd x_i\,\la( \dd z_i) \end{multlined}
  \\
		&\qquad \le C^k   \int_{\X_t^{(k_2)}}\bigg(\int_{ \X_t^{(k_1)}}     \ind_{\cP}(\bt,\mathbf{z})  \frac{\rho_{t,0}( \bt , \bx)^{\theta}}{\rho(t,0)^{\theta}} \prod_{i\in J_1} z^{\theta}_i  \,\dd t_i \,\dd x_i \,\gl(\dd z_i) \bigg)^{\frac{p}{\theta}}   \prod_{i\in J_2} z^p_i\,\dd t_i \,\dd x_i \,\gl(\dd z_i). 
	\end{split} 
	\eeq
\begin{rem} \label{abuse}
 In the above integrals and for the remainder of the proof, with a small abuse of notation,  the coordinates of elements of $\X_t^{(k_1)}$ are indexed by $J_1$ instead of $\lint k_1 \rint$ (and similarly for $J_2$). With this convention,
 $\X^{(k)}_t$ is a strict subset of  $\X_t^{(k_1)}\times \X_t^{(k_2)}$, but this is not a problem since the indicator function $\ind_{\cP}$ restricts the integral to a subset of $\X^{(k)}_t$.
\end{rem}

The next step is to carry out
 integration with respect to $x_1,\dots,x_k$ explicitly. 
 First of all, notice that for any $m\ge 1$  and any $p>1$, we have for every $\bt \in \mathfrak{X}_m(s)$
\begin{equation}\label{eq:iteinte}
	\int_{(\R^d)^m}\rho_{s,y}( \bt , \bx)^p\prod_{i=1}^m\dd x_i  
	= \vartheta(p)^{m}\rho(s,y)^{p} s^{1-\nu_p} \prod_{i=1}^{m+1} (\Delta t_i)^{\nu_p-1} \,.
\end{equation}
This can be proved by induction on $m$ after checking the case $m=1$ by hand, using~\eqref{help3}; we leave the details to the reader.

 We then successively apply  \eqref{eq:iteinte} with $p=\theta$ 
  to the segments of $J_1$ ({i.e.}, maximal sets of consecutive indices in $J_1$). Recalling the notation \eqref{deltajnota}, we get
\begin{equation}\label{eq:caseofJ1}
	\int_{(\bbR^d)^{k_1}}\rho_{t,0}( \bt , \bx)^{\theta}\prod_{i\in J_1}\dd x_i=\vartheta(\theta)^{k_1}\prod_{i=1}^{k+1} (\Delta t_i)^{\nu_{\theta}-1} \prod_{j\in J_2\cup\{k+1\}}(\Delta_{J_2} t_j)^{1-\nu_{\theta}}\rho(\Delta_{J_2} t_j, \Delta_{J_2} x_j)^{\theta} \,.
\end{equation}
 Applying~\eqref{eq:iteinte} in the case $y=0$,
 we also get
\begin{equation}\label{eq:caseofJ2}
	\int_{(\bbR^d)^{k_2}}\prod_{j\in J_2\cup\{k+1\}}\rho(\Delta_{J_2} t_j, \Delta_{J_2} x_j)^{p} \prod_{i\in J_2} \dd x_i
	= \vartheta(p)^{k_2} \rho(t,0)^p t^{1-\nu_p}\prod_{j\in J_2\cup\{k+1\}}  (\Delta_{J_2}t_j)^{\nu_p-1} \,.
\end{equation}
We conclude the proof of Lemma~\ref{lemmaparti} by inserting \eqref{eq:caseofJ1} and \eqref{eq:caseofJ2} in \eqref{eq:Vktp}.
\end{proof}

\subsection{The proof of Proposition \ref{prop:localbisdone}}
\label{diffikult}
The second moment was computed in Remark~\ref{rem:secmom}. In order to prove \eqref{mobdd1},
we only need to separate small and large values of $z$.
We consider a partition $\mathfrak{P}_k$ indexed by the subsets $J\subseteq \lint k\rint$ and we define 
\begin{equation}\label{partisimples}
 \cP(J):=\{ (\bt,\bz)\in \calx^{(k)}_t \ :\  z_i<\eta \text{ for all } i\in J \text{ and }  z_i\ge \eta \text{ for all } i\in \lint k\rint \setminus J\}
\end{equation}
as well as $J_1:=J$ and $J_2:=\lint k\rint \setminus J$.
Applying Lemma \ref{lemmaparti} with $\theta=2$ and noting that $\nu_2=\frac12$ and $\nu_p =1- \frac12(p-1)$ in dimension $d=1$, we obtain
\begin{equation}
\label{boundVaktp}
\begin{split}
 	\bbE&\Big[ \big\lvert V_{a,k}(t)\big\rvert^{p} \Big]^{\frac1p}  
 	 \le  C^k \sum_{J\subseteq \lint k\rint } t^{\frac{1-\nu_p}p} \mu_{0,\eta}(2)^{\frac {k_1}{2}}\mu_{\eta,\infty}(p)^{\frac{k_2}p} \\ 
		&\times\Bigg(\int_{(0,t)^{k_2}}\Biggl(\int_{(0,t)^{k_1}}     \ind_{\{t_1<\dots<t_k\}} \prod_{i=1}^{k+1} (\Delta t_i)^{-\frac{1}{2}}     \prod_{i\in J_1} \dd t_i  \Biggr)^{\frac{p}{2}} \prod_{i\in J_2\cup\{k+1\}} (\Delta_{J_2} t_i)^{\frac{2-p}{4}}\prod_{i\in J_2}\dd t_i\Bigg)^{\frac1p}. 
		\end{split}
\end{equation}
Integrating successively over the segments of $J_1$
and writing $\Delta_{J_2}(i) = i- i(J,-1)$ for the distance between $i\in J_2 \cup\{k+1\}$ and the previous index in $J_2$ (recall  \eqref{deltajnota}),
we get from Lemma~\ref{lemGamma}
that
\begin{equation}
\label{facil}
 \int_{(0,t)^{k_1}}     \bone_{\mathfrak{X}_k(t)} (\bt) \prod_{i=1}^{k+1} (\Delta t_i)^{-\frac{1}{2}}     \prod_{i\in J_1} \dd t_i
 =   \bone_{\mathfrak{X}_{k_2}(t)} ( (t_i)_{i\in J_2})  \prod_{i \in J_2\cup\{k+1\}} (\Delta_{J_2} t_i )^{\frac12\Delta_{J_2}(i)-1} \frac{\gG(\frac12)^{\Delta_{J_2}(i)}}{\gG(\frac12\Delta_{J_2}(i))}\,.
\end{equation}
Now, if we bound  $(\Delta_{J_2} t_i)^{ (\Delta_{J_2}(i)-1)/2} \leq t^{ (\Delta_{J_2}(i)-1)/2}$ and $\gG(\frac12\ell)\geq \frac12\sqrt{\pi}$ for any $\ell\geq 1$, we  get that
\[
\int_{(0,t)^{k_1}}   \bone_{\mathfrak{X}_k(t)} (\bt) \prod_{i=1}^{k+1} (\Delta t_i)^{-\frac{1}{2}}     \prod_{i\in J_1} \dd t_i 
 \leq \bone_{\mathfrak{X}_{k_2}(t)} ( (t_i)_{i\in J_2})  
 2^{k_2+1}  (\pi t)^{\frac12 k_1}\prod_{i\in J_2\cup\{k+1\}} (\Delta_{J_2} t_i)^{-\frac12}.
\]
Going back to~\eqref{boundVaktp} and taking the factor  $[(\Delta_{J_2} t_i)^{-1/2}]^{p/2}$ over to the outer integral, we compute
\begin{equation}
\label{facil2}
 \int_{(0,t)^{k_2}} \bone_{\mathfrak{X}_{k_2}(t)} ( (t_i)_{i\in J_2})   \prod_{i\in J_2\cup\{k+1\}} (\Delta_{J_2} t_i)^{\frac{1-p}{2}}\prod_{i\in J_2}\dd t_i  = t^{\nu_p k_2 +\nu_p-1} \frac{\gG(\nu_p)^{k_2+1}}{\gG(\nu_p(k_2+1))}\,,
\end{equation}
thanks again to Lemma~\ref{lemGamma}.
We therefore conclude that 
\begin{equation}
\label{lastVakt}
\begin{split}
  	 	\bbE\Big[ \big\lvert V_{a,k}(t)\big\rvert^{p} \Big]^{\frac1p}   &\le \sum_{J\subseteq \lint k\rint }  \Big(   \sqrt{\pi t}\mu_{0,\eta}(2) \Big)^{\frac{k_1}{2}} \Big( 2^{\frac p2} \mu_{\eta,\infty}(p) t^{\nu_p}\gG(\nu_p) \Big)^{\frac{k_2}p} \frac{  \sqrt{2} \gG(\nu_p)^{\frac1p}}{\gG(\nu_p(k_2+1))^{\frac1p}}\\
  	&\le (2C)^k  \max_{k_1+k_2=k} \Big(   \sqrt{\pi t}\mu_{0,\eta}(2) \Big)^{\frac{k_1}{2}} \Big(  2^{\frac p2}\mu_{\eta,\infty}(p) t^{\nu_p}\gG(\nu_p) \Big)^{\frac{k_2}p}\frac{ \sqrt{2}\gG(\nu_p)^{\frac1p}}{\gG(\nu_p(k_2+1))^{\frac1p}}.
 \end{split}
\end{equation}
Absorbing $\pi^{k_1/4}$ and $(2^{p/2}\gG(\nu_p))^{k_2/p}$, which is uniformly bounded in $p\in [1,2]$, into the constant~$C$ and  replacing the maximum by a sum, we derive \eqref{mobdd1} from 
\eqref{eq:triangle}, \eqref{eq:jedecouple} and \eqref{lastVakt}.

\medskip
If $p\in(2,3)$, the subadditivity argument in \eqref{eq:Vktp} does not apply. Instead, we shall use a variant of the BDG estimate \eqref{eq:BDG-2} that only contains the intensity measure $\nu:=\dd t\otimes\dd x\otimes\la(\dd z)$ instead of~$\delta_\om$. 
For any $p\geq2$, there exists a constant $C'_p\in(0,\infty)$ such that for all $\calp\otimes\calb(\R^d\times(0,\infty))$-measurable process $K=K(\om,w)=K(\om,t,x,z)$, 
\begin{equation}\label{eq:BJ}
	\E\Bigg[\bigg\lvert \int_{\X} K(w)\,(\delta_\om-\nu)(\dd w)\bigg\rvert^p\Bigg] \leq C'_p \left\{ \E\Bigg[\bigg(\int_{\X} K(w)^2\,\nu(\dd w)\bigg)^{\frac p2}\Bigg] + \E\Bigg[\int_{\X} \lvert K(w)\rvert ^p\,\nu(\dd w)\Bigg] \right\},
\end{equation}
where $C'_p=(2^{p-1}p(p-1)(p/(p-1))^p)^{p/2}\vee 2(p/(p-1))^p(2^p+1+p)\leq 729$ when $p\in[2,3]$;
see \cite[Theorem~1 (b)]{Novikov75} (and its proof for the value of $C'_p$). We also refer to \cite{Marinelli14} for a survey of various versions and proofs (and names) of this inequality. 
Let us define
$$ \wt W_{a,0}(t,x):=\rho(t,x),\qquad \wt W_{a,k}(t,x):= \int_{ X^{(k)}_t}   {\rho_{t,x}( \bt , \bx)}  \prod_{i=1}^k \bar \xi^a_{\go} (\dd t_i , \dd x_i ) \quad \text{ for } k\geq1,$$
 so that in particular $W_{a,k}(t) = \wt W_{a,k}(t,0)/\rho(t,0)$.
With this definition, applying~\eqref{eq:BJ}, we have
\begin{align*}
		\E\Big[ |\wt W_{a,k}(t,x)|^p\Big]^{\frac1p}	&=
  \bbE\Bigg[ \bigg| \int_{X_t} \rho(t-t_k,0-x_k)\wt W_{a,k-1}(t_k,x_k)\, \bar \xi^a_\om(\dd t_k,\dd x_k)\bigg|^p\Bigg ]^{\frac1p}\\
	&\leq  C' \,\Bigg\{\bbE\Bigg[ \bigg( \mu_{0,\infty}(2)\int_{X_t} (\rho(t-t_k,x_k)\wt W_{a,k-1}(t_k,x_k))^2\,\dd t_k\,\dd x_k\bigg)^{\frac p2}\Bigg ]^{\frac 1p}\\
	& \qquad \qquad \qquad + \bigg(\mu_{0,\infty}(p)\int_{X_t} \rho(t-t_k,x_k)^p\bbE[| \wt W_{a,k-1}(t_k,x_k)|^p]\,\dd t_k\,\dd x_k\bigg)^{\frac1p} \Bigg\} \,.\\
\end{align*}
Thus, applying Minkowski's integral inequality,
we get
\begin{align*}
\E\Big[ |\wt W_{a,k}(t,x)|^p\Big]^{\frac1p} &\leq   C' \Bigg\{\,\mu_{0,\infty}(2)^{\frac12} \bigg( \int_{X_t} \rho(t-t_k,x_k)^2\E[\lvert \wt W_{a,k-1}(t_k,x_k)\rvert^p]^{\frac 2p}\,\dd t_k\,\dd x_k\bigg)^{\frac12}\\
	&\qquad \qquad \qquad	+ \mu_{0,\infty}(p)^{\frac1p} \bigg( \int_{X_t} \rho(t-t_k,x_k)^p\bbE[| \wt W_{a,k-1}(t_k,x_k)|^p]\,\dd t_k\,\dd x_k \bigg)^{\frac1p} \Bigg\} \,. 
\end{align*}
Repeating this estimate and recalling~\eqref{eq:triangle}, we get 
\begin{equation}\label{eq:help4}
		\E\Big[ \bar \calz^{\om,a}_{\beta}(t,0)^p\Big]^{\frac1p}	\le \rho(t,0)+  \sum_{k=1}^\infty (C'\beta)^k\sum_{\btheta\in\{2,p\}^k}\mu_{0,\infty}(p)^{\frac{k_1}{p}} \mu_{0,\infty}(2)^{\frac {k_2}{2}} \lVert \rho_{t,0}(\bt,\bx)\rVert_{\btheta,t},
\end{equation}
where $k_1$ is the number of $p$'s and $k_2$ the number of $2$'s in $\btheta$,
and where we have defined
\begin{equation*}
 \lVert f\rVert_{\btheta,t}:= \bigg(\int_{X_t} \bigg(\cdots\bigg(\int_{X_t} \bigg(\int_{X_t} f(\bt,\bx)^{\theta_1} \,\dd t_1\,\dd x_1\bigg)^{\frac{\theta_2}{\theta_1}}\,\dd t_2\,\dd x_2\bigg)^{\frac{\theta_3}{\theta_2}}\cdots\bigg)^{\frac{\theta_k}{\theta_{k-1}}}\,\dd t_k\,\dd x_k\bigg)^{\frac1{\theta_k}}
 \end{equation*}
for $f:((0,\infty)\times\R^d)^k\to[0,\infty)$ and $\btheta=(\theta_1,\dots,\theta_k)\in[1,\infty)^k$.

 Now in order to conclude, we want to  replace $\lVert \rho_{t,0}(\bt,\bx)\rVert_{\btheta,t}$ by an integral which is analogous to that found in  the on the right-hand side of  \eqref{eq:Vktp}. To do so, we use the following identity, valid for any positive function~$f$ and any measures~$\mu_1$ and $\mu_2$:
\begin{equation*}
\label{interchanging}
 \left( \int_{\Omega_1}  \left(\int_{\Omega_2}  f(w_1,w_2)^{2}  \,\mu_2(\dd w_2) \right)^{\frac{p}{2}} \mu_1(\dd w_1)\right)^{\frac1p}  \le   \left( \int_{\Omega_2}  \left(\int_{\Omega_1}  f(w_1,w_2)^{p} \, \mu_1(\dd w_1) \right)^{\frac{2}{p}} \,\mu_2(\dd w_2)\right)^{\frac12} \,.
\end{equation*}
This inequality is a special case of \cite[Lemma~3.3.1]{Kwapien92}. 
Now, letting~$J_1$ denote the set of indices for which $\theta_i=p$ and $J_2$ those for which $\theta_i=2$,
we apply the above inequality iteratively to take all the integrals with respect to $J_1$ inside. We obtain (recall Remark~\ref{abuse})
\begin{equation}
 \lVert \rho_{t,0}(\bt,\bx)\rVert_{\btheta,t}  \le  \bigg(  \int_{X_t^{(k_2)}}\bigg(\int_{ X_t^{(k_1)}}   \ind_{\mathfrak X_k(t)}(\bt)\rho_{t,0}( \bt , \bx)^{p} \prod_{i\in J_1} \dd t_i \,\dd x_i  \bigg)^{\frac{2}{p}}   \prod_{i\in J_2} \dd t_i \,\dd x_i   \bigg)^{\frac12} .
 \end{equation}
 Now we can first integrate with respect to the $x_i$'s using \eqref{eq:caseofJ1} and \eqref{eq:caseofJ2} (with $p$ instead of $\theta$ and $2$ instead of $p$). Recalling that $\nu_2=\frac12$, we  obtain
 \begin{multline*}
 \lVert \rho_{t,0}(\bt,\bx)\rVert_{\btheta,t}  \le 
 \vartheta(p)^{\frac{k_1}{p} } \vartheta(2)^{\frac{k_2}2} \rho(t,0)  \\ \times  t^{\frac14} \bigg(\int_{(0,t)^{k_2}} 
   \bigg(\int_{(0,1)^{k_1}} \ind_{\mathfrak X_k(t)}(\bt)\prod_{i=1}^{k+1} (\Delta t_i)^{\nu_p-1} \prod_{i\in J_1} \dd t_i  \bigg)^{\frac{2}{p}}   \prod_{i\in J_2\cup \{k+1\}} (\Delta_{J_2} t_i)^{ \frac{2}{p}(1-\nu_p)-\frac12} \prod_{i\in J_2}  \dd t_i \bigg)^{\frac12} .
 \end{multline*}
Using Lemma~\ref{lemGamma} to integrate first with respect to $(t_i)_{i\in J_1}$, 
we get as in~\eqref{facil}
\beq\label{facil3}\begin{split}
\int_{(0,1)^{k_1}} \ind_{\mathfrak X_k(t)}(\bt)\prod_{i=1}^{k+1} (\Delta t_i)^{\nu_p-1} \prod_{i\in J_1} \dd t_i  
&=\bone_{\mathfrak{X}_{k_2}(t)} ( (t_i)_{i\in J_2})   \prod_{i \in J_2\cup\{k+1\}} \frac{(\Delta_{J_2} t_i )^{\nu_p \Delta_{J_2}(i)-1}  \gG(\nu_p)^{\Delta_{J_2}(i)} }{\gG(\nu_p\Delta_{J_2}(i))}   \\
& \leq \bone_{\mathfrak{X}_{k_2}(t)} ( (t_i)_{i\in J_2})  \gG(\nu_p)^{k+1} t^{\nu_p k_1} 
\frac{\prod_{i \in J_2\cup\{k+1\}} (\Delta_{J_2} t_i )^{\nu_p -1}  }{\prod_{i \in J_2\cup\{k+1\}}\gG(\nu_p\Delta_{J_2}(i))} \,,
\end{split}\eeq
because $(\Delta_{J_2}t_i)^{\nu_p\Delta_{J_2}(i)-1} \leq t^{\nu_p(\Delta_{J_2}(i)-1)} (\Delta_{J_2}t_i)^{\nu_p-1}$. Pulling $(\Delta_{J_2}t_i)^{\nu_p-1}$ to the outer integral, we evaluate
\begin{align*}
 \int_{(0,t)^{k_2}} \ind_{\mathfrak X_{k_2}(t)}( (t_i)_{i\in J_2}) \prod_{i\in J_2\cup \{k+1\}} (\Delta_{J_2} t_i)^{-\frac12} \prod_{i\in J_2} \dd t_i   = t^{\frac{k_2}{2} -\frac12} \frac{\Gamma(\frac12)^{k_2+1}}{ \Gamma( \frac{k_2+1}{2})} \,,
\end{align*}
using once more Lemma~\ref{lemGamma}.
Altogether, bounding $\vartheta(p),\vartheta(2)\leq 1$, we obtain that
\[
 \lVert \rho_{t,0}(\bt,\bx)\rVert_{\btheta,t}  \le  \rho(t,0) t^{\frac{\nu_p}{p} k_1 +\frac 14 k_2} \frac{\Ga(\nu_p)^{\frac1p(k+1)}\pi^{\frac14 (k_2+1)}}{\Ga(\frac{k_2+1}{2})^{\frac12} \prod_{i\in J_2\cup \{k+1\}}\Ga( \nu_p \Delta_{J_2}(i) )^{\frac1p}} \,.
\]

Going back to~\eqref{eq:help4}, bounding $\vartheta(p),\vartheta(2)\leq 1$ and expressing everything in terms of $J_2$ and $k_2 =\lvert J_2\rvert$, we get that $\E[ \rho(t,0)^{-p} \bar \calz^{\om,a}_{\beta}(t,0)^p]^{1/p}$ is bounded by
\begin{multline*}
\pi^{\frac14}  \Gamma(\nu_p)^{\frac1p} \sum_{k=0}^\infty \sum_{J_2\subseteq \lint k\rint} 
 \frac{ \big( C'\beta \mu_{0,\infty}(p)^{\frac1p} t^{\frac{\nu_p}{p}} \Gamma(\nu_p)^{\frac1p} \big)^{k_1} \big( C'\beta  \mu_{0,\infty}(2)^{\frac12}  \Gamma(\nu_p)^{\frac1p}\pi^{\frac14}  t^{\frac14} \big)^{k_2}}{ \Gamma(\frac{k_2+1}{2})^{\frac12}\prod_{i\in J_2\cup \{k+1\}}\Ga(  \nu_p\Delta_{J_2}(i) )^{\frac1p}} \\
 =\pi^{\frac14} \Gamma(\nu_p)^{\frac1p}   \sum_{k=0}^{\infty} \sum_{k_2=0}^{k} \frac{ \big( C'\beta  \mu_{0,\infty}(2)^{\frac12}  \Gamma(\nu_p)^{\frac1p}\pi^{\frac14}  t^{\frac14} \big)^{k_2}  }{  \Gamma(\frac{k_2+1}{2})^{\frac12}} \\
\times \sumtwo{\ell_1,\ldots, \ell_{k_2+1} \geq 1}{\ell_1+\cdots +\ell_{k_2+1} =k+1}  \prod_{i=1}^{k_2+1}   \frac{ \big( C'\beta \mu_{0,\infty}(p)^{\frac1p} t^{\frac{\nu_p}{p}} \Gamma(\nu_p)^{\frac1p} \big)^{\ell_i-1} }{ \Ga(  \nu_p \ell_i )^{\frac1p}} \,,
\end{multline*}
where we have used a change of variable and the fact that $\ell_1+\cdots +\ell_{k_2+1} =k+1$ implies that $k_1 = \sum_{i=1}^{k_2+1} (\ell_i-1)$.
Exchanging the first two sums and factorizing the last one yields \eqref{mobdd2}.
\qed

\subsection{Intermezzo: $d\ge 2$ under stronger moment conditions}\label{intermission}
The method we have used for $d=1$ can further be used  to prove boundedness of moments  under the more restrictive assumption given in  
\eqref{assump1}.
We illustrate this by Proposition \ref{replica} below whose proof can be achieved by replicating that of \eqref{mobdd1}.
This  quantitative estimate is sufficient to prove the upper bound parts of  Theorem \ref{thm:heavyup} and   similar ideas could in principle be used to prove Theorem \ref{thm:SHE} under the more restrictive assumption \eqref{assump1}. The  more involved method used in the proof of Proposition \ref{prop:localbis} below, however, is   necessary  to bridge the gap between the conditions \eqref{assump1} and \eqref{assump2}; recall that the latter   is optimal in dimension~$2$ and very close to optimal when $d\ge 3$. 

\begin{proposition}\label{replica}
There exists a constant $C\in(0,\infty)$ that only depends on $d$  such that if 
 $d\ge 2$, $\mu_{0,1}(q)<\infty$ and $\mu_{1,\infty}(p)<\infty$ for some $1<p\le q< 1+\frac{2}{d}$, we have, for all $a\in[0,1]$,
 $\eta\in(0,1]$ and $\beta,t>0$,
 \begin{equation*}
 \E\Big[\rho(t,x)^{-p}\bar \calz^{\om,a}_{\beta}(t,x) ^p\Big]^{\frac1p}
	\le   \sum_{k_1,k_2=0}^\infty \bigg(\frac{C \Gamma(\nu_q) \beta}{p-1}\bigg)^{k_1+k_2}  (\mu_{0,\eta}(q)t^{\nu_q} )^{\frac{k_1}q} \frac{  (\Ga(\nu_p)\mu_{\eta,\infty}(p) t^{\nu_p} )^{\frac{k_2}p}\Gamma(\nu_p)^{\frac1p}}{\Gamma(\nu_p(k_2+1))^{\frac1p}}.
	 \end{equation*}
\end{proposition}
\begin{proof}
 Let us give a very short guideline for the proof: we use the same partition $\mathfrak{P}_k$ as in \eqref{partisimples}
 and apply Lemma \ref{lemmaparti} with $\theta=q$.
 The proof is identical to that of   Proposition \ref{prop:localbisdone} \textit{(ii)} except that \eqref{facil} 
   has to be replaced by~\eqref{facil3}, which together with the bound  $\Gamma(\nu_q \Delta_{J_2}(i)) \geq \frac12$ gives
$$
 	\int_{(0,t)^{k_1}} \bone_{\mathfrak{X}_{k}(t)} ( \bt) \prod_{i=1}^{k+1}(\Delta t_i)^{\nu_q-1}\prod_{i\in J_1} \dd t_i \leq \bone_{\mathfrak{X}_{k_2}(t)} ( (t_i)_{i\in J_2}) \, 2^{k_2+1} \Gamma(\nu_q)^k t^{\nu_q k_1}\prod_{i\in J_2\cup\{k+1\}} (\Delta_{J_2} t_i)^{\nu_q-1}\,.\qedhere
$$
\end{proof}

\subsection{The proof of Proposition \ref{prop:localbis}}

We use the same idea as for the proof of Proposition~\ref{prop:localbisdone}.
We will consider a partition $\mathfrak P_k$ (defined below)
of the parameter space $\mathcal X_t^{(k)}$:
each element $\cP \in \mathfrak P_k$ induces a partition of $\llbracket k\rrbracket$ into two sets $J_1$ and $J_2$ and we are going to use Lemma \ref{lemmaparti} with $\theta=1+\frac{2}{d}$ (and hence $\nu_{\theta}=0$).
The main technical part of this section will then be to bound 
\begin{equation}\label{defcalu}\mathtoolsset{multlined-width=0.9\displaywidth}\begin{multlined}
 U(\cP,t):=t^{{1-\nu_p}} \int_{\calx_t^{(k_2)}}\Bigg(\int_{\calx_t^{(k_1)}}     \ind_{\cP}(\bt,\mathbf{z})   \prod_{i=1}^{k+1} (\Delta t_i)^{-1} \prod_{i\in J_1}  z^{1+\frac{2}{d}}_i  \,\dd t_i \,\gl(\dd z_i)\Bigg)^{\frac{p}{1+2/d}}\\ \times		\prod_{j\in J_2\cup\{k+1\}} (\Delta_{J_2} t_j)^{{\frac{p}{1+2/d}}+\nu_p-1} \prod_{i\in J_2} z^p_i \, \dd t_i \,\gl(\dd z_i) \, .
 \end{multlined}
\end{equation}
The following is the main technical estimate of this section.
\begin{proposition}\label{prop:dekz}
Recall \eqref{eq:definingl} and \eqref{eq:defzezeta}.
With the choice $\mathfrak P_k$ defined in Section~\ref{sec:partition} below (see in particular~\eqref{defcalp}),
we have for every $\cP\in \mathfrak P_k$, $\eta\in(0,1]$ and $t>0$ that
	\begin{equation}\label{eq:finalbound}\begin{split}
		U(\cP,t)\le \zeta_1(\eta,p,t)^{k_1}
		\zeta_2(\eta,p)^{k_2}\Lambda(k_2,t,p),
	\end{split}
 	\end{equation}
 where $k_1:=\lvert J_1\rvert$ and $k_2:=\lvert J_2\rvert$.
\end{proposition}
Proposition~\ref{prop:localbis} follows immediately from Proposition~\ref{prop:dekz}.
\begin{proof}[Proof of Proposition \ref{prop:localbis}]
  Using Lemma \ref{lemmaparti}, combined with \eqref{eq:jedecouple} and \eqref{eq:jeBDG}, we obtain that 
  \begin{equation}\label{eq:help6}\begin{split}
   \bbE[W_{a,k}(t)^p]^{\frac1p}&\le \bigg(\frac{C}{p-1}\bigg)^k \lvert \mathfrak{P}_k\rvert
   \max_{\cP\in \mathfrak P_k} U(\cP,t)^{\frac1p}\\
   &\le 2\bigg(\frac{9C}{p-1}\bigg)^k \max_{k_1+k_2=k}\Big(\zeta_1(\eta,p,t)^{k_1}
		\zeta_2(\eta,p)^{k_2}\Lambda(k_2,t,p)\Big)^{\frac1p} \,,
		\end{split}
  \end{equation}
  where we have used that $\lvert\mathfrak P_k\rvert \leq 2 \times 9^k$ (see below).
Replacing the $\max$ by a sum yields \eqref{compinek}.
\end{proof}

\subsubsection{Constructing the partition $\mathfrak{P}_k$}
\label{sec:partition}

The construction of $\mathfrak{P}_k$ when $d\ge 2$ is considerably more involved,
as it   no longer suffices to only differentiate between $z_i\geq\eta$ and $z_i<\eta$ as we did for $d = 1$. Recall that 
in Section~\ref{sec:k1},  in order to obtain optimal bounds, we had to split the integral according to how $z_i$ and $\Delta t_i^{d/2}$ compare to each other. A problem arises when trying to generalize this method to $k\ge 2$: 
On the right-hand side of \eqref{defcalu},
the value of $\Delta t_i$ plays the same role as $\Delta t_{i+1}$, which is the reason why we cannot break down the $k$-fold integral into $2^k$ parts as we did in the proof of Proposition \ref{prop:localbisdone}.

We now describe our solution to this problem.
To give the idea behind our partition and its link to the bound~\eqref{eq:finalbound}, recall from~\eqref{defcalu} that
we first integrate with respect to the variables with indices in $J_1$ and then with respect to those with indices in $J_2$.
Loosely speaking, indices in $J_1$ correspond to values of
 $z_i$ that are small compared to $(\Delta t_i)^{d/6}$ and $(\Delta t_{i+1})^{d/6}$; this gives rise to a factor $\zeta_1(\eta,p,t)^{k_1}$ that can be explained by the calculations in~\eqref{eq:samspi}--\eqref{eq:samspibis} below.
Note that after integrating with respect to such an index,
one may have to update the set of parameters since for the next step one has to compare $z_i$ with time increments formed between the remaining variables. These may differ from $\Delta t_i$ and $\Delta t_{i+1}$, which explains why  the partition has to be defined iteratively.

Each element of our partition  $\cP \in \mathfrak{P}_k$ is encoded by a finite sequence $( L^{j}, I^{j}_-, I^{j}_+, D^j)^m_{j=1}$ of partitions of $\lint k+1\rint$;
the length $m$ of the sequence is a variable. With some abuse of notation, we identify $\cP$ with this sequence. We also use the notation $I^j:=I^j_-\cup I^j_{+}$.
Not every sequence is admissible, so let us present the rules for constructing the set of admissible sequences:
\begin{itemize}
	\item First we partition $\lint k+1\rint$ into three sets $L^1$, $I^1_+$, and $I^1_-$, imposing that  $k+1 \in I^1$, and we define $D^1:=\emptyset$.
	\item
	The procedure is then iterative. Assume that $m\ge j$ and that  one has constructed the sets of the first $j$ steps. We let 
	$i(+1,j)$ denote the successor of $i$ and $i(-1,j)$ denote the predecessor of $i$ in $L^j\cup I^j_- \cup I^j_+$, that is, 
	\begin{equation}\begin{split}\label{predesuc}
	 i(+1,j)&:=\min\{ \ell \in L^j\cup I^j_- \cup I^j_+  \ : \ \ell\ge i+1  \}\,,\\
	 	 i(-1,j)&:=\max\{ \ell \in L^j\cup I^j_- \cup I^j_+  \ :\ \ell\le i-1  \} \,,
	 	 \end{split}
	\end{equation}
with the convention $\max\emptyset = 0$; we   only use the notation $i(+1,j)$ if the set over which the minimum is taken is non-empty.
	Unless
		\beq\label{eq:m}
	L^j_{\times}:= \{i \in \lint k\rint  \ : \ i\in L^{j} \text{ and } i(+1,j)\in I_-^{j} \}=\emptyset\,,
	\eeq
	we need to add extra terms corresponding to $j+1$ to our sequence, which will be described in the next point. This procedure is repeated until at some stage $j$, \eqref{eq:m} is satisfied. 
		 We then define $m=j$ and our sequence is complete.

	\item When $L^j_{\times}\ne \emptyset$, the next sets $(L^{j+1}, I^{j+1}_-, I^{j+1}_+, D^{j+1})$ are only partially determined by $(L^{j}, I^{j}_-, I^{j}_+, D^j)$: while we prescribe the choice 
	 $D^{j+1}:=D^j\cup  L^j_{\times}$,  there is some liberty for choosing the  sets
	$L^{j+1}$, $I^{j+1}_-$ and $I^{j+1}_+$.
	We set 
$$I^j_{\times}:= \{i(+1,j) \ :\  i\in L^{j}_{\times} \}$$
	and consider an arbitrary partition $\{I^{(j+1)}_-, I^{(j+1)}_+, L^{(j+1)}\}$  of 
	$I^j_{\times}$ subject to only one constraint: if   $k+1\in I^j_{\times}$, then 
	$k+1\in I^{(j+1)}_-\cup I^{(j+1)}_+$ (this is to guarantee that $k+1\in I^j$ for all $j$). We then 
 define $L^{j+1}$, $I^{j+1}_-$ and $I^{j+1}_+$
 as follows:
	\begin{equation*}\left\{\begin{aligned}
		L^{j+1}&:=(L^j \setminus L^j_{\times})\cup L^{(j+1)}, \\  I^{j+1}_-&:=(I^j_- \setminus
		I^j_{\times})\cup I^{(j+1)}_-,\\
		I^{j+1}_+&:=I^j_+ \cup I^{(j+1)}_+.
		\end{aligned}\right.
	\end{equation*}
\end{itemize}
Given an element in $\mathfrak{P}_k$, we further define
\begin{equation*}
	J_1:=D^m\cup \{ i\in L^{m}\ : \ i(+1,m)\in L^{m}\}  \qquad  \text{and} \qquad J_2:=\lint k\rint\setminus J_1.
\end{equation*}
An example of an admissible sequence of partitions as well as the associated sets $J_1$ and $J_2$ are shown in Figure~\ref{fig:partitions}.
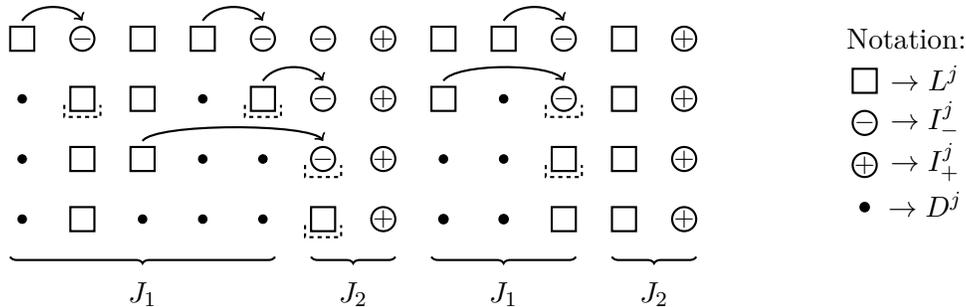
\begin{figure}[h!]
	\begin{center}
		\begin{tikzpicture}[scale=0.8]
			\draw[thick]  (-0.2,2.8) -- (-0.2,3.2) -- (0.2,3.2) -- (0.2,2.8) -- cycle ;
			\draw[thick] (1,3) circle (0.2);
			\draw (1,3) node {\small $-$};
			\draw[thick,->] (0,3.3) .. controls (0.2,3.6) and (0.8,3.6) .. (1,3.3);
			\draw[thick]  (1.8,2.8) -- (1.8,3.2) -- (2.2,3.2) -- (2.2,2.8) -- cycle ;
			\draw[thick]  (2.8,2.8) -- (2.8,3.2) -- (3.2,3.2) -- (3.2,2.8) -- cycle ;
			\draw[thick] (4,3) circle (0.2);
			\draw (4,3) node {\small $-$};
			\draw[thick,->] (3,3.3) .. controls (3.2,3.6) and (3.8,3.6) .. (4,3.3);
			\draw[thick] (5,3) circle (0.2);
			\draw (5,3) node {\small $-$};
			\draw[thick] (6,3) circle (0.2);
			\draw (6,3) node {\small $+$};
			\draw[thick]  (6.8,2.8) -- (6.8,3.2) -- (7.2,3.2) -- (7.2,2.8) -- cycle ;
			\draw[thick]  (7.8,2.8) -- (7.8,3.2) -- (8.2,3.2) -- (8.2,2.8) -- cycle ;
			\draw[thick] (9,3) circle (0.2);
			\draw (9,3) node {\small $-$};
			\draw[thick,->] (8,3.3) .. controls (8.2,3.6) and (8.8,3.6) .. (9,3.3);
			\draw[thick]  (9.8,2.8) -- (9.8,3.2) -- (10.2,3.2) -- (10.2,2.8) -- cycle ;
			\draw[thick] (11,3) circle (0.2);
			\draw (11,3) node {\small $+$};
			\filldraw (0,2) circle (2pt);
			\draw[thick]  (0.8,1.8) -- (0.8,2.2) -- (1.2,2.2) -- (1.2,1.8) -- cycle ;
			\draw[thick,dash pattern={on 1.55 off 1.55}] (0.7,1.9) -- (0.7,1.7) -- (1.3,1.7) -- (1.3, 1.9);
			\draw[thick]  (1.8,1.8) -- (1.8,2.2) -- (2.2,2.2) -- (2.2,1.8) -- cycle ;
			\filldraw (3,2) circle (2pt);
			\draw[thick]  (3.8,1.8) -- (3.8,2.2) -- (4.2,2.2) -- (4.2,1.8) -- cycle ;
			\draw[thick,dash pattern={on 1.55 off 1.55}] (3.7,1.9) -- (3.7,1.7) -- (4.3,1.7) -- (4.3, 1.9);
			\draw[thick] (5,2) circle (0.2);
			\draw (5,2) node {\small $-$};
			\draw[thick,->] (4,2.3) .. controls (4.2,2.6) and (4.8,2.6) .. (5,2.3);
			\draw[thick] (6,2) circle (0.2);
			\draw (6,2) node {\small $+$};
			\draw[thick]  (6.8,1.8) -- (6.8,2.2) -- (7.2,2.2) -- (7.2,1.8) -- cycle ;
			\filldraw (8,2) circle (2pt);
			\draw[thick] (9,2) circle (0.2);
			\draw (9,2) node {\small $-$};
			\draw[thick,dash pattern={on 1.55 off 1.55}] (8.7,1.9) -- (8.7,1.7) -- (9.3,1.7) -- (9.3, 1.9);
			\draw[thick,->] (7,2.3) .. controls (7.2,2.6) and (8.8,2.6) .. (9,2.3);
			\draw[thick]  (9.8,1.8) -- (9.8,2.2) -- (10.2,2.2) -- (10.2,1.8) -- cycle ;
			\draw[thick] (11,2) circle (0.2);
			\draw (11,2) node {\small $+$};
			\filldraw (0,1) circle (2pt);
			\draw[thick]  (0.8,0.8) -- (0.8,1.2) -- (1.2,1.2) -- (1.2,0.8) -- cycle ;
			\draw[thick]  (1.8,0.8) -- (1.8,1.2) -- (2.2,1.2) -- (2.2,0.8) -- cycle ;
			\filldraw (3,1) circle (2pt);
			\filldraw (4,1) circle (2pt);
			\draw[thick] (5,1) circle (0.2);
			\draw (5,1) node {\small $-$};
			\draw[thick,->] (2,1.3) .. controls (2.2,1.6) and (4.8,1.6) .. (5,1.3);
			\draw[thick,dash pattern={on 1.55 off 1.55}] (4.7,0.9) -- (4.7,0.7) -- (5.3,0.7) -- (5.3, 0.9);
			\draw[thick] (6,1) circle (0.2);
			\draw (6,1) node {\small $+$};
			\filldraw (7,1) circle (2pt);
			\filldraw (8,1) circle (2pt);
			\draw[thick]  (8.8,0.8) -- (8.8,1.2) -- (9.2,1.2) -- (9.2,0.8) -- cycle ;
			\draw[thick,dash pattern={on 1.55 off 1.55}] (8.7,0.9) -- (8.7,0.7) -- (9.3,0.7) -- (9.3, 0.9);
			\draw[thick]  (9.8,0.8) -- (9.8,1.2) -- (10.2,1.2) -- (10.2,0.8) -- cycle ;
			\draw[thick] (11,1) circle (0.2);
			\draw (11,1) node {\small $+$};
			\filldraw (0,0) circle (2pt);
			\draw[thick]  (0.8,-0.2) -- (0.8,0.2) -- (1.2,0.2) -- (1.2,-0.2) -- cycle ;
			\filldraw  (2,0) circle (2pt);
			\filldraw  (3,0) circle (2pt);
			\filldraw  (4,0) circle (2pt);
			\draw[thick]  (4.8,-0.2) -- (4.8,0.2) -- (5.2,0.2) -- (5.2,-0.2) -- cycle ;
			\draw[thick,dash pattern={on 1.55 off 1.55}] (4.7,-0.1) -- (4.7,-0.3) -- (5.3,-0.3) -- (5.3,-0.1);
			\draw[thick] (6,0) circle (0.2);
			\draw (6,0) node {\small $+$};
			\filldraw[thick] (7,0) circle (2pt);
			\filldraw[thick] (8,0) circle (2pt);
			\draw[thick]  (8.8,-0.2) -- (8.8,0.2) -- (9.2,0.2) -- (9.2,-0.2) -- cycle ;
			\draw[thick]  (9.8,-0.2) -- (9.8,0.2) -- (10.2,0.2) -- (10.2,-0.2) -- cycle ;
			\draw[thick] (11,0) circle (0.2);
			\draw (11,0) node {\small $+$};
			\draw [thick,decoration={brace,mirror,raise=0.5cm},decorate] (-0.2,0) -- (4.2,0);
			\draw (2,-1.25) node {$J_1$}; 
			\draw [thick,decoration={brace,mirror,raise=0.5cm},decorate] (6.8,0) -- (9.2,0);
			\draw (8,-1.25) node {$J_1$}; 
			\draw [thick,decoration={brace,mirror,raise=0.5cm},decorate] (4.8,0) -- (6.2,0);
			\draw (5.5,-1.25) node {$J_2$}; 
			\draw [thick,decoration={brace,mirror,raise=0.5cm},decorate] (9.8,0) -- (11.2,0);
			\draw (10.5,-1.25) node {$J_2$}; 
			\draw (14.7,3) node {Notation:};
			\draw[thick]  (13.8,2.1) -- (13.8,2.5) -- (14.2,2.5) -- (14.2,2.1) -- cycle ;
			\draw (15,2.35) node {$\to L^{j}$};
			\draw[thick] (14,1.6) circle (0.2);
			\draw (14,1.6) node {\small $-$};
			\draw (15,1.65) node {$\to I_-^{j}$};
			\draw[thick] (14,0.9) circle (0.2);
			\draw (14,0.9) node {\small $+$};
			\draw (15,0.95) node {$\to I_+^{j}$};
			\filldraw[thick] (14,0.2) circle (2pt);
			\draw (15,0.25) node {$\to D^{j}$};
		\end{tikzpicture}
	\end{center}
	\captionsetup{width=.9\linewidth}
	\caption{\footnotesize Example of an admissible sequence of partitions of $\lint k+1 \rint$ with $k=11$. The successive partitions $(L^j,I_-^j,I_+^j,D^j)_{1\leq j\leq m}$ (with $m=4$) are represented from top to bottom.
		Each curved arrow represents an element of $I^j_\times$ and points from an index in $L_{\times}^j$ to an index in~$I_{\times}^j$.
		The symbol~\protect\tikz[scale=0.7]
		\protect\draw[thick,dash pattern={on 1.55 off 1.55}] (0.7,1.9) -- (0.7,1.7) -- (1.3,1.7) -- (1.3, 1.9);
		represents indices that can be chosen to belong to $L^{(j+1)}$
		(\,\protect\tikz[baseline=-3pt,scale=0.7]{
		\protect\draw[thick]  (-0.2,-0.2) -- (-0.2,0.2) -- (0.2,0.2) -- (0.2,-0.2) -- cycle ;}\,), 
		$I_-^{(j+1)}$
		(\!\protect\tikz[baseline=-3pt,scale=0.7]{
			\protect\draw[thick] (0,0) circle (0.2);
			\protect\draw (0,0) node {\small $-$};}\!)
		or $I_+^{(j+1)}$
		(\!\protect\tikz[baseline=-3pt,scale=0.7]{
			\protect\draw[thick] (0,0) circle (0.2);
			\protect\draw (0,0) node {\small $+$};}\!).
In the proof of Proposition~\ref{prop:dekz} we first integrate with respect to variables with indices in $J_1$, proceeding in the order   $L^1_{\times},    L^2_{\times}, \dots,   L^{m-1}_\times, \{ i\in L^{m} : \ i(+1,m)\in L^{m}\} $, and afterwards   integrate with respect to variables with indices in $J_2$. }
\label{fig:partitions}
\end{figure}

Let us now determine, or rather bound from above, the cardinality of $\mathfrak{P}_k$, that is, the number of admissible sequences that can be constructed according to the above rules.
Note that we have $2\times3^k$ possibilities for choosing $L^1$, $I^1_+$ and $I^1_-$. Afterwards, at each step, one has to assign one out of three labels $I^{(j+1)}_-$, $I^{(j+1)}_+$ or $L^{(j+1)}$ to  each element in $I^j_\times$. Since $\lvert I^j_\times\rvert=\lvert L^j_{\times}\rvert$ and  the  $L^j_{\times}$'s are disjoint by construction, there are at most $k$ choices to be made, hence at most $3^k$ possibilities. We therefore have $\lvert\mathfrak{P}_k\rvert\le 2\times9^k$.

\smallskip

Next, let us explain how each sequence in  $\mathfrak{P}_k$ is associated to a subset  $\cP\subseteq \calx^{(k)}_t$.
With the convention $(L^{(1)},I^{(1)}_-, I^{(1)}_+):=(L^1,I^1_-,I^1_+)$, we define 
\begin{equation}\label{defcalp}\mathtoolsset{multlined-width=0.9\displaywidth}\begin{multlined}
	\ind_{\cP}(\bt,\bz):= \prod_{j=1}^m \Bigg(\prod_{i\in I_-^{(j)}} \bone_{\{ z_i \geq  (\Delta_j t_i )^{d/6}\wedge \eta,\, z_{i(-1,j)}<  (\Delta_{j} t_i )^{d/6}\wedge \eta\}} 
\prod_{i\in I_+^{(j)}} \bone_{\{ z_i\wedge z_{i(-1,j)} \geq (\Delta_j t_i )^{d/6}\wedge \eta   \}} \\
	\times \prod_{i\in L^{(j)}} \bone_{\{  z_i <  (\Delta_j t_i  )^{d/6}\wedge \eta \}}\Bigg),
	\end{multlined}
\end{equation}
where $\Delta_j t_i\ceq t_i-t_{i(-1,j)}$ (recall \eqref{predesuc}) and  $z_{k+1}:=\infty$ by convention.

It is easy to verify that \eqref{defcalp} induces a partition of  $\calx^{(k)}_t$ indexed by $\mathfrak{P}_k$: Indeed, for any fixed $(\bt,\bz)\in\calx^{(k)}_t$, the values of $z_i$ and $\Delta t_i$ uniquely determine $(L^1, I^1_-, I^1_+)=(L^{(1)},I^{(1)}_-, I^{(1)}_+)$. Assuming that $(L^j, I^j_-, I^j_+)$ has been identified up to some $j\geq1$, we can check whether \eqref{eq:m} is satisfied. If so, we set $m=j$ and we are done. Otherwise, we again use the values of $z_i$ and $\Delta t_i$ to find $(L^{(j+1)}, I^{(j+1)}_-, I^{(j+1)}_+)$, from which we can then determine $(L^{j+1}, I^{j+1}_-, I^{j+1}_+)$. Therefore, every $(\bt,\bz)$ is contained in some $\calp\in\mathfrak{P}_k$. The uniqueness of $\calp$ is straightforward.

\subsubsection{Proof of Proposition~\ref{prop:dekz}}
Recall the formula~\eqref{defcalu} of $U(\cP,t)$.
The proof is divided into two parts: first we integrate with respect to variables with indices in $J_1$ and then with respect to those with indices in $J_2$.

\medskip
\noindent
{\it Step 1. Integrating with respect to indices in $J_1$.}
Our first task is to prove that, roughly speaking, the integral with respect to each $z_i$ and $t_i$ with indices $i\in J_1$ yields at most a factor  $\zeta_1(\eta,p,t)^{k_1}$. More precisely, we prove that
\begin{equation}\label{redoxJ2}
 U(\cP,t)\le\zeta_1(\eta,p,t)^{k_1} U_2(\cP,t),
\end{equation}
where
\begin{equation*}
 U_2(\cP,t):= t^{1-\nu_p}\int_{\calx_t^{(k_2)}} \ind_{\cP'} \ (\Delta_{J_2}t_{k+1} )^{\nu_p-1}\prod_{i\in J_2} (\Delta_{J_2} t_i)^{\nu_p-1}
	z^p_i (3\lvert\log  z_{i}\rvert+1 )^{\ind_{\{z_i< \eta\}}}\,\dd t_i\,\gl(\dd z_i)
\end{equation*}
and 
\begin{equation}
\label{defcalprime}
	\ind_{\cP'}  := \prod_{i\in I_-^{(m)}} \bone_{\{ z_i \geq  (\Delta_{J_2} t_i )^{d/6}\wedge \eta\}} 
\prod_{i\in I_+^{(m)}} \bone_{\{ z_i\wedge z_{i(-1,J_2)} \geq (\Delta_{J_2} t_i )^{d/6}\wedge \eta   \}} 
	 \prod_{i\in L^{(m)}\cap J_2} \bone_{\{  z_i <  (\Delta_{J_2} t_i  )^{d/6}\wedge \eta \}}
\end{equation}
represents constraints on the values of $\Delta_{J_2} t_i$ and $z_i$ that are inherited from $\cP$ for $i\in J_2$.
Note that compared to \eqref{defcalp},  we only have $j=m$ and we have replaced 
$\Delta_{m} t_i$ by $\Delta_{J_2} t_i$. This makes no difference when $i\in I^{(m)}_-$ or $i\in I^{(m)}_+$ since $i(-1,m)\in J_2$  in both cases; for $i \in  L^{(m)}$, the constraint is implied by $\cP$ since
$\Delta_{J_2} t_i\ge \Delta_{m} t_i$.

\smallskip
First, let us reduce to the case where $m=0$,  that is,  $L^1_{\times} =\emptyset$, :
the idea is to integrate with respect to variables with indices in $ L^1_{\times}$, then with respect to those with indices in $L^2_{\times}$, etc. One may refer to Figure~\ref{fig:partitions} to understand how this procedure goes on.
We start with  integration with respect to $z_i$ and $t_i$  when $i\in L^1_{\times}$.
In this case, we have
$ z_i< [\Delta t_i \wedge  \Delta t_{i+1}]^{d/6}\wedge \eta$. 
To treat  these indices, note that  by symmetry,
\begin{equation}\label{eq:samspi}
	\int^{t}_0s^{-1} (t-s)^{-1}  \ind_{\{ z< [s\wedge (t-s)]^{d/6}\}}\, \dd s \le 2 \int^{t/2}_0 s^{-1} (t-s)^{-1}\ind_{\{ z< s^{d/6}\}}\,\dd s \le 4t^{-1} \log\Big( \tfrac12 tz^{-\frac 6d} \Big)\,,
\end{equation}
where for the last bound we used that   $t-s\ge t/2$.
From now on, we assume $\eta\le 1$. Then, if
 $z\in (0,\eta)$,  using also   that $6/d\leq 3$,
we have 
$$
\int^{t}_0s^{-1} (t-s)^{-1}  \ind_{\{ z< [s\wedge  (t-s)]^{d/6}\}} \,\dd s  \le  8t^{-1}  (3\lvert \log z\rvert+1 )(1+\log_+ t)\,.
$$ 
Recalling the definition \eqref{eq:defzezeta} of $\zeta_1(\eta, p, t)$, we end up with
\begin{equation}
\label{eq:samspibis}
	\int^{t}_0\int_{(0,\eta)} s^{-1} (t-s)^{-1}  \ind_{\{ z< [s\wedge  t-s]^{d/6}\}}z^{1+\frac{2}{d}}\, \dd s\,\la(\dd z) \le
	t^{-1} \zeta_1(\eta,p,t)^{\frac{1+2/d}{p}}.
\end{equation}
So if $i\in L^1_{\times}$, using the trivial fact that $\Delta^2 t_i:= t_{i+1}-t_{i-1}\le t$, we have 
\begin{equation}\label{eq:samspi2}
	 \int^{t_{i+1}}_{t_{i-1}} \int_{(0,\infty)} (\Delta t_i \Delta t_{i+1})^{-1} z_i^{1+\frac{2}{d}}\ind_{\cP}(\bt,\bz)\,\dd t_i\,\gl(\dd z_i) 
	\\
	\le  \zeta_1(\eta,p,t)^{\frac{1+2/d}{p}} ({ \Delta^2 t_{i}})^{-1}\bone_{\calp^{(i)}}(\bt^{(i)},\bz^{(i)}),
\end{equation}
 where $\bt^{(i)}$ and $\bz^{(i)}$ denote the vectors $\bt$ and $\bz$ with the $i$th coordinate omitted and $\cP^{(i)}$ is obtained from $\cP$ by ignoring all the constraints that involve either $z_i$ or $t_{i}$.
 Thus, at the cost of a multiplicative factor $\zeta_1(\eta,p,t)^{(1+2/d)/p}$
one can, for each  $i\in L^1_{\times}$, simplify the integrals with respect to $\gl(\dd z_i)\,\dd t_i$.
After relabeling the indices, we can then iterate this process, making use of the iterative construction of $\cP$. As a result,  we are left to prove \eqref{redoxJ2} in the case where $L^1_\times=\emptyset$. Since $m=1$ in this case, we drop the superscript $1$ from the notation.

\smallskip

Under the assumption  $L^1_\times=\emptyset$, we have $J_1= \{i\in L :  i+1 \in L\}$. We may treat different segments of $J_1$ separately. Let us assume, for instance, that $\lint i, i+\ell-1\rint \subseteq J_1$, where $\lint i,j\rint:=\{i,i+1,\dots, j\}$. In the same spirit as \eqref{eq:samspi}--\eqref{eq:samspi2}, we have that
\begin{multline*}
	\int^{t_{i+1}}_{t_{i-1}} (\Delta t_i \Delta t_{i+1})^{-1}\ind_{\cP}(\bt,\bz) \,\dd t_i\\
	\le \frac{8(1 + \log_+ t)}{\Delta^2 t_{i}}  (1+ 3\lvert \log  (z_i\wedge z_{i+1})\rvert )   \ind_{\cP^{(i)}}(\bt^{(i)},\bz^{(i)}) \ind_{\{z_i \vee z_{i+1} < (\Delta^2 t_i)^{6/d}\wedge \eta \}} .
\end{multline*}
Iterating this  and noticing that $\lvert \log  (z_i\wedge z_{i+1})\rvert= \lvert \log z_i \rvert \vee \lvert \log z_{i+1}\rvert$
since $z_i,z_{i+1}<\eta\leq 1$, we get  
\begin{equation*}
	\int_{\frakX_\ell(t_{i-1},t_{i+\ell})} \prod_{j=0}^{\ell}(\Delta t_{i+j})^{-1}\ind_{\cP}(\bt,\bz)\, \dd t_i
	\cdots \dd t_{i+\ell-1}  
	\le \frac{( 8 (1 + \log_+ t))^{\ell}}{t_{i+\ell}-t_{i-1}} \times\frac{\prod_{j=0}^{\ell}(3\lvert\log  z_{i+j}\rvert+1)}{ \min_{j\in \lint \ell\rint} (3\lvert \log  z_{i+j}\rvert+1)}. 
\end{equation*}
The second denominator can be ignored since it is larger than $1$. If one then integrates with respect to $z_{i}, \dots , z_{i+\ell-1}$ (recall  \eqref{defcalprime} and the remark afterwards), one obtains 
\begin{multline*}
	\int_{\frakX_\ell(t_{i-1},t_{i+\ell})\times(0,\infty)^\ell} \Biggl(\prod_{j=0}^{\ell}(\Delta t_{i+j})^{-1}\Biggr)\ind_{\cP}(\bt,\bz) \prod_{j=0}^{\ell-1}z^{1+\frac 2d}_{i+j}\,
	\dd t_{i+j} \,\gl(\dd z_{i+j})\\
	\le \ind_{\cP'} \frac{\zeta_1(\eta,p,t)^{\frac{1+2/d}{p}\ell} }{t_{i+\ell}-t_{i-1}} (3\lvert\log  z_{i+\ell}\rvert+1 ).
\end{multline*}
This completes  integration with respect to $(t_i,z_i)$ for $i\in J_1$ (note that $i+\ell-1\in J_1$ implies $i+\ell\in L$ and hence $z_{i+\ell}<\eta$).
Recalling~\eqref{defcalu} and bounding $(3|\log z|+1)^{p/(1+2/d)} \leq 3|\log z|+1$, we obtain   \eqref{redoxJ2}.

\medskip
\noindent
{\it Step 1. Integrating with respect to indices in $J_2$.} To complete the proof, we need to show that 
\begin{equation}\label{bornonu2}
 U_2(\cP,t)\le \zeta_2(\eta,p)^{k_2}\Lambda(k_2,t,p).
\end{equation}
To simplify notation, let us assume that  $J_2=\lint k\rint$ and hence $k_2=k$ (again, this simply amounts to relabeling the vertices).
Recalling the definition~\eqref{eq:defzezeta} of $\zeta_2(\eta,p)$,
the bound \eqref{bornonu2} follows once we have shown that if $(\bt,\bz)\in \cP'$, then 
\begin{equation}\label{groumph}
(\Delta t_{k+1})^{\nu_p-1}\prod_{i=1}^{k} (\Delta t_i)^{\nu_p-1}
	z^p_i
	 \le \mathbf{G}_{p}(\bt)
	\prod_{i=1}^k \bigg(z^p_i \ind_{\{z_i\ge \eta\}}+ z^{1+\frac{2}{d}}_i\ind_{\{z_i< \eta\}}\bigg),
\end{equation}
where, recalling the definition~\eqref{eq:Gp} of $G_p(s)$, we set
\begin{equation*}
 \mathbf{G}_{p}(\bt):= \begin{cases}
                         (\Delta t_{k+1} )^{\nu_p-1} \prod_{i\in \lint k\rint}G_{p}(\Delta t_i) &\text{ if } k\notin L,\\
                        (\Delta t_k )^{\nu_p-1} \prod_{i\in \lint k+1\rint \setminus \{k\}}G_{p}(\Delta t_i) &\text{ if } k\in L.
                       \end{cases}
\end{equation*}
Note that in both cases, $t^{1-\nu_p}$ times the integral of $\mathbf{G}_{p}(\bt)$ over the simplex $\mathfrak{X}_k(t)$ 
is equal to $\Lambda(k,t,p)$ as defined in~\eqref{eq:definingl}   (in the second case one simply needs to exchange the role of $\Delta t_k$ and $\Delta t_{k+1}$).

  The reason why~\eqref{groumph} is needed is that, on the left-hand side of \eqref{groumph}, the exponent $p$ makes $z^p_i$ \textit{a priori} not integrable near zero. But since $\nu_p-1>-1$, there is some margin for the integrability of  $(\Delta t_i)^{\nu_p-1}$: the idea is to use the constraints in $\cP'$ to ``transfer'' a part of the exponent of $\Delta t_i$ onto that of  $z_i$ or $z_{i-1}$ or both. 

To this end, we shall use the equivalence
\begin{equation}\label{prototype}
 (\Delta t)^{\frac d6} \leq  z \quad \iff \quad    (\Delta t)^{\frac{\nu_p}3} \leq  z^{1+\frac{2}{d}-p}
\end{equation} 
in  two cases:
(i)  if $i\in I_-\cup I_+$ and $z_i< \eta$, with $z=z_i$ and $\Delta t=\Delta t_i$;
(ii)  if $i \in L$ (hence $z_i<\eta$), with  $z=z_i$ and $\Delta t=\Delta t_{i+1}$.
 The reader can check that for $(\bt,\bz)\in \cP'$,  
  the left-hand side of \eqref{prototype} is satisfied in these two cases.
 Notice now  that if $i\in L$, then necessarily we have $i+1\in I_+$: indeed, $i+1$ cannot be in $I_-$ since $L^1_{\times}=\emptyset$ and it cannot be in $L$, either, since $i\in J_2$ by assumption. Therefore,
we obtain 
\begin{equation}\label{groumph2}\mathtoolsset{multlined-width=0.9\displaywidth}\begin{multlined}
 (\Delta t_{k+1} )^{\nu_p-1}\prod_{i=1}^k (\Delta t_i)^{\nu_p-1}
	z^p_i  \le \prod_{i\in I_-} (\Delta t_i)^{\nu_p-1-\frac{\nu_p}{3}\ind_{\{z_i<\eta\}}}  \prod_{i\in I_+} (\Delta t_i)^{\nu_p-1-\frac{\nu_p}{3}(\ind_{\{z_i<\eta\}}+\ind_{\{i-1\in L\}})}
	  \\ 
	\times \prod_{i\in L} (\Delta t_i)^{\nu_p-1}	  \prod_{i=1}^k
	\bigg(z^p_i \ind_{\{z_i\ge \eta\}}+ z^{1+\frac{2}{d}}_i\ind_{\{z_i< \eta\}}\bigg).
 \end{multlined}\end{equation}
The bound \eqref{groumph} now follows from the estimate
 \begin{equation}\label{ultimecompa}
 	\prod_{i\in I_-} (\Delta t_i)^{\nu_p-1-\frac{\nu_p}{3}\ind_{\{z_i<\eta\}}}	 \prod_{i\in I_+} (\Delta t_i)^{\nu_p-1-\frac{\nu_p}{3}(\ind_{\{z_i<\eta\}}+\ind_{\{i-1\in L\}})}
	  \prod_{i\in L} (\Delta t_i)^{\nu_p-1}\le \mathbf{G}_{p}(\bt),
	  \end{equation}
  which is a consequence of the  following three observations concerning the exponent of $\Delta t_i$ on the left-hand side of \eqref{ultimecompa}:
\begin{itemize}
 \item It is larger than or equal to $\frac{\nu_p}{3}-1$ (which we use if  $\Delta t_i\le 1$).
 \item It is equal to $\nu_p-1$ if $\Delta t_i\ge 1 >\eta$ (as can be checked from the definition \eqref{defcalprime} of $\cP'$).
 \item  It is equal to $\nu_p-1$ if $i=k$ and $k\in L$ 
 or if $i=k+1$ and $k\notin L$ (this follows from the
 convention $z_{k+1}=\infty$).
\end{itemize}
This completes the proof of \eqref{bornonu2}, which, together with \eqref{redoxJ2}, implies \eqref{eq:finalbound}.
\qed

\section{Proof of Theorem~\ref{thm:local} and upper bounds in Theorems~\ref{thm:thinup}  and \ref{thm:heavyup}}\label{sec:deducing}

\subsection{Proof of Theorem~\ref{thm:local}}\label{sec:51}

First let us note that the results for the free-end partition function follows from those for the point-to-point version  and Jensen's inequality, as observed in \eqref{supersub}.
The $L^p$-convergence \eqref{eq:lpconv} for $p> 1$ is a direct consequence of Corollary \ref{labornedesmoments} and the fact that 
$(\cZ^{\go,a}_{\beta}(t,x))_{a\in(0,1]}$ is a time-reversed martingale for the filtration $\cG$ defined in \eqref{filtraG}.
The convergence in $L^1$ when $\mu <\infty$ cannot be obtained in this manner but can be recovered by a truncation argument. More precisely, we consider the partition function
$\cZ^{\go,[a,b)}_{\beta}(t,x)$ obtained by replacing $\go$ by
$\go\cap  ( \bbR\times \bbR^d \times [0,b) )$ (see \cite[Equation (2.40) and below]{BL20_cont}) and reproduce the argument used in \cite[Prop.~4.6 and 4.7]{BL20_cont} but with a moment of order $p\in (1,1+\frac{2}{d})\cap(1,2]$ instead of $2$. 

\smallskip
Let us move to the proof of \eqref{eq:condconv}.
	Recall that $\go_<$ and $\go_\geq$ denote the set of points in $\om$ with third coordinate in $(0,1)$ and $[1,\infty)$, respectively.
	Consider the reduced  partition function 
	\begin{equation*}
		\cZ_{\gb}^{\go_<,a}(t,x) =\rho(t,x)+ \sum_{k=1}^{\infty}  \gb^k \int_{ \frakX_k(t) \times (\bbR^d)^k}   \rho_{t,x}( \bt , \bx)   \prod_{i=1}^k \xi_{\go_<}^a (\dd t_i , \dd x_i ).
	\end{equation*}
It is easy to obtain that
	\begin{equation}\label{eq:decompobs}
		\cZ_{\gb}^{\go,a}(t,x)= \cZ^{\go_<,a}_{\beta}(t,x)+
		\sum_{k=1}^{\infty}  \gb^k \int_{ \frakX_k(t) \times (\bbR^d)^k}    \cZ^{\go_<,a}_{\beta,t,x}( \bt , \bx)   \prod_{i=1}^k \xi_{\go_\geq}^a (\dd t_i , \dd x_i ),
	\end{equation}
	where we have defined
	\begin{equation*}
		\cZ^{\go_<,a}_{\beta,t,x}(\bt , \bx):= \prod_{i=1}^{k+1}\cZ^{\go_<,a}_{\beta}(t_{i-1}, x_{i-1};  t_i, x_i)\,,
	\end{equation*}
 with the convention $x_0:=0$, $x_{k+1}:=x$, $t_0:=0$, $t_{k+1}:=t$.
	As a consequence, using  Minkowski's inequality, the identity~\eqref{eq:theinlaws} and Corollary \ref{labornedesmoments}, we have for any $p\in (1,1+\frac2d)$,
	\begin{multline*}
		\bbE_{<} \Big[ \cZ^{\go,a}_{\beta}(t,x)^{p} \Big ]^{\frac1p}
		\\\le   C(\beta,p,t)\rho(t,x) +\sum_{k=1}^{\infty} \gb^k  \int_{ \frakX_k(t) \times (\bbR^d)^k}   \prod_{i=1}^{k+1}\bbE_<\Big[ \cZ^{\go_<,a}_{\beta}( \Delta t_i , \Delta x_i)^{p}\Big]^{\frac1p}   \prod_{i=1}^k \xi_{\go_\geq}^a (\dd t_i , \dd x_i )  .
	\end{multline*}
	Using  Corollary \ref{labornedesmoments}  again,
	we obtain  
	\begin{equation*}
		\bbE_<\Big[ \cZ^{\go_<,a}_{\beta}( \Delta t_i , \Delta x_i)^{p}\Big]^{\frac1p} 
		\le C(\beta,p,t) \rho(\Delta t_i, \Delta t_i)
	\end{equation*}
	 for every $a\in (0,1)$. Setting $\beta':=\beta C(\beta,p,t)$, we conclude that  
	\begin{equation}\label{largatom}
		\bbE_<\Big[ \cZ^{\go,a}_{\beta}(t,x)^{p} \Big]^{\frac1p}
		\le C(\beta,p,t) \cZ^{\go_\geq}_{\beta'}(t,x)<\infty
	\end{equation}
for every $a\in (0,1)$.
	As $\cZ^{\go,a}_{\beta}(t,x)$ is a reversed martingale in $a$ under $\bbP_<$, this shows \eqref{eq:condconv}. For the proof of the positivity statement in \eqref{eq:conv-2}, we refer to \cite[Section~4.7]{BL20_cont}. 
\qed


\subsection{Proof of Theorem \ref{thm:thinup} (first half: upper bounds)} \label{sub:fh}
In this section, 
we prove all  upper bounds on 
$|\bar \gamma_{\beta}(p)|$.
By convexity and the fact that $\bar \gamma_{\beta}(1)=0$,
we have for $q\in (0,1)$ and $p\in (1,1+\frac{2}{d})$ that
\begin{equation}\label{convexinho}
\frac{q-1}{p-1} \bar \gamma_{\beta}(p) \le  \bar \gamma_{\beta}(q) \le 0. 
\end{equation}
Hence
it is sufficient to prove only an upper bound on $\bar \gamma_{\beta}(p)$  for $p\in (1,1+\frac{2}{d})$.
Furthermore, by Proposition~\ref{prop:finitemom} \textit{(ii)}, we can consider the point-to-point partition function in our computations.

\subsubsection{Dimension $d=1$}
If $d=1$ and $p=2$, the first part of Proposition \ref{prop:localbisdone} immediately yields
\begin{equation*}
\bar \gamma_{\beta}(2)  = \lim_{t\to \infty} \frac{1}{t} \log \bbE\Big[ \rho(t,x)^{-2}\bar \calz^{\go}_{\beta}(t,x)^2\Big]  = \frac{\beta^4 \mu_{0,\infty}(2)^2}{4}.
\end{equation*}
The result for other values of $p\in(1,2)$
follows by convexity, which gives $ 0\leq \bar \gamma_{\beta}(p)\le (p-1)\bar \gamma_{\beta}(2)$.

If $p\in(2,3)$, we use the bound~\eqref{mobdd2} for some $t=t(\beta)$.
Indeed, by sub-additivity of $\log \E[ \bar \calz^{\om}_{\beta}(t,\ast)^p ]$ (see Lemma~\ref{lem:submult}),
we get that
\[
\bar \gamma_{\gb}(p) = \lim_{t\to\infty} \frac1t \log \E\Big[ \bar \calz^{\om}_{\beta}(t,\ast)^p\Big] = \inf_{t>0}  \frac1t \log \E\Big[ \bar \calz^{\om}_{\beta}(t,\ast)^p\Big] \,.
\]
Using~\eqref{supersub}, we therefore get that for any $p>1$ and any $t>0$,
\begin{equation}
\label{eq:gammasub}
\bar \gamma_{\beta}(p) \leq \frac1t \log  \E\Big[ \rho(t,0)^{-p} \bar \calz^{\om}_{\beta}(t,0)^p\Big] \,.
\end{equation}
If $t = t(\gb) = c_{p} \gb^{-4}$ with a sufficiently small constant $c_p$,
 the bound~\eqref{mobdd2} gives that
\[
\begin{split}
\E\Big[ \rho(t,0)^{-p} \bar \calz^{\om}_{\beta}(t,0)^p\Big] & \leq C_p \sum_{k_2=0}^{\infty} \frac{(C_p \beta t^{\frac14})^{k_2}}{\Gamma(\frac{k_2+1}{2})^{\frac12}} \Bigg( \sum_{\ell=0}^{\infty} \frac{ (C_p' \gb t^{\frac{\nu_p}{p}})^{\ell} }{\Gamma(\nu_p (\ell+1))^{\frac1p}} \Bigg)^{k_2+1} \\
& \leq  2C_p \sum_{k_2=0}^{\infty} \frac{(2C_p \beta t^{\frac14})^{k_2}}{\Gamma(\frac{k_2+1}{2})^{\frac12}} \leq 4C_p \,.
\end{split}
\]
Indeed, since $\frac1p\nu_p <\frac14$ for $p\in (2,3)$, we can choose $c_p$ sufficiently small so that both the internal sum over $\ell$ and the sum over $k_2$ in the second line are  bounded by $2$.
Altogether, thanks to~\eqref{eq:gammasub}, we get   $\bar \gamma_{\gb}(p) \leq C''_p \gb^4$ with $C''_p= c_p^{-1}\log 4 C_p$.

%
%
%
%

\subsubsection{Dimension $d\geq 2$}

If $d\ge 2$, we are going to  use Proposition~\ref{prop:localbis} with $\eta=1$ and drop the dependence in $\eta$ in the notation.
Our first task is to replace the functions $\zeta_1(p,t)$, $\zeta_2(p)$ and $\Lambda(k_2,t,p)$ in \eqref{compinek} by something more tractable.
First, we let $C_1$ and $C_2$ be constants, which only depend on the measure $\gl$ and are chosen to be larger than $1$, such that for every $p\in  [1+\frac{1}{d},1+\frac 2 d  )$ and  $t\ge e$, we have
\begin{equation}\label{settingcc}
    \bigg(\frac{C}{p-1}\bigg)^{k} \zeta_1(p,t)^{\frac{k_1}p}  \zeta_2(p)^{\frac{k_2}p} 
	\le \Big(C_1 (\log t)^{\frac{1}{1+2/d}}\Big)^{k_1}  C_2^{k_2},
	\end{equation}
where    $C$  is the constant  in \eqref{compinek}.
For $\Lambda(k_2,t,p)$, we use the following lemma, which we prove afterwards.

\begin{lemma}\label{lem:convenient}
For any $k\ge 0$, $p\in(1,1+\frac 2d)$, $t\ge 1$ and $\alpha\in (0,1)$,
\begin{equation*}
 \Lambda(k,t,p)\le(k+1)^2 
  ( 4\nu_p^{-1} \alpha^{-\nu_p})^{k}e^{\alpha t}.
\end{equation*}
\end{lemma}

Using Proposition \ref{prop:localbis}, \eqref{settingcc} and Lemma \ref{lem:convenient} (and the trivial bound $(k+1)^{2/q}\leq (k+1)^2$), we have
for any $\alpha>0$ and $q\in (1,1+\frac2d)$
\begin{equation}\label{sumk1k2}
	\bbE\Big[ \rho(t,0)^{-q}\bar \cZ^{\go}_{\beta}(t,0)^q\Big]^{\frac1q}\le  e^{\frac{\alpha}{q} t} \sum_{k_1= 0}^\infty(C_1\beta  (\log  t)^{\frac{1}{1+ {2}/{d}}})^{k_1}  \sum_{k_2= 0}^\infty (k_2+1)^2
	  \bigg( 4^{\frac1q}C_2\beta \nu_q^{-\frac1q} \alpha^{-\frac{\nu_q}q}\bigg)^{k_2} .
\end{equation}
Now, for sufficiently small $\gb$ we choose $q \in (1+\frac1d,1+\frac2d)$, $\alpha\in(0,1)$, $t\geq e$  in the following way (the dependence in $\beta$ may sometimes be omitted in the computations below, for readability):
\begin{equation}\label{parachoiz}
q(\gb) := 1+\frac2d-\frac{8 e}d(2C_2\beta)^{1+\frac 2d} \,, \quad 
 \al(\gb):=(4(2C_2\beta)^{q}\nu_{q}^{-1})^{\frac1{\nu_{q}}}  \quad \text{ and } \quad 
 t(\beta):=e^{(2C_1\beta)^{-(1-\frac 2d)}}\,.
\end{equation}
Note that with this choice we have $\nu_q=4 e(2C_2\beta)^{1+\frac{2}{d}}$ and thus
\begin{equation}\label{parachoiz2}
 \alpha = (2C_2\beta)^{-\frac 2 d }\exp\bigg( -\frac{1}{4e(2C_2\beta)^{1+\frac{2}{d}}} \bigg)\,,
\end{equation}
\begin{equation}\label{geocontrol}
 C_1\beta (\log  t)^{\frac{1}{1+ {2}/{d}}}= \frac{1}{2} \qquad \text{and} \qquad  4^{\frac1q}C_2\beta \nu_q^{-\frac1q} \alpha^{-\frac{\nu_q}q} = \frac12\,,
\end{equation}
so that the sums in $k_1$ and $k_2$ in \eqref{sumk1k2} are respectively equal to $2$ and $12$. We end up with 
\begin{equation}\label{specifiq}
 \bbE\Big[\rho(t,0)^{-q} \bar \cZ^{\go}_{\beta}(t,0)^q\Big]^{\frac1q}\le 24\, e^{\frac{\alpha}q t} \,.
\end{equation}

%
For a fixed $p\in(1,1+\frac{2}{d})$, we can now deduce the upper bound in \eqref{eq:1+2d} from \eqref{specifiq}. We consider~$\beta$ sufficiently small so that $q=q(\beta)\ge p$.
 Using first \eqref{eq:gammasub} and then Jensen's inequality, we get
\begin{equation}\label{check}
 \bar \gamma_{\beta}(p)\le   \frac{1}{t}\log \bbE\Big[\rho(t,0)^{-p} \bar \cZ^{\go}_{\beta}(t,0)^p\Big]
 \le \frac{p}{ t}\log \bbE\Big[\rho(t,0)^{-q} \bar \cZ^{\go}_{\beta}(t,0)^q\Big]^{\frac1q}
 \le \frac{p\log 24}{ t} + \frac{p\alpha}{q}. 
\end{equation}
Using \eqref{parachoiz2} and replacing $t$ by its value \eqref{parachoiz} yield the upper bound in  \eqref{eq:1+2d}  since $p/q\le 1$.
Let us finally comment on our choice \eqref{parachoiz}. The parameters $\alpha$ and $t$ have been chosen so that~\eqref{geocontrol}, hence \eqref{specifiq}, holds. The value of $q=q(\beta)$ is chosen so that $\al$, seen as a function of $q$ for fixed~$\beta$, is minimized.
This completes the proof of the upper bounds in \eqref{eq:d1} and \eqref{eq:1+2d}.  \qed

\medskip 
The lower bounds in \eqref{eq:d1} and \eqref{eq:1+2d} will   be shown after Proposition~\ref{prop:dgeq2}.

\begin{proof}[Proof of Lemma \ref{lem:convenient}]
Recall the definitions~\eqref{eq:definingl} and \eqref{eq:Gp} of $\Lambda(k,t,p)$ and $G_p(s)$.	For fixed $j\in \lint k\rint$, using a symmetry argument, we have that 
	\begin{multline*}
		\int_{\mathfrak{X}_k(t)} \ind_{\{\Delta t_{k+1} = \max\limits_{i\in \lint k+1\rint}\Delta t_{i}\}} (\Delta t_{k+1})^{\nu_p-1}\prod_{i=1}^{k}G_p(\Delta t_{i}) \,\dd t_i\\
		= \int_{\mathfrak{X}_k(t)} \ind_{\{\Delta t_{j} = \max\limits_{i\in \lint k+1\rint}\Delta t_{i}\}} (\Delta t_{j})^{\nu_p-1}\prod_{i\in \lint k+1\rint \setminus \{ j\}}G_p(\Delta t_{i}) \,\dd t_i\\
		\ge \int_{\mathfrak{X}_k(t)} \ind_{\{\Delta t_{j} = \max\limits_{i\in \lint k+1\rint}\Delta t_{i}\}} (\Delta t_{k+1})^{\nu_p-1}\prod_{i\in \lint k\rint}G_p(\Delta t_{i}) \,\dd t_i\,,
	\end{multline*}
where we have used that $(\Delta t_{k+1})^{\nu_p-1}G_p(\Delta t_j)\leq (\Delta t_j)^{\nu_p-1}G_p(\Delta t_{k+1})$ if $\Delta t_{k+1}\leq \Delta t_j$ for the second step.
As a consequence, again by a symmetry argument, we have
 \begin{equation*}
  \Lambda(k,t,p)\le   (k+1)  \, t^{1-\nu_p}
 \int_{\mathfrak{X}_k(t)} \ind_{\{\Delta t_{k+1} = \max\limits_{i\in \lint k+1\rint}\Delta t_{i}\}} (\Delta t_{k+1})^{\nu_p-1}\prod_{i=1}^{k}G_p(\Delta t_{i}) \,\dd t_i\,.
 \end{equation*}
Together with the inequality $\max_{i\in \lint k+1\rint}\Delta t_{i}\ge t/(k+1)$, this implies
\begin{equation*}
  \Lambda(k,t,p)\le (k+1)^{2-\nu_p} \int_{\mathfrak{X}_k(t)}  \prod_{i=1}^{k}G_p(\Delta t_{i})\,\dd t_i.
 \end{equation*}
We can conclude our proof by observing that 
 \begin{equation*}
\int_{\mathfrak{X}_k(t)}  \prod_{i=1}^{k}G_p(\Delta t_{i})\,\dd t_i \le e^{\alpha t} \bigg(
	\int^{\infty}_{0} G_p(s)e^{-\alpha s} \,\dd s  \bigg)^{k}
 \end{equation*}
and
\begin{equation*}
	\int^{\infty}_{0} G_p(s)e^{-\alpha s}\,\dd s \le   \int^{1}_{0} s^{\frac13\nu_p-1}\,\dd s+ \int^{\infty}_{0} s^{\nu_p-1}e^{-\alpha \nu_p}\,\dd s
	= 3\nu_p^{-1}+\gG(\nu_p)\alpha^{-\nu_p}\le 4\nu_p^{-1}
	\alpha^{-\nu_p}\,,
\end{equation*}
where the last step is valid because $\nu_p \in (0,1)$ and $\alpha \in (0,1)$.
\end{proof}

\subsection{Proof of Theorem \ref{thm:heavyup} (first half: upper bounds)}
\label{sub:fattails}
Again, by \eqref{convexinho} it suffices 
 to prove~\eqref{sioum}  for $p> 1$. Moreover, \eqref{eq:liminf} follows immediately from \eqref{sioum}.
 
 \smallskip
If $d\ge 2$,  the arguments from Section~\ref{sub:fh} remain valid if we  let $q$  be equal to the value corresponding to our assumption $\mu_{1,\infty}(q)<\infty$ (we consider $C_1$ and $C_2$ such that \eqref{settingcc} is valid for this value of $q$) but retain the choice in \eqref{parachoiz} for $\alpha$ and $t$.
The conclusion \eqref{check} is still valid for any $p\in(1,q]$, and because $\alpha(\beta)= C_3 \beta^{q/\nu_q}$,
 \eqref{sioum} follows.

 \smallskip

If $d=1$, it suffices by convexity to treat the case when $p=q$. 
By \eqref{mobdd1} (with $\eta=1$) and our assumptions, there are constants $C_1$ and $C_2$, which may depend on $q$, such that
\begin{equation}
\label{eq:d=1lighttail}
	\E\Big[\rho(t,x)^{-q}\bar \calz^{\om}_{\beta}(t,x) ^q\Big]^{\frac1q}
	\le 
	  \sum_{k_1=0}^\infty   (C_1 \beta t^{\frac14} )^{k_1} \sum_{k_2=0}^\infty  \frac{  (C_2 \beta t^{\frac{\nu_q}{q}} )^{k_2}}{\Gamma(\nu_q(k_2+1))^{\frac1q}}.
\end{equation}
Let us choose $t=t(\beta):= (2C_1)^{4} \beta^{-4} $ so that the first sum is equal to $2$.

For the second sum,  consider the three-parameter function $F_{\al,\delta}^{(\ga)}(x)=\sum_{m=0}^\infty x^{m}/\Ga(\al m+\delta)^\ga$.  A precise asymptotic of $F_{\al,\delta}^{(\ga)}$ is given  in \cite[Theorem~1]{Gerhold12}, from which one gets that there is  $K_{\al,\delta}^{(\ga)}>0$ such that for all $x\geq0$,
\begin{equation}
 \label{eq:6}
	  F_{\al,\delta}^{(\ga)}(x) \leq K_{\al,\delta}^{(\ga)}e^{2\ga x^{\frac 1 {\alpha\gamma}}}  \,.
\end{equation}
Thanks to this estimate,  the second sum in~\eqref{eq:d=1lighttail} is bounded by $C_3(1+\beta^{c_q}) \exp(\frac1q (C_2\beta)^{q/\nu_q} t)$ with $c_q = (q-1) (\frac{1}{\nu_q}- \frac4q)<0$. 
In view of \eqref{eq:gammasub}, we obtain that 
\[
\bar\gamma_{\gb}(q)\leq  \frac{q\log (2C_3 (1+\gb^{c_q}) )}{t(\beta)} + C_2\beta^{\frac{q}{\nu_q}}  \leq  C_4 \gb^4 \log (1+\gb^{-1}) + C_2 \gb^{\frac{q}{\nu_q}}\,.
\]
Recalling that  $q/\nu_q <4$ if $p<2$, this gives the upper bound in  \eqref{sioum}.
Finally, the lower bound in \eqref{eq:stable} is an immediate consequence of~\eqref{eq:liminf} since our assumption  
$ \lim_{z\to \infty} \log  \gl( [z,\infty))/\log  z = \alpha$ implies  that
$\mu_{1,\infty}(q)<\infty$ for all $q<\alpha$. 
\qed

\medskip
The second half of Theorem~\ref{thm:heavyup} will be shown after Proposition~\ref{prop:dgeq2}.

\section{Existence  and uniqueness of solutions to the SHE}
\label{sec:SHE}
This section is devoted to the proof of Theorems~\ref{thm:SHE} and  \ref{thm:SHEunique}.
The existence part, that is, Theorem~\ref{thm:SHE} is detailed in the Sections \ref{sec:plan}--\ref{sec:geninitial}.
Theorem~\ref{thm:SHEunique}, is addressed afterwards: in Section~\ref{sec:intesol}, we verify that the solution defined in Theorem~\ref{thm:SHE} satisfies the condition~\eqref{eq:cond}, while uniqueness is shown in Section~\ref{sec:unique}.

\subsection{Overview of the proof of Theorem \ref{thm:SHE}}
\label{sec:plan}
	For the ease of exposition, we first present  the case where the initial condition $u_0$ is the Dirac function $\delta_0$. 
	By \eqref{eq:pppf-2}, we have
	\beq\label{eq:zeidentity}
	\mathtoolsset{multlined-width=0.85\displaywidth}\begin{multlined} \calz^{\om,a}_{\beta}(t,x)=\rho(t,x)+\beta\int_{(0,t)\times \bbR^d}\rho(t-s,x-y)\calz^{\om,a}_{\beta}(s,y) \,\xi^a_{\go_<}(\dd s,\dd y)\\
		+\beta\int_{(0,t)\times \bbR^d}\rho(t-s,x-y)\calz^{\om,a}_{\beta}(s,y)\,\xi_{\om_\geq}(\dd s,\dd y).
	\end{multlined}
	\eeq
	We want to let $a$ tend to $0$ and conclude that 
	\beq\label{eq:10}
	\mathtoolsset{multlined-width=0.85\displaywidth}\begin{multlined} \calz^\om_{\beta}(t,x)=\rho(t,x)+\beta\int_{(0,t)\times \bbR^d}\rho(t-s,x-y)\calz^\om_{\beta}(s,y)\,\xi_{\go_<}(\dd s,\dd y)\\
		\quad+\beta\int_{(0,t)\times \bbR^d}\rho(t-s,x-y)\calz^{\om}_{\beta}(s,y)\,\xi_{\om_\geq}(\dd s,\dd y).
	\end{multlined}\eeq
	We have to prove that the two integrals in the on the right-hand side are well defined in It\^o's and Lebesgue's sense, respectively (cf.\ Remark~\ref{rem:Ito}),  that they are finite and that they are the limit of the integrals displayed in \eqref{eq:zeidentity}. For this last point, setting 
	\begin{align*}
		I^a_1(t,x)&:=	\int_{(0,t)\times \bbR^d}\rho(t-s,x-y)(\calz^{\om,a}_{\beta}(s,y)-\calz^{\om}_{\beta}(s,y))\,\xi^a_{\go_<}(\dd s,\dd y),\\
		I^a_2(t,x)&:=	\int_{(0,t)\times \bbR^d}\rho(t-s,x-y)\calz^{\om}_{\beta}(s,y)\,\xi^{[0,a)}_{\go_<}(\dd s,\dd y),\\
		I^a_3(t,x)&:=	\int_{(0,t)\times \bbR^d}\rho(t-s,x-y)(\calz^{\om,a}_{\beta}(s,y)-\calz^{\om}_{\beta}(s,y))\,\xi_{\go_\geq}(\dd s,\dd y),
	\end{align*}
with $\xi^{[a,b)}_{\go}:=\xi^a_\go-\xi^b_\go$ (and $\xi^0_\go:=\xi_\go$), we need to prove that
	for all $(t,x)$ and $j=1,2,3$, the following convergence holds in probability:
	\begin{equation}\label{eq:I}
		\lim_{a\to 0} I^a_j(t,x)=0.
	\end{equation}
	 The finiteness and convergence of these integrals are proved in Section~\ref{sec:fin} and Section~\ref{sec:conv}, respectively. The case of general initial condition is discussed in Section~\ref{sec:geninitial}.

\subsection{Finiteness of stochastic integrals}\label{sec:fin}
We first introduce a technical statement that is a direct consequence of Corollary \ref{labornedesmoments}.

\begin{lemma}
	\label{lem:pint} Assuming \eqref{eq:log} and \eqref{eq:log-2},  for all $p\in(1,1+\frac 2d)$,  $T>0$, $x\in\R^d$ and $\beta >0$,  there exist positive constants $\beta'=\beta'(\beta,p,T)$ and $C=C(\beta,p,T)$ such that for every $t\in [0,T]$,
	\begin{equation*}
	 \bbE_<\bigg[ \sup_{a\in(0,1)} \calz^{\om,a}_{\beta}(t,x)^p \bigg]^{\frac1p} \leq C\calz^{\om_{\geq}}_{\beta'}(t,x)\qquad \text{$\bbP_\geq$-a.s.} 
	\end{equation*}
\end{lemma}

\begin{proof} 
	By Doob's maximal inequality, it is sufficient to bound
	$\E_<[\calz^{\om,a}_{\beta}(t,x)^p]^{1/p}$ uniformly in $a$,~$t$ and $x$.
	With the conventions that $t_{k+1}:=t$, $x_{k+1}:=x$ and the term corresponding to $k=0$ is equal to 
	$\calz^{\om_<,a}_{\beta}(t,x)$, we have that
	\begin{equation}\label{defofW} 
		\calz^{\om,a}_{\beta}(t,x)=    \sum_{k=0}^{\infty}\beta^{k}\int_{\frakX_{k}(t)\times(\R^d)^k}   \prod_{i=1}^{k+1}\calz^{\om_<,a}_{\beta}(t_{i-1},x_{i-1};t_i,x_{i})  \prod_{j-1}^k\xi_{\om_{\geq}}(\dd t_j, \dd x_j). 
	\end{equation}
	Averaging with respect to $\go_<$ and using the Minkowski inequality, translation invariance
	and Corollary \ref{labornedesmoments} (we let $C_0=C_0(\beta,p,T)$ denote the constant obtained from Corollary \ref{labornedesmoments}), we obtain the desired bound as follows:
	\begin{align*}
		\bbE_<\Big[\calz^{\om,a}_{\beta}(t,x)^p\Big]^{\frac1p}&\le\sum_{k=0}^{\infty}\beta^{k}\int_{\frakX_{k}(t)\times(\R^d)^k}  \prod_{i=1}^{k+1} \bbE_<\Big[\calz^{\om_<,a}_{\beta}(\Delta t_i, \Delta x_i)^p\Big]^{\frac1p}  \prod_{j-1}^k\xi_{\om_{\geq}}(\dd t_j, \dd x_j)\\
		 &\le C_0\sum_{k=0}^{\infty}[C_0\beta]^{k}  \int_{\frakX_{k}(t)\times(\R^d)^k} \rho_{t,x}(\bt,\bx) \prod_{j=1}^{k}   \xi_{\om_{\geq}}(\dd t_j, \dd x_j) = C_0\calz^{\om_{\geq}}_{C_0\beta}(t,x).\qedhere
	\end{align*}
	\end{proof}

\begin{rem} 
In this section and the next, the dependence in $\beta,p$ and $T$ of the constant do not have a major importance: most of the time, they are omitted in  computations.
\end{rem}

We now proceed to showing that the integrals in \eqref{eq:10} are well defined and finite. The integral  $I_3^a$ is a Lebesgue integral. Because of \eqref{eq:conv-2}, it clearly suffices to show that 
\begin{equation}\label{eq:zup}
 \int_{(0,t)\times \bbR^d}\rho(t-s,x-y)\sup_{a\in(0,1)}\calz^{\om,a}_{\beta}(s,y)\,\xi_{\om_\geq}(\dd s,\dd y)<\infty.
\end{equation}
In fact, we are going to show that the $L^p(\bbP_<)$-norm of the quantity in \eqref{eq:zup} is finite $\bbP_\geq$-a.s. Using Minkowski's integral inequality and Lemma~\ref{lem:pint}, we have
\begin{equation} 
\label{eq:Mink1}
\begin{split}
 \bbE_< \Bigg[ \bigg(\int_{(0,t)\times \bbR^d}  & \rho(t-s,x-y)\sup_{a\in(0,1)}\cZ^{\om,a}_{\beta}(s,y)\,\xi_{\go_\geq}(\dd s,\dd y )\bigg)^p\Bigg]^{\frac1p}\\
& \le\int_{(0,t)\times \bbR^d}  \rho(t-s,x-y) \bbE_< \bigg[ \sup_{a\in(0,1)} \cZ^{\om,a}_{\beta}(s,y)^p \bigg]^{\frac1p}  \,\xi_{\go_\geq}(\dd s,\dd y)\\
& \leq C \int_{(0,t)\times\R^d} \rho(t-s,x-y)\calz^{\om_\geq}_{\beta'}(s,y)\,\xi_{\om_\geq}(\dd s,\dd y).
 \end{split}
 \end{equation}
By \eqref{eq:zeidentity}, applied with $a=1$, the integral in the last line is equal to $(\calz^{\om_{\geq}}_{\beta'}(t,x)-\rho(t,x))/\beta'$,
which is finite and therefore proves \eqref{eq:zup}.

\smallskip

Let us now deal with   $I_1^a$ and~$I_2^a$.
We show that for $\P_\geq$-a.e.\  realization of $\go_\geq$, the process $(\om_<,s,y)\mapsto\rho(t-s,x-y)\calz^{\om_<\cup \om_\geq}_\beta(s,y)$ is $L^p(\P_<)$-integrable with respect to $\xi_{\om_<}$, in the sense described in Section~\ref{sec:sol}.
Thanks to Lemma~\ref{lem:integrab}, this implies that $(\om,s,y)\mapsto\rho(t-s,x-y)\calz^{\om}_\beta(s,y)$ is $L^0(\bbP)$-integrable with respect to $\xi_{\om_<}$.
To this end, according to \eqref{eq:integ} and \eqref{eq:condizione},
we only need to prove that for some $p\in (1,1+\frac{2}{d})$, we have
\begin{equation}\label{eq:tobeproved}
 \bbE_<\Bigg[ \bigg( \int_{(0,t)\times \bbR^d\times (0,1)}  \rho(t-s,x-y)^2\calz^{\om}_{\beta}(s,y)^2 z^2\, \delta_{\go_<}(\dd s,\dd y, \dd z )     \bigg)^{\frac p2}\Bigg]<\infty.
\end{equation}
Using Fatou's Lemma for the inner integral, we can replace $\calz^{\om}_{\beta}(s,y)$
by $\calz^{\om,a}_{\beta}(s,y)$ and add a restriction $\ind_{\{z\ge a\}}$ provided that the bound we prove is uniform in $a$.
Now using the BDG inequality in the reverse direction, we obtain that 
\beq\label{eq:BDG8}\begin{split}
  &\bbE_<\Bigg[ \bigg( \int_{(0,t)\times \bbR^d\times  (0,1)}  \rho(t-s,x-y)^2\cZ^{\om,a}_{\beta}(s,y)^2 z^2 \ind_{\{z\ge a\}}\,\delta_{\go_<}(\dd s,\dd y, \dd z )     \bigg)^{\frac p2}\Bigg]
  \\ 
  &\qquad\le C\,\bbE_< \Bigg[ \bigg\lvert\int_{(0,t)\times \bbR^d}  \rho(t-s,x-y)\cZ^{\om,a}_{\beta}(s,y)\,\xi^a_{\go_<}(\dd s,\dd y )\bigg\rvert^p\Bigg].
\end{split}\eeq
Now using \eqref{eq:zeidentity} and the Minkowski inequality, we have
\begin{equation}
\label{eq:Mink2}
\begin{split}
  & \beta\bbE_<  \Bigg[ \bigg\vert\int_{(0,t)\times \bbR^d}  \rho(t-s,x-y)\cZ^{\om,a}_{\beta}(s,y)\,\xi^a_{\go_<}(\dd s,\dd y )\bigg\rvert^p\Bigg]^{\frac1p}\\ 
  &\quad \le  \rho(t,x)+ \bbE_< \Big[ \cZ^{\go,a}_\beta(t,x)^{p}\Big]^{\frac1p}
  +\beta\, \bbE_< \Bigg[ \bigg(\int_{(0,t)\times \bbR^d}  \rho(t-s,x-y)\cZ^{\om,a}_{\beta}(s,y)\,\xi_{\go_\geq}(\dd s,\dd y)\bigg)^p\Bigg]^{\frac1p}.
  \end{split}\!\!\!
\end{equation}
We conclude by observing that all summands on the right-hand side are uniformly bounded in $a$:
we use Lemma~\ref{lem:pint} for the second one, while the third one has been controlled in the previous paragraph (see~\eqref{eq:Mink1}). \qed

\subsection{Convergence of stochastic integrals}\label{sec:conv}
We give the proof of \eqref{eq:I}  in this section.
The convergence of $I^a_3$ follows from  \eqref{eq:zup} and dominated convergence.
The convergence of $I^a_2$ follows from dominated convergence for stochastic integrals. 
The main piece of work consists in proving  the convergence of $I^a_1$.
We define
\begin{equation*}
\begin{split}
J^{a,b}_1(t,x)&:= \int_{(0,t)\times \bbR^d} \rho(t-s,x-y)(\calz^{\om,a}_{\beta}(s,y)-\calz^{\om}_{\beta}(s,y))\,\xi^{[a,b)}_{\go_<}(\dd s,\dd y)\,,\\ 
J^{a,b}_2(t,x)& 
:=\int_{(0,t)\times \bbR^d} \rho(t-s,x-y) \big(\calz^{\om,a}_{\beta}(s,y)-\calz^{\om}_{\beta}(s,y)\big)\,\xi^{b}_{\go_<}(\dd s,\dd y) \,,
\end{split}
\end{equation*}
and we prove that the following holds in $\P_\geq$-probability:
\begin{equation*}
\lim_{b\to 0}\sup_{a\in(0,1)}\E_<\Big[\lvert J^{a,b}_1(t,x)\rvert\Big]=0 \qquad \text{and} \qquad
\lim_{a\to 0} \E_<\Big[ \lvert J^{a,b}_2(t,x) \rvert\Big]  =0.
\end{equation*}

\subsubsection*{Convergence of $J^{a,b}_2(t,x)$}
By Jensen's inequality,
it suffices to prove that $\E_<[|J^{a,b}_2(t,x) |^{p}]\to0$
for some $p\in(1,\min(2, 1+\frac 2d))$.
By the BDG inequality, we can further reduce this to proving   
\begin{equation}\label{BDGTVPART}
  \lim_{a\to 0}\bbE_<\Bigg[ \bigg(\int_{(0,t)\times \bbR^d\times [b,1)} \rho(t-s,x-y)^2(\calz^{\om,a}_{\beta}(s,y)-\calz^{\om}_{\beta}(s,y))^2z^2\,\delta_{\go_<}(\dd s, \dd y,\dd z)\bigg)^{\frac p2}\Bigg]=0.
\end{equation}
 By subadditivity \eqref{eq:subbaditiv}, for $a<b$ this quantity  is smaller than
\begin{multline*}
   \bbE_<\Bigg[\int_{(0,t)\times \bbR^d\times [b,1)} \rho(t-s,x-y)^p\lvert\calz^{\om,a}_{\beta}(s,y)-\calz^{\om}_{\beta}(s,y)\rvert^p  z^p\,\delta_{\go_<}(\dd s, \dd y,\dd z)\Bigg]
 \\=\mu_{b,1}(p)  \int_{(0,t)\times \bbR^d} \rho(t-s,x-y)^p\bbE_<\Big[\lvert\calz^{\om,a}_{\beta}(s,y)-\calz^{\om}_{\beta}(s,y)\rvert^p \Big]\,\dd s\,\dd y.
\end{multline*}
As the integrand tends to zero thanks to~\eqref{eq:condconv},
by dominated convergence we will get that the last integral converges to $0$ for $\P_\geq$-a.s.\   if
we show that for $\P_\geq$-a.e.\ realization of $\om_\geq$,
\begin{equation}\label{eq:doomine}
 \int_{(0,t)\times \bbR^d} \rho(t-s,x-y)^p\bbE_<\bigg[\sup_{a\in(0,1)}\calz^{\om,a}_{\beta}(s,y)^p \bigg]\,\dd s\,\dd y<\infty  \,.
\end{equation}
To   this end, we use Lemma~\ref{lem:pint} to bound the left-hand side of \eqref{eq:doomine} by a constant times
\beq\label{eq:doomine2}\mathtoolsset{multlined-width=0.9\displaywidth}\begin{multlined}
\int_{(0,t)\times\R^d} \rho(t-s,x-y)^p\calz^{\om_\geq}_{\beta'}(s,y)^p\,\dd s\,\dd y\\
	 \leq 	\mathtoolsset{multlined-width=0.85\displaywidth}\begin{multlined}[t]  \Bigg(\sum_{k=0}^\infty  (\beta')^k\int_{\frakX_k(t)\times(\R^d)^k} \Bigg(\int_{(t_k,t)\times\R^d} \rho(t-s,x-y)^p\rho(s-t_k,y-x_k)^p\,\dd s\,\dd y\Bigg)^{\frac1p} \\
	\times\prod_{i=1}^k \rho(\Delta t_i,\Delta x_i) \,\xi_{\om_\geq}(\dd t_i,\dd x_i)\Bigg)^p,\end{multlined}
\end{multlined}\eeq
where the last inequality follows from Minkowski's inequality (for the $L^p$-norm of the measure $\rho(t-s,x-y)^p\,\dd s\,\dd y$).
Recalling \eqref{help3} and using Lemma~\ref{lemGamma}, one has
\begin{align}
	\int_{(t_k,t)\times\R^d} \rho(t-s,x-y)^p\rho(s-t_k,y-x_k)^p\,\dd s\,\dd y&   = \vartheta(p)^2 \frac{\Ga(\nu_p)^2}{\Ga(2\nu_p)}  (t-t_k)^{2\nu_p-1}  \rho\big( \tfrac1p (t-t_k),x-x_k \big) \nonumber\\
	&   = \vartheta(p) \frac{\Ga(\nu_p)^2}{\Ga(2\nu_p)} (t-t_k)^{\nu_p }\rho(t-t_k,x-x_k)^p.
	\label{eq:calculrhorho}
\end{align}
Therefore, the right-hand side of \eqref{eq:doomine2} is further bounded by  $\vartheta(p) \Ga(\nu_p)^2\Ga(2\nu_p)^{-1} t^{\nu_p}$  times
\begin{equation}\label{findureasoning}
 \Bigg(\sum_{k=0}^\infty  (\beta')^k\int_{\frakX_k(t)\times(\R^d)^k} \rho_{t,x}(\bt,\bx) \prod_{i=1}^k \xi_{\om_\geq}(\dd t_i,\dd x_i)\Bigg)^p=\calz^{\om_\geq}_{\beta'}(t,x)^p,
\end{equation}
which is finite.  This concludes the proof of \eqref{BDGTVPART}.
%
\qed

\subsubsection*{Convergence of $J^{a,b}_1(t,x)$} Using  Jensen's inequality, we  need to show that for some $p\in (1,1+\frac 2 d )$,
\begin{equation}\label{eq:2lines}\begin{split}
 \lim_{b\to0} \sup_{a\in(0,1)} \bbE_<\Bigg[ \bigg\lvert\int_{(0,t)\times \bbR^d} \rho(t-s,x-y)\calz^{\om,a}_{\beta}(s,y)\,\xi^{[a,b)}_{\go_<}(\dd s, \dd y)\bigg\rvert^{p}\Bigg]=0,\\
  \lim_{b\to0} \sup_{a\in(0,1)} \bbE_<\Bigg[ \bigg\lvert\int_{(0,t)\times \bbR^d} \rho(t-s,x-y)\calz^{\om}_{\beta}(s,y)\,\xi^{[a,b)}_{\go_<}(\dd s, \dd y)\bigg\rvert^{p}\Bigg]=0.
 \end{split}
\end{equation}
For the second line, we apply the  BDG inequality 
and realize that 
\begin{multline*}
 \sup_{a\in(0,1)} \bbE_<\Bigg[ \bigg\lvert\int_{(0,t)\times \bbR^d} \rho(t-s,x-y)\calz^{\om}_{\beta}(s,y)\,\xi^{[a,b)}_{\go_<}(\dd s, \dd y)\bigg\rvert^{p}\Bigg]\\ \le C\,  \bbE_<\Bigg[ \bigg(\int_{(0,t)\times \bbR^d\times [0,b)} \rho(t-s,x-y)^2\calz^{\om}_{\beta}(s,y)^2z^2\,\delta_{\go_<}(\dd s, \dd y,\dd z)\bigg)^{\frac p2}\Bigg].
\end{multline*}
Thanks to \eqref{eq:tobeproved}, this converges to $0$ as $b\to0$, by dominated convergence.
Concerning the first line in \eqref{eq:2lines}, we  rely on the proof of Propositions~\ref{prop:localbisdone} and \ref{prop:localbis} but with a small variation.
Let us define a modified partition function by
$$\mathcal Y^{\go,a,b}_{\beta}(s,x;t,y):= \sum_{k= 1}^\infty  \beta^{k-1}\int_{\mathfrak X_k(s,t)\times (\bbR^d)^k} \rho_{s,x;t,y}(\bt,\bx)\Bigg(\prod_{i=1}^{k-1}\xi^{a}_{\go}(\dd t_i,\dd x_i) \Bigg)\, \xi^{[a,b)}_{\go}(\dd t_k, \dd x_k)$$
and $\caly^{\om,a,b}_\beta(t,x):=\caly^{\om,a,b}_\beta(0,0;t,x)$.
The reader can check that the $\E_<[(\cdots)]$-term in the first line of \eqref{eq:2lines} is simply $\bbE_<[ \lvert\mathcal Y^{\go,a,b}_{\beta}(t,x)\rvert^p]$. 
In analogy with \eqref{eq:decompobs}, we have
\begin{align*} 
\label{eq:decompobsy}
 \mathcal Y^{\go,a,b}_{\beta}(t,x)&=\mathcal Y^{\go_<,a,b}_{\beta}(t,  x)\\
 &\quad+
 \sum_{k=1}^{\infty} \gb^k  \int_{ \frakX_k(t) \times (\bbR^d)^k}    \cY^{\go_<,a,b}_{\beta}(t_k, x_k;t,x)  \prod_{i=1}^k\cZ^{\go_<,a}_{\beta}(t_{i-1},x_{i-1};t_i,x_i) \,\xi_{\go_\geq} (\dd t_i , \dd x_i ).\nonumber
\end{align*}
 Applying Minkowski's inequality, we get
\begin{multline*}
\!\!\!  \bbE_<\Big[  \lvert\mathcal Y^{\go,a,b}_{\beta}(t,x)\rvert^p\Big]^{\frac1p}
 \le \bbE_<\Big[\mathcal Y^{\go_<,a,b}_{\beta}(t,  x)^p\Big]^{\frac1p}+ \sum_{k=1}^{\infty} \gb^k  \int_{ \frakX_k(t) \times (\bbR^d)^k}  
	\bbE_<\Big[\lvert\cY^{\go_<,a,b}_{\beta}(\Delta t_{k+1}, \Delta x_{k+1})\rvert^p \Big]^{\frac1p}\\
 \times \prod_{i=1}^k \bbE_<\Big[\cZ^{\go_<,a}_{\beta}(\Delta t_{i}, \Delta x_{i})^p \Big]^{\frac1p}\,
 \xi_{\go_\geq} (\dd t_i , \dd x_i ). 
\end{multline*}
We then need the following estimate, which comes from a small adaptation of the proofs of Propositions~\ref{prop:localbisdone} and \ref{prop:localbis} and Corollary \ref{labornedesmoments}.  The proof of this lemma is postponed to the end of the section.
\begin{lemma}\label{lem:localter}
Assume that $\gl([1,\infty))=0$. Given $\beta,T>0$, and $p\in (1,1+\frac 2 d)$ 
 there exists a function $\delta(b) =\delta_{\beta,p,T}(b)$ such that $\lim_{b\to 0}\delta(b)=0$ and
 \begin{equation*}
  \sup_{t\in(0,T]} \sup_{x\in \bbR^d} \sup_{a\in(0,1]} \bbE\Big[ \rho(t,x)^{-p} \lvert\mathcal Y^{\go,a,b}_{\beta}(t,x)\rvert^p \Big]^{\frac1p}\le \delta(b).
 \end{equation*}
\end{lemma}

\noindent
By Corollary \ref{labornedesmoments} and Lemma \ref{lem:localter}, we obtain   a constant $C_0=C_0(\beta, p,T)$ such that 
\begin{align*}
 \bbE_<\Big[ \lvert \mathcal Y^{\go,a,b}_{\beta}(t,x)\rvert^p\Big]^{\frac1p}
 &\le \delta(b)\Bigg( \rho(t,x)+  \sum_{k=1}^{\infty} (C_0 \gb)^k \int_{ \frakX_k(t) \times (\bbR^d)^k}   
 \rho_{t,x}(\bt,\bx) \prod_{i=1}^k 
 \xi_{\go_\geq} (\dd t_i , \dd x_i )\Bigg)\\
&= \delta(b) \cZ^{\go_\geq}_{\beta'}(t,x) \,,
 \end{align*}
 with $\beta':=C_0\beta$.
Since $\cZ^{\go_\geq}_{\beta'}(t,x)$ is $\bbP_{\geq}$-a.s.\ finite, this shows the first line of \eqref{eq:2lines}. Thus, the proof of the convergence to zero of  $J^{a,b}_1(t,x)$ is complete. \qed

\begin{proof}[Proof of Lemma \ref{lem:localter}]
 The proof is considerably easier for $d=1$, so we only provide details  in the case $d\ge 2$.
 We follow exactly the plan of the proof of Proposition \ref{prop:localbis} up to Equation~\eqref{eq:jeBDG}. Because of the restriction on the last iteration of the integral over the noise, the quantity we need to bound is not  $\bbE[\lvert V_{a,k}(t)^{p}\rvert]$ but rather $\bbE[\lvert V^b_{a,k}(t)\rvert^{p}]$ where 
 \begin{equation*}
 V^b_{a,k}(t):=\int_{ \X_t^{(k)}}   \frac{\rho_{t,0}( \bt , \bx)}{{\rho(t,0)}}  \bone_{\{z_k< b\}} \prod_{i=1}^k z_i\bone_{\{z_i\geq a\}}\, (\delta_{\go_i}-\nu)(\dd t_i , \dd x_i, \dd z_i).
 \end{equation*}
To conclude, we only need to slightly improve the bounds in further computations.
The reader can check that the proof of Lemma \ref{lemmaparti} yields  the inequality \eqref{eq:dolemaparti} for  $V^b_k(t)$ with an additional  $\ind_{\{z_k< b\}}$ in each of the integrals on the right-hand side.
Let $\mathcal U(\cP,t,b)$ denote the corresponding integrals in the case $\theta=1+\frac{2}{d}$, cf.\ \eqref{defcalu}.
Adapting the proof of Proposition \ref{prop:dekz}, we want to prove that 
\begin{equation}\label{thekey}
\mathcal U(\cP,t,b) \le  \gep(b,\eta)
 \zeta_1(\eta,p,t)^{k_1}\zeta_2(\eta,p)^{k_2} \Lambda(k_2,t,p)
\end{equation}
for some $\gep(b,\eta)$ that converges to $0$ as $b\to0$ for any fixed $\eta\in(0,1)$.
Because of \eqref{thekey}, the factor $\gep(b,\eta)$ appears on the right-hand side of \eqref{eq:help6} 
  so that, provided that $\eta$ has been chosen sufficiently small, we obtain that for every $t\in[0,T]$, $x\in \bbR^d$ and $a\in(0,b)$, 
\begin{equation*}
\bbE\Big[ \rho(t,x)^{-p}\mathcal Y^{\go,a,b}_{\beta}(t,x)^p  \Big]
\le \beta^{-1}\gep(b,\eta) C(\beta,p,T)
\end{equation*}
 with the constant $C(\beta,p,T)$ of Corollary \ref{labornedesmoments}. This is exactly the desired result.
  Note that the  factor $\gb^{-1}$ comes from the fact that in
  the definition of $\mathcal Y^{\go,a,b}_{\beta}(t,x)$,
  we have $\beta^{k-1}$ instead of $\beta^{k}$.

 To prove \eqref{thekey}, we follow the proof of Proposition \ref{prop:dekz} and realize that 
 we can improve the upper bound by a factor of
 $ ( {\mu^{\log}_{0,b} (1+\frac{2}{d} )}/{\mu^{\log}_{0,\eta} (1+\frac{2}{d} )}  )^{ {p}/(1+2/d)}$ if $k\in J_1$ 
 and a factor of $ {\mu^{\log}_{0,b} (1+\frac{2}{d} )}/{\zeta_2(\eta,p)}$ if $k\in J_2$ so that \eqref{thekey} is valid for  
\begin{equation*}
 \gep(b,\eta):=   \Bigg(\frac{\mu^{\log}_{0,b} (1+\frac{2}{d} )}{\mu^{\log}_{0,\eta} (1+\frac{2}{d} )} \Bigg)^{\frac{p}{1+2/d}} \vee \frac{\mu^{\log}_{0,b} (1+\frac{2}{d} )}{\zeta_2(\eta,p)} . \qedhere
\end{equation*}
\end{proof}

\subsection{General initial condition}\label{sec:geninitial}
Using translation invariance, we have already shown that the point-to-point partition function $\cZ_{\gb}^{\om}(y;t,x)$ is a solution to \eqref{eq:integz} with $u_0=\delta_y$, for all $y\in\R^d$. The fact that $v$ as defined in \eqref{eq:v} is a mild solution to the SHE with initial condition $u_0$ follows immediately by integrating \eqref{eq:SHE} on both sides with respect to $u_0$, provided that integrals can be permuted in the following manner:
\begin{multline*}
\int_{\bbR^d}\left( \int_{(0,t)\times\bbR^d}\rho(t-s,x-y')\cZ_{\gb}^{\om}(y;s,y') \,\xi_{\om}(\dd s, \dd y')\right)\,u_0(\dd y)\\
= \int_{(0,t)\times\bbR^d}\rho(t-s,x-y')\left( \int_{\bbR^d}\cZ_{\gb}^{\om}(y;s,y')\,u_0(\dd y) \right)\,\xi_{\om}(\dd s, \dd y').
\end{multline*} 
Writing $\xi_\om=\xi_{\om_<}+\xi_{\om_\geq}$, we can use the ordinary Fubini theorem for the $\xi_{\om_\geq}$-integral
and  a stochastic version of  Fubini's theorem \cite[Theorem~A.3]{Chong19} for the integration with respect to $\xi_{\om_<}$. 
Both cases require  some integrability properties, which
 we are going to show   by recycling the estimates from Section~\ref{sec:fin}.
 
 Concerning the integral with respect to $\xi_{\om_\geq}$, we simply need to check integrability in the Lebesgue sense, that is, 
$$
 \int_{\R^d}\int_{(0,t)\times\R^d}\rho(t-s,x-y')\calz^{\om}_\beta(y;s,y')\,\xi_{\om_\geq}(\dd s,\dd y') \,\lvert u_0\rvert(\dd y)<\infty 
 \qquad \P_{\geq}\text{-a.s.}
 $$
By Minkowski's integral inequality and Lemma~\ref{lem:pint},
for $p\in(1,1+\frac 2d)$, the $L^p(\P_<)$-norm of the left-hand side  is bounded by a constant times
\begin{multline*}
	\int_{\R^d} \int_{(0,t)\times\R^d}\rho(t-s,x-y') \calz^{\om_\geq}_{\beta'}(y;s,y')\, \xi_{\om_\geq}(\dd s,\dd y') \,\lvert u_0\rvert (\dd y) \\
	= (\beta')^{-1}\left( \int_{\R^d} \calz^{\om_\geq}_{\beta'}(y;t,x) \lvert u_0\rvert (\dd y) -\int_{\R^d}\rho(t,x-y)\,\lvert u_0\rvert (\dd y)\right),
\end{multline*}
where the equality follows from \eqref{eq:zeidentity} with $a=1$.
 The second integral on the right-hand side is finite as a consequence of our assumption \eqref{eq:u0growth}. The first one is finite by \cite[Prop.~2.21]{BL20_cont}.
 
 \smallskip
 
  For the integral with respect to $\xi_{\om_<}$, according to \cite[Theorem~A.3]{Chong19}, changing the order of integration is permitted if we have
$$ \int_{\R^d} \E_<\Bigg[\bigg( \int_{(0,t)\times\R^d\times(0,1)}\rho(t-s,x-y')^2\calz^{\om}_\beta(y;s,y')^2z^2\,\delta_{\om_<}(\dd s,\dd y',\dd z) \bigg)^{\frac p2}\Bigg]^{\frac1p} \,\lvert u_0\rvert(\dd y) <\infty.$$
Using~\eqref{eq:BDG8} and \eqref{eq:Mink2}, together with~Lemma~\ref{lem:pint} and~\eqref{eq:Mink1}, we can bound the quantity above by a constant times  
\begin{align*}
\int_{\R^d} \left( \rho(t,x-y) + \calz^{\om_\geq}_{\beta'}(y;t,x) +  \int_
{(0,t)\times \bbR^d} \rho(t-s,x-y') \calz^{\om_\geq}_{\beta'}(y;s,y') \, \xi_{\go_\ge} (\dd s ,\dd y')   \right)\, \lvert u_0\rvert(\dd y),
\end{align*}
The integrability of each of the three terms has already been shown above. 
\qed

\subsection{Integrability property of the solution}\label{sec:intesol}

Let us show that the solution $v$ defined in Equation~\eqref{eq:v} satisfies the integrability condition \eqref{eq:cond}. The important part is to show that whenever $u_0$ satisfies \eqref{eq:u0growth} for a given $T>0$, then for any  $p\in (1,1+\tfrac 2 d)$, we have
\begin{equation}\label{zatsfinite}
 \int_{(0,T)\times\R^d} \rho(\theta (T-t),x)^p \bbE_{<}\Bigg[ \bigg\lvert\int_{\R^d}  \calz^{\om}_{\beta}(y;t,x)\, u_0(\dd y) \bigg\rvert^p\Bigg]\,\dd t\,\dd x  <\infty.
  \end{equation}
In our proof of \eqref{zatsfinite}, we will check along the way that 
 the expectation term, which equals $\E_<  [\lvert v(t,x)\rvert^p  ]$, is finite for any $t$ and $x$, proving  the first line of \eqref{eq:cond}. 
By \eqref{eq:u0growth},  we can find $\gep>0$  such  that 
\begin{equation}\label{eq:choicegep}
2 T \limsup_{r\to\infty} r^{-2} \log \Big(|u_0|([-r,r]^d)) \Big) <1-2\gep \,.
\end{equation}
We also let $\theta:=(1-\gep)^{-1}$.
In the following proof, the parameters $p\in (1,1+\tfrac 2 d)$, $T$, $\beta$ and $\gl$ and $\gep$ are all fixed, and the constants appearing in the inequalities may depend on them.

Using Minkowski's inequality first and Lemma~\ref{lem:pint} afterwards, we have for any $t\in[0,T]$,
$$
\E_< \Bigg[\bigg\lvert  \int_{\R^d} \calz^{\om}_\beta(y;t,x)\,u_0(\dd y)\bigg\rvert^p \Bigg]^{\frac1p}\leq\int_{\R^d} \E_< \Big[\calz^{\om}_\beta(y;t,x)^p \Big]^{\frac1p}\,\lvert u_0\rvert(\dd y) \le  C \int_{\R^d}  \calz^{\om_\geq}_{\beta'}(y;t,x)\,\lvert u_0\rvert(\dd y)\,.
$$
 The a.s.\ finiteness of the right-hand side for fixed $x$ and $t$ is a consequence of \cite[Prop.\ 2.19]{BL20_cont}.
Using the inequality above, to prove \eqref{zatsfinite}, we now need to show that  $\bbP_\geq$-a.s.,
\begin{equation}\label{anotherfiniteintegral}
 \int_{(0,T)\times\R^d} \rho(\theta (T-t),x)^p \Bigg(\int_{\R^d}  \calz^{\om_\geq}_{\beta'}(y;t,x)\,\lvert u_0\rvert(\dd y)\Bigg)^p\,\dd t\,\dd x  <\infty.
  \end{equation}
We do so by comparing $\calz^{\om_\geq}_{\beta'}(y;t,x)$ with a variable with finite $p$th moment. For a small $\gep>0$,
we introduce the quantity
\begin{equation}
	\label{eq:TT}\calt\ceq \sup_{k\geq1}\sup_{y\in \bbR^d} \suptwo{(t_i,x_i,z_i)_{i=1}^k\subseteq\om_\ge:}{ 0<t_1<\dots<t_k<T} \left[\sum_{i=1}^{k} \log z_i -\eps \sum_{i=2}^{k}\frac{\lVert x_i-x_{i-1}\rVert^2}{2(t_i-t_{i-1})}- \eps\bigg(\frac{\|y\|^2}{2T} +\frac{\lVert x_1-y\rVert^2}{2t_1} \bigg)\right].
\end{equation}
Note that optimizing over $y$ leads to the simpler expression
\begin{equation}
	\label{eq:TTprim}
	\calt= \sup_{k\geq1} \suptwo{(t_i,x_i,z_i)_{i=1}^k\subseteq\om_\ge:}{ 0<t_1<\dots<t_k<T} \left[\sum_{i=1}^{k} \log z_i  -\eps \sum_{i=2}^{k}\frac{\lVert x_i-x_{i-1}\rVert^2}{2(t_i-t_{i-1})}-  \frac{\gep\lVert x_1\rVert^2}{2(T+t_1)} \right].
\end{equation}
The following lemma is an easy  consequence of  \cite[Lemma~4.17]{BL20_cont}.
\begin{lemma}\label{quatrepointdixseptbis}
For any $T>0$ and $\gep>0$,
 $\mathcal T$ is finite $\bbP_\geq$-a.s.\ 
\end{lemma}

\begin{proof}
 By \cite[Lemma~4.17]{BL20_cont}, we have (with the convention $x_0=0$ and $t_0=0$) 
 \begin{equation*}
	\sup_{k\geq1} \suptwo{(t_i,x_i,z_i)_{i=1}^k\subseteq\om_\ge:}{ 0<t_1<\dots<t_k<2T} \sum_{i=1}^{k} \Biggl(\log z_i -\eps\frac{\lVert x_i-x_{i-1}\rVert^2}{2(t_i-t_{i-1})}\Biggr)<\infty.
\end{equation*}
Taking a further restriction (namely $t_1>T$) and shifting time by $-T$, we obtain by translation invariance that
 \begin{equation*}
	\sup_{k\geq1} \suptwo{(t_i,x_i,z_i)_{i=1}^k\subseteq\om_\ge:}{ 0<t_1<\dots<t_k<T} \left[\sum_{i=1}^{k} (\log z_i)  -\eps \sum_{i=2}^{k}\frac{\lVert x_i-x_{i-1}\rVert^2}{2(t_i-t_{i-1})}- \eps\frac{\lVert x_1\rVert^2}{2(t_1+T)}  \bigg)\right]<\infty.\qedhere
\end{equation*}
\end{proof}

  As a consequence of Lemma \ref{quatrepointdixseptbis}, we have for every $x,y\in \mathbb R^d$ and $t\in [0,T]$ that
\begin{equation}\label{lacompapa}
\calz_{\beta'}^{\om_\geq}(y;t,x)\leq e^{\calt+ \frac{\gep \|y\|^2}{2T}} \sum_{k=0}^\infty(\beta')^k \int_{\X_t^{(k)}}\rho_{y;t,x}(\bt,\bx)\prod_{i=1}^{k+1} e^{\eps\frac{\lVert \Delta x_i\rVert^2}{2\Delta t_i}} \prod_{j=1}^k \delta_{\om_{\geq}}(\dd t_j,\dd x_j,\dd z_j),
\end{equation}
with the convention that $x_0=y$.
Because 
\begin{equation}
	\label{eq:scalingrho}
	\rho(t,x) e^{\eta \frac{\lVert  x\rVert^2}{2 t}} = \theta_{\eta}^{\frac d2} \rho(\theta_{\eta} t,x)\, \qquad \text{with}\qquad \theta_\eta = (1-\eta)^{-1} 
\end{equation}
and   $\theta=(1-\gep)^{-1}$,
we have $\calz_{\beta'}^{\om_\geq}(y;t,x)\leq e^{\calt+ \frac{\gep \|y\|^2}{2T}} \wh\calz_{\beta'}^{\om_\geq}(y;t,x)$, where
\begin{equation}
\label{def:hatcalz}
\wh\calz_{\beta'}^{\om_\geq}(y;t,x) := \sum_{k=0}^\infty (\theta^{\frac d2}\beta')^k \int_{\X_t^{(k)}}\rho_{y;\theta t,x}(\theta\bt,\bx)  \prod_{j=1}^k \delta_{\om_{\geq}}(\dd t_j,\dd x_j,\dd z_j) \,.
\end{equation}
Setting $\hat u_0(\dd y):= e^{\frac{\gep\|y\|^2}{2T}} |u_0|(\dd y)$, we can reduce the proof of~\eqref{anotherfiniteintegral} to showing
\begin{equation}\label{lastfiniteintegral}
 \int_{(0,T)\times\R^d} \rho(\theta (T-t),x)^p \Bigg(\int_{\R^d}  \wh\calz^{\om_\geq}_{\beta'}(y;t,x)   \hat u_0(\dd y)\Bigg)^p\,\dd t\,\dd x  <\infty.
  \end{equation}
%
We 
take the expectation and apply Minkowski's inequality twice to get
\begin{equation}
\label{momentthat-1}
\begin{split}
	\E\Bigg[\int_{(0,T)\times\R^d} \rho&(\theta (T-t),y)^p \Bigg(\int_{\R^d}  \wh\calz^{\om_\geq}_{\beta'}(y;t,x) \hat u_0(\dd y)\Bigg)^p\,\dd t\,\dd y\Bigg]\\
	& \leq\int_{(0,T)\times\R^d} \rho(\theta (T-t),y)^p \Bigg(\int_{\R^d} \E\Big[ \wh\calz^{\om_\geq}_{\beta'}(y;t,x)^p\Big]^{\frac1p}\,\hat u_0(\dd y)\Bigg)^p\,\dd t\,\dd x \\
	&\leq \Bigg(\int_{\R^d}\Bigg(\int_{(0,T)\times\R^d} \rho(\theta (T-s),y)^p \E\Big[ \wh\calz^{\om_\geq}_{\beta'}(y;t,x)^p\Big]\,\dd t\,\dd x\Bigg)^{\frac1p}\, \hat u_0 (\dd y)\Bigg)^p.
	\end{split}
\end{equation}
Observe now that
$\wh\calz^{\om_\geq}_{\beta'}(y;t,x)$ is a (non-normalized) point-to-point partition function corresponding to a Poisson environment with intensity measure $\gl([1,\infty)) \delta_{\theta^{d/2}\beta'}$ and a time-rescaling by $\theta$.
Therefore, by Corollary~\ref{labornedesmoments}, 
\begin{equation}\label{momenthat}
 \E  [ \wh\calz^{\om_\geq}_{\beta'}(y;t,x)^p  ] \le C \rho(\theta t,x-y)^p,
\end{equation} 
for a constant $C$ that depends on all parameters but not on $y$, $t$ and $x$.
Hence, the $\frac1p$th power of~\eqref{momentthat-1} is bounded by a constant times
\begin{equation*}
\int_{\R^d}\Bigg(\int_{(0,T)\times\R^d} \rho(\theta (T-t),x)^p\rho(\theta t,x-y)^p  \,\dd t\,\dd x\Bigg)^{\frac1p} \,\hat u_0 (\dd y)
\leq C'\int_{\R^d} T^{\nu_p} e^{- \frac{\lVert y\rVert^2}{2\theta T}} \,\hat u_0 (\dd y) \,,
\end{equation*}
where the inner integral has been computed exactly, as in~\eqref{eq:calculrhorho}.
The above integral is finite thanks to our choice of $\gep$ and $\theta$ in \eqref{eq:choicegep}. Thus, \eqref{lastfiniteintegral} holds and the proof is complete.
\qed

\subsection{Uniqueness of  solutions to the SHE}
\label{sec:unique}

\newcommand{\Difference}{v}

We assume in this section that $\lambda$ satisfies \eqref{assump1} and   prove uniqueness among solutions satisfying \eqref{eq:cond}.
By the fact that the condition \eqref{eq:cond} is stable under linear combinations, it is sufficient to show that any solution to \eqref{eq:integz} with $u_0\equiv 0$ is equal to zero.
%
%
Consider  $v$ that satisfies \eqref{eq:cond} and is such that for every $t\in (0,T]$ and $x\in \bbR^d$,
\[
\Difference(t,x)= \beta\int_0^t\int_{\R^d} \rho(t-s,x-y)\Difference(s,y)\,\xi_\om(\dd s,\dd y) \,.
\]
Applying this identity to $\Difference(s,y)$ in the integrand and repeating this, we obtain that for any $k\ge 1$,
\begin{equation}\label{kreplik}
\Difference(t,x)= \beta^k\int_{\calx^{(k)}_t} \rho_{t_1,x_1;t,x}(\bt^{(1)},\bx^{(1)})\Difference(t_1,x_1)\prod_{i=1}^k \xi_\om(\dd t_i,\dd x_i),
\end{equation}
where $\bt^{(1)}$ and $\bx^{(1)}$ are obtained from $\bt$ and $\bx$, respectively, by removing the first component.  
Fixing $x\in \bbR^d$ and $t\in (0,T)$, we are going to prove that  $v(t,x)=0$ by showing that the right-hand side in \eqref{kreplik} is summable in $k$. More precisely, we set  
$$
\calu(t,x):= \sum_{k=0}^\infty \beta^k\int_{\calx^{(k)}_t} \rho_{t_1,x_1;t,x}(\bt^{(1)},\bx^{(1)})\Difference(t_1,x_1)\prod_{i=1}^k \xi_\om(\dd t_i,\dd x_i) \,,
$$
the term for $k=0$ being simply $\Difference(t,x)$.
We fix $p\in(1,1+\frac{2}{d})$  for which $\mu_{0,1}(p)<\infty$ (when $d=1$ we fix $p=2$) and we are going to prove that
\begin{equation}\label{ucefini}
\bbE_{<}\Big[\big|\calu(t,x)\big|^p\Big]<\infty \qquad \text{$\bbP_{\geq}$-a.s.}
\end{equation}
Separating $\om_<$ and $\om_\geq$ in $\xi_\om$,  we have that 
\begin{equation}\label{summu}
\calu(t,x)=  \sum_{k=0}^\infty \beta^k\int_{\calx^{(k)}_t}  \calv(t_1,x_1)\prod_{i=2}^{k+1}  \calz^{\om_<}_\beta(t_{i-1},x_{i-1};t_i,x_i) \prod_{i=1}^k \xi_{\om_\geq}(\dd t_i,\dd x_i),
\end{equation}
where
\begin{equation}\label{summuprim}
\calv(t,x):= \sum_{k=0}^\infty \beta^k\int_{\calx^{(k)}_t} \rho_{t_1,x_1;t,x}(\bt^{(1)},\bx^{(1)})\Difference(t_1,x_1) \prod_{i=1}^k \xi_{\om_<}(\dd t_i,\dd x_i) 
\end{equation}
and the terms for $k=0$ in \eqref{summu} and \eqref{summuprim} are $\calv(t,x)$ and  $\Difference(t,x)$ respectively.
By Corollary~\ref{labornedesmoments} and translation invariance, there exists $\beta'>0$ such that
\beq\label{eq:help7}\begin{split}
& \E_< \Big[\lvert\calu(t,x)\rvert^p \Big]^{\frac1p}\\ &\qquad\leq  
	 \E_< \Big[\lvert \calv(t,x)\rvert^p \Big]^{\frac1p}+\sum_{k=1}^\infty (\beta')^k\int_{\calx^{(k)}_t} \E_< \Big[\lvert \calv(t_1,x_1)\rvert^p \Big]^{\frac1p}  \prod_{i=2}^{k+1}\rho(\Delta t_i,\Delta x_i) \prod_{i=1}^k \xi_{\om_\geq}(\dd t_i,\dd x_i).
\end{split}\eeq
To prove \eqref{ucefini}, we bound each of the 
two summands in \eqref{eq:help7} separately.

Let us start with  $\E_<  [\lvert \calv(t,x)\rvert^p  ]^{1/p}$.
Using the triangle inequality and iterating the BDG inequality and the subadditivity property~\eqref{eq:subbaditiv} for each term  (as done in \eqref{eq:BDG}, for example), we obtain
\begin{equation*}
\E_<\Big[\lvert	\calv(t,x)\rvert^p\Big]^{\frac1p}\leq \sum_{k=0}^\infty (C_p \beta\mu_{0,1}(p)^{\frac1p})^k \Bigg(\int_{\calx^{(k)}_t} \rho_{t_1,x_1;t,x}(\bt^{(1)},\bx^{(1)})^p\E_<\Big[\lvert\Difference(t_1,x_1)\rvert^p\Big] \prod_{i=1}^k  \dd t_i\,\dd x_i\Bigg)^{\frac1p} \,.
\end{equation*}
Then, applying \eqref{eq:iteinte} and   Lemma~\ref{lemGamma}
with the trivial inequality $t-t_1\le t$,
we get that
\begin{equation}
\label{eq:help8}
\begin{split}
&\E_<\Big[\lvert	\calv(t,x)\rvert^p\Big]^{\frac1p}\leq \E_<\Big[\lvert\Difference(t,x)\rvert^p\Big]^{\frac1p} \\
&\quad+ \sum_{k=1}^\infty \frac{ \left[C_p\beta(\mu_{0,1}(p)\Ga(\nu_p)t^{\nu_p})^{\frac1p} \right]^k}{\Ga(k\nu_p)^{\frac1p}}
	\Bigg(\int_{(0,t)\times\R^d} \rho(t-t_1,x-x_1)^p \E_<\Big[\lvert\Difference(t_1,x_1)\rvert^p\Big]\,\dd t_1\,\dd x_1\Bigg)^{\frac1p}.
	\end{split}
\end{equation}
Since the sum in $k$ in finite, the two conditions in  our assumption \eqref{eq:cond} implies the $\P_{\geq}$-a.s.\ finiteness of the right-hand side of \eqref{eq:help8}.
 

\smallskip

  To show that the series in   \eqref{eq:help7} is finite, we apply a similar trick to \eqref{lacompapa} and replace the $z_i$'s coming from the atoms of $\go_{\ge}$ by one. More precisely, 
we set 
\begin{equation}\label{eq:TTxt}
 \calt(t,x):= \sup_{k\geq1} \suptwo{(t_i,x_i,z_i)_{i=1}^k\subseteq\om_\ge:}{ 0<t_1<\dots<t_k<t} \left[\sum_{i=1}^{k} (\log z_i)  -\eta \sum_{i=2}^{k+1}\frac{\lVert x_i-x_{i-1}\rVert^2}{2 \Delta t_i}\right] \,,
\end{equation}
with the convention that $t_{k+1}=t$ and $x_{k+1}=x$.
The value of $\eta>0$ in~\eqref{eq:TTxt}
 is chosen  such that $\theta = ( 1- \eta \frac{3p-1}{2} )^{-1}$, with the same $\theta$ as in the assumption~\eqref{eq:cond}.
By the translation invariance and symmetry properties of $\om_\ge$, \cite[Lemma~4.17]{BL20_cont} guaranties that 
$\calt(t,x)<\infty$, $\bbP_{\ge}$-a.s.\
Defining
\begin{equation}\label{defchapo}
\wh\calz_{\beta'}^{\om_\geq}(s,y;t,x) := \sum_{k=0}^\infty(  \beta')^k \int_{\X_t^{(k)}}\rho_{s,y;t,x}(\bt,\bx) \prod_{i=1}^{k+1} e^{ \eta \frac{\lVert \Delta x_i\rVert^2}{2\Delta t_i}} \prod_{j=1}^k \delta_{\om}(\dd t_j,\dd x_j\,\dd z_j)\,,
\end{equation}
we find that the series in \eqref{eq:help7} is smaller than the following integral (we have replaced $t_1,x_1,z_1$ by $s,y,z$):
\begin{equation}\label{finitintegral}
 \beta' e^{\calt(t,x)}\int_{(0,t)\times \bbR^d\times [1,\infty)} \wh\calz_{\beta'}^{\om_\geq}(s,y;t,x)\E_< \Big[\lvert \calv(s,y)\rvert^p \Big]^{\frac1p}\,\delta_{\go_{\ge}}(\dd s,\dd y, \dd z).
\end{equation}
Note that the integral in \eqref{finitintegral} is a sum and that all terms are finite almost surely  by \eqref{eq:help8} and \eqref{momenthat}. 
In order to show that the integral in~\eqref{finitintegral} is finite, we proceed in several steps. The  first one is to replace 
$\wh\calz_{\beta'}^{\om_\geq}(s,y;t,x)$ by $\rho(t-s,x-y)e^{\eta {\lVert x-y\rVert^2}/{(t-s)}}$.
This is the purpose of the following lemma whose short proof is postponed to the end of the section.

\begin{lemma}\label{finipointz}
 Let  
 \begin{equation}\begin{split}
   \Theta(t,x)&:= \bigg\{ (s,y)  \in (0,t)\times \bbR^d  \ : \ \wh\calz^{\om_\geq}_{\beta'}(s,y;t,x)\ge \rho(t-s,x-y)e^{\eta\frac{\lVert x-y\rVert^2}{(t-s)}} \bigg\},\\
   \bar \go_{\ge}&:= \Big\{  (s,y) \in (0,t)\times \bbR^d  \ : \ \exists z\ge 1,\ (s,y,z)\in \go \Big\}.  
    \end{split}
 \end{equation}
Then we have  $\#\left(\Theta(t,x)\cap \bar \go_{\ge}\right)<\infty$ 
  $\bbP_{\ge}$-a.s.
\end{lemma}

\noindent As a consequence of Lemma \ref{finipointz}, the integral in \eqref{finitintegral} is finite if 
\begin{equation}\label{onestepbeyond}
\int_{(0,t)\times \bbR^d\times [1,\infty)} N(s,y)\, \delta_{\go_{\ge}} (\dd s,\dd y, \dd z)<\infty
\end{equation}
where 
$$N(s,y):= \E_< \Big[\lvert \calv(s,y)\rvert^p \Big]^{\frac1p} \rho(t-s,x-y)e^{\eta\frac{\lVert x-y\rVert^2}{(t-s)}}.$$
By \cite[Prop.\ 12.1]{PoiBook}, 
\eqref{onestepbeyond} holds if we can show that 
\begin{equation}\label{wuhuhu}
 \int_{(0,t)\times \bbR^d}\E_< \Big[\lvert \calv(s,y)\rvert^p \Big]^{\frac1p}  \rho(t-s,x-y)e^{\frac{\eta\lVert x-y\rVert^2}{(t-s)}}\, \dd s\, \dd y<\infty\,.
\end{equation}
To this end,
we use the bound \eqref{eq:help8}  for $\E_< [\lvert	\calv^{\om}_\beta(s,y)\rvert^p ]^{1/p}$ and check that the integrals corresponding to each of the two summands on  the right-hand side of \eqref{eq:help8} are finite. In the first case, we must show that 
\begin{equation}
\int_{(0,t)\times \bbR^d}\E_< \Big[ \big\lvert	\Difference(s,y)\big\rvert^p \Big]^{\frac1p}  \rho(t-s,x-y)e^{\eta\frac{\lVert x-y\rVert^2}{(t-s)}}\, \dd s\, \dd y<\infty \,.
\end{equation}
Applying Jensen's inequality for the finite measure $e^{-\eta {\lVert x-y\rVert^2}/{2(t-s)}}\,\dd s\, \dd y$, we can bound the quantity above by
\begin{equation}\label{jensenexperto}
 \Bigg( \int_{(0,t)\times \bbR^d} \E_< \Big[ \big\lvert	\Difference(s,y)\big\rvert^p \Big]   \rho(t-s,x-y)^p   e^{ \eta(3p-1) \frac{\lVert x-y\rVert^2}{2(t-s)}}\,\dd s\, \dd y \Bigg)^{\frac1p}.
 \end{equation}
Recalling~\eqref{eq:scalingrho} and the relation $\theta = (1-\eta \frac{3p-1}{2})^{-1}$, the finiteness of~\eqref{jensenexperto} follows from our assumption \eqref{eq:cond}.

For the second summand on  the right-hand side of \eqref{eq:help8},
 we have to show that 
\begin{equation}\label{thatsover}
		\int_{(0,t)\times\R^d} e^{\eta \frac{\lVert x-y\rVert^2}{(t-s)}} \rho(t-s,x-y)\Bigg(\int_{(0,s)\times\R^d} \rho(s-r,y-v)^p \E_<\Big[\lvert\Difference(r,v)\rvert^p\Big]\,\dd r\,\dd v\Bigg)^{\frac1p}\,\dd s\,\dd y<\infty \,.
\end{equation}
Using the same trick as in \eqref{jensenexperto},
we see that it is sufficient to prove the finiteness of 
\begin{equation*}
		\int_{(0,t)\times\R^d}  \int_{(0,s)\times\R^d} e^{-\eta(3p-1)\frac{\lVert x-y\rVert^2}{2(t-s)}} \rho(t-s,x-y)^p\rho(s-r,y-v)^p \E_<\Big[\lvert\Difference(r,v)\rvert^p\Big]\,\dd r\,\dd v\,\dd s\,\dd y \,.
\end{equation*}
Now, by~\eqref{eq:scalingrho} and thanks to our choice of $\eta$, we can replace $e^{-\eta(3p-1) {\lVert x-y\rVert^2}/{2(t-s)} } \rho(t-s,x-y)^p$ by a constant times $\rho(\theta (t-s),x-y)^p$
and bound $\rho(s-r,y-v)^p$ by a constant times $\rho(\theta(s-r),y-v)^p$.
Then, using~\eqref{help3} as in\eqref{eq:calculrhorho}, we have
\begin{align*}
\int_{(r,t)\times\R^d} \rho( \theta(t-s),x-y)^p\rho( \theta(s-r),y-v)^p\,\dd s\,\dd y
\leq C'_{\eta,p} (t-r)^{\nu_p} \rho( \theta  (t-r) ,x-v)^p \,.
\end{align*}
Hence,  the finiteness of \eqref{thatsover} is also consequence of \eqref{eq:cond},
which finally concludes the proof.
\qed

\begin{proof}[Proof of Lemma \ref{finipointz}]
 By Markov's inequality and \eqref{momenthat},
 we have for any given $s,y\in  (0,t)\times \bbR^d $,
 \begin{align*}
  \bbP_{\ge}\bigg[ \wh\calz^{\om_\geq}_{\beta'}(s,y;t,x)\ge \rho(t-s,x-y)e^{\eta \frac{ \lVert x-y\rVert^2}{(t-s)}} \bigg] 
  & \le C \bigg( \frac{\rho\big(\frac{t-s}{1-\eta},x-y  \big)}{\rho(t-s,x-y)} e^{- \eta \frac{\lVert x-y\rVert^2}{(t-s)}} \bigg)^p 
  &= C' e^{- \eta p\frac{\lVert x-y\rVert^2}{2(t-s)}}\,,
 \end{align*}
  using~\eqref{eq:scalingrho} for the last inequality.
Hence, we have 
\begin{align*}
 \bbE_{\ge}\Big[ \#\left(\Theta(t,x)\cap \bar \go_{\ge}\right)\Big] 
& \le C' \int_{(0,t)\times \bbR^d\times [1,\infty)} e^{- \eta p \frac{\lVert x-y\rVert^2}{2(t-s)}}\,\dd s\, \dd y\, \gl (\dd z)\\ 
&= C_p \gl([1,\infty))\int_{(0,t)} (t-s)^{\frac d2}\, \dd s <\infty.\qedhere
\end{align*}
\end{proof}

\section{Moments of order $p\in (0,1)$ and lower bounds in Theorems~\ref{thm:thinup} and \ref{thm:heavyup}} \label{sec:theotherbound}

This section is dedicated to estimating fractional   moments of $\cZ^{\go}_{\beta}(t,\ast)$. More precisely, we focus on
$\bbE [ \cZ^{\go}_{\beta}(t,\ast)^{1/2} ]$ to simplify notation, but the method would be equally efficient to directly estimate $\bbE [ \cZ^{\go}_{\beta}(t,\ast)^p ]$ for $p\in(0,1)$.
The estimates obtained in this section allow us to complete the proof of Theorems \ref{thm:thinup} and \ref{thm:heavyup}, since by convexity and the fact that $\bar \gamma_{\beta}(1)=0$ and $\bar \gamma_{\beta}(0)=0$,
we have 
\begin{equation}\label{convexinho2}
 \begin{cases}
  \bar\gamma_{\beta}(p) \le 2p \bar \gamma_{\beta}(\frac12) \quad &  \text{ for } p\in (0,\frac12),\\
    \bar\gamma_{\beta}(p) \le 2(1-p) \bar \gamma_{\beta}(\frac12) \quad &  \text{ for } p\in (\frac12,1),\\
     \bar\gamma_{\beta}(p) \ge -2(p-1)  \bar\gamma_{\beta}(\frac12)
    \quad & \text{ for } p>1.
 \end{cases}
\end{equation}
We also
prove Proposition \ref{prop:waitingforbetter} in Section~\ref{sec:waitingforbetter},
partially using ideas developed in Sections~\ref{sec:coarse}--\ref{sec:multibod}.

\smallskip

In Section \ref{sec:coarse}, we introduce a method that combines a coarse-graining and a change-of-measure argument,  which allows  to obtain upper bounds that decay exponentially in $t$ (this  is crucial for the proof of    Theorems \ref{thm:thinup} and \ref{thm:heavyup}).
This approach originates in the study of the discrete pinning model \cite{DGLT09,GLT10hier}: in its refined form, which first appeared in  \cite{Ton09}, it can be used in a continuum setup such as the SHE.  Note that the method   found  many implementations in the last decade, such as for disordered pinning \cite{BL18pin,GLT10}, directed polymers and variants \cite{BL17,Ber15,Lac10pol,Lac11}, the random walk pinning model \cite{BT10, BS11}, large deviations of random walks in a random environment \cite{YZ10}, self-avoiding walks in a random environment \cite{Lac14SAW}, and anomalous path detection in a random environment \cite{CZ18}.

\smallskip

The achievement of the method (presented in Lemma \ref{lem:changemeas}) is to reduce the problem of estimating moments to that of showing that the original measure $\bbP$ significantly differs (in total variation) from an alternative measure where the environment has been modified along a Brownian trajectory (the size-biased measure).
This last statement is proved in Sections \ref{sec:onebody} (for $d=1$ or heavy-tailed noises) and \ref{sec:multibod} (for $d\ge 2$  and light-tailed noises).
While Lemma \ref{lem:changemeas} is quite general and can be adapted to various settings, the proof that the original and the tilted measures differ in total variation heavily depends on context.

The proof presented in Section \ref{sec:onebody} for $d=1$ relies on  ideas found in \cite[Section 3]{Lac10pol}: in essence, we discriminate between the original and the tilted environments by looking at the average in a box of length $T$ and wdith $\sqrt{T}$.
The proof for heavy-tailed environments, also found in Section~\ref{sec:onebody}, relies on a similar idea to that found in \cite[Section 2]{Vi19}.

On the other hand, the approach taken in Section \ref{sec:multibod} and Section \ref{sec:waitingforbetter} are completely new and have no discrete analogue. We show that what makes the tilted measure
different from the original one is the presence of clusters of points that are very close to one another. It requires a fine analysis to identify exactly what the characteristics of these clusters are.

%

 \subsection{A general coarse-graining lemma}\label{sec:coarse}
 
For notational simplicity, we set  $\cZ^{\go}_{\beta,t}:=\bar\cZ_{\beta}^{\go}(t,\ast)$ in the remainder of this section.
Our coarse-graining approach reduces the problem of bounding fractional moments of $  \cZ_{\beta,t}^{\go}$ to   identifying a single event that is unlikely under the original measure $\bbP$ but becomes likely under the size-biased measure $\wt \P^{0}_{\beta,t}$; recall Section~\ref{sec:sizebias} and notation therein. 

We say that a measurable event $\cA$ for the point process $\go$ has \emph{time range} $[0,T)$ if  $\ind_{\cA}(\go)$ is a function of $\go\cap ([0,T)\times \bbR^d\times (0,\infty))$. Moreover,  defining $\wh \go=\wh \go(\go,\go',B)  := \go \cup \wt \go (\go',B)$, we will use, for a generic real-valued function $f$ defined on the set of environments, the abbreviation
\begin{equation}\label{shorthat}
\wh f(\go,\go',B):= f\circ \wh \go(\go,\go',B).
\end{equation}
Also, recall that $\bQ_x$ is the law of a  $d$-dimensional standard Brownian motion starting from $x$. 

 \begin{lemma}
\label{lem:changemeas} Assume that $\mu <\infty$.
There exists a constant $K\in(0,\infty)$, which only depends  on the dimension $d$, with the following property:
For any  $\beta>0$ and $T>0$, if there exists an event $\cA$ 
with time range  $[0,T)$ satisfying both
\begin{gather}
\bbP(\go \in \cA ) \leq e^{-K} \,, \tag{H1}\label{H1}\\
\max_{x \in [0,\sqrt{T})^d}  \bbP \otimes \bbP'_{\beta}\otimes\bQ_x (  \wh \go \notin \cA  ) \leq e^{-K} \, , \tag{H2}\label{H2}
\end{gather}
 then  for all $m\ge 1$ we have
\begin{equation}\label{coarse1}
\bbE\Big[ ( \cZ^{\go}_{\beta,mT})^{\frac12}\Big]  \leq 2^{-m}  \,,
\end{equation}
 and as a consequence
$\bar\gamma_{\gb}(\frac12) \leq - (\log 2) T^{-1}.$
\end{lemma} 

\noindent Below, we are going to apply Lemma \ref{lem:changemeas} to events of the type 
 \beq\label{eq:A}\cA:= \{\go \ : \  f(\go)-\E[f(\go)] > h \} \qquad \text{or}\qquad \cA:= \{\go \ : \  f(\go) > h \} \eeq
 where $f$ is a function that only depends on $\go \cap ([0,T)\cap \bbR^d \times (0,\infty))$ and $h\in \bbR$ is a threshold  to  be chosen appropriately.
 More precisely, we consider functions $f$ of the form
 \begin{equation}\label{padrao}
f(\om)= \int_{\X^{(k)}_T} \Upsilon(\bt,\bx,\bz) \prod_{i=1}^k\delta_\om(\dd t_i,\dd x_i,\dd z_i), 
\end{equation} 
where $k$ is a positive integer and $\Upsilon$ is a non-negative function on $\X^{(k)}_T$ (recall the notation~\eqref{indexnotation}).
Given $f$ as above, we define  
\begin{equation}\label{thefprim}
f'(\go',B)= \int_{\calx^{(k)}_T} \Upsilon(\bt,(B_{t_i})^k_{i=1},\bz) \prod_{i=1}^k\delta_{\om'}(\dd t_i,\dd z_i).
\end{equation}
In other words, $f'$ is the total contribution of the  additional terms in $\wh f$  that are obtained by considering (only) the atoms on the added ``environment spine'' $(\om',B)$ (recall \eqref{shorthat}).
Then, for $f$ of the form~\eqref{padrao}, since $\Upsilon$ is non-negative, we have
\begin{equation}\label{magic}
\wh f(\go,\go',B)\ge f(\go)+f'(\go',B).
\end{equation}
In this setting, verifying  \eqref{H1} and~\eqref{H2} boils down to first and second moment computations for $f$ and $f'$.

\begin{cor}\label{cheapcoro}
	Assume that $f$ is as in   \eqref{padrao} 
	and satisfies one of the two following conditions:
	\begin{enumerate}
	 \item[\labelword{\rm(C1)}{(C1)}] 
	 For all $x\in [0,\sqrt{T})^d$, we have that  $$\E'_{\beta}\otimes \bQ_x [ f'(\go',B)  ]^2  \geq     4e^{K}    \Bigl\{ \var(f(\go)) +\var_{\P'_\beta \otimes \bQ_x}(f'(\go',B))\Bigr\};$$
	 \item[\labelword{\rm(C2)}{(C2)}] 
	  The function $\Upsilon$ is  integer-valued, we have $\bbE[ f(\go)] \le e^{-K}$ and  $$\forall x\in [0,\sqrt{T})^d: \quad  \E'_{\beta}\otimes \bQ_x [ f'(\go',B)  ]^2   \ge e^{K} \var_{\P'_\beta \otimes \bQ_x}(f'(\go',B)).$$ 
	\end{enumerate}
In both cases, also assume that the variances on the right-hand side are finite.
	Then  \eqref{H1} and~\eqref{H2} are satisfied and $\gamma_{\gb}(\frac12) \leq - (\log 2) T^{-1}.$
\end{cor}

\begin{proof} 
 If \ref{(C1)} is satisfied, define
\[
m:=\bbE[f(\go)] \qquad\text{and}\qquad \sigma:=\sqrt{  \var(f(\go))+\var_{\P'_\beta \otimes \bQ_x}(f'(\go',B))}
\]
 and consider the event
$$\cA:=\{ \go \ : \ f(\go) >  m+ e^{\frac K2} \sigma \}.$$
We can apply Chebyshev's inequality, using that $\sigma^2\geq \mathrm{Var}(f(\go))$, to check the validity of \eqref{H1}. For \eqref{H2}, we observe that as a consequence of \eqref{magic}, we have
 \begin{equation*}
\{ \wh \go \notin \cA\}
\subseteq
\{ f(\go) + f'(\go',B)\le m+  e^{\frac K2}  \sigma\}. 
 \end{equation*}
 Using Chebyshev's inequality, we deduce that 
 \begin{equation*}
  \bbP\otimes \bbP'_{\beta}\otimes \bQ_x (\wh \go \notin \cA )
  \le \frac{ \var(f(\go))+ \var_{\P'_\beta \otimes \bQ_x}(f'(\go',B))}{ \Big(e^{\frac K2}\sigma-\E'_{\beta}\otimes \bQ_x [ f'(\go',B)] \Big)^2 } \le e^{-K}.
 \end{equation*}
If \ref{(C2)} is satisfied,   we set  
 $$\cA:= \{ \go \ : \ f(\go) \ge 1 \} =\{ \go \ : \ f(\go) >0 \} \,.$$ 
We then deduce \eqref{H1} from Markov's inequality applied to $f(\go)$ and \eqref{H2}  from Chebychev's inequality applied to $f'(\go',B)$.
\end{proof}

\begin{rem}
 We will use the assumption \ref{(C1)} for $d=1$ in the light-tailed case and the assumption~\ref{(C2)} in all heavy-tailed cases and for $d\ge 2$. These two assumptions correspond to two different ways of distinguishing between the original measure and the size-biased one. In
 Assumption \ref{(C1)}, $f$ is increased under the size-biased measure by an amount $f'$ that, on average, exceeds the typical fluctuation (measured as $\var(f(\om))$) under the original measure. Assumption \ref{(C2)} implies that there are patterns in $\go$ that appear with large probability under the size-biased measure  and are most likely absent under the original one. 
\end{rem}

\begin{proof}[Proof of Lemma~\ref{lem:changemeas}]
Let us assume that $T$ and $\cA$ satisfy \eqref{H1} and \eqref{H2}. 
We divide $\bbR^d$ into cubes of side length $\sqrt{T}$: define
\[
\cC_j = \cC^{(T)}_{j}:= j\sqrt{T}+[0,\sqrt{T})^d,\qquad j\in \bbZ^d \, .
\]
Given $\bj =(j_1,\dots, j_{m}) \in (\bbZ^d)^{m}$, we define the function $\chi_{\bj}: C([0,mT])\to \bbR$ 
by 
\begin{equation*}
 \chi_{\bj}(\varphi):= \ind\{\forall i \in \lint m\rint: \ \varphi(iT)\in \cC_{j_i} \}.
\end{equation*}
In other words, $\chi_{\bj} (\varphi)$ is the indicator that the function $\varphi$ passes through  a prescribed sequence of cubes of diameter  $\sqrt{T}$  at times $iT$, where $i\in \lint  m\rint$. If $\varphi=(B_t)_{t\in [0,mT]}$ is a Brownian motion,  the unique value of  $\bj$ such that $\chi_{\bj}(B)=1$ corresponds to a coarse-grained trajectory (or skeleton) of~$B$ on time scale $T$.
 Now, recalling  that we have set $\cZ^{\go}_{\beta,t} = \bar \cZ^{\go}_{\beta}(t,\ast)$ and recalling the notation from Section~\ref{sec:sizebias}, we have 
\begin{equation}\label{decompsa}
\cZ^{\go}_{\beta,mT}=\int_{(\bbR^d)^m}  \prod_{i=1}^M \bar\cZ^{\go}_{\beta}((i-1)T,x_{i-1} ; iT,x_{i}) \, \dd x_i \,.
\end{equation}
This can be checked directly from the definition  if $\cZ^{\go}_{\beta,mT}$ and $\bar\cZ^{\go}_{\beta}((i-1)T,x_{i-1} ; iT,x_{i}) $ are replaced by 
$\bar\cZ^{\go,a}_{\beta}(mT,\ast)$ and $\bar\cZ^{\go,a}_{\beta}((i-1)T,x_{i-1} ; iT,x_{i})$, respectively. The case $a=0$ then follows from the $L^1$-convergence in Theorem \ref{thm:local}.
We therefore get 
$$
\cZ^{\go}_{\beta,mT}=  \sum_{\bj\in (\bbZ^d)^{m}} \cZ^{\go}_{\beta,mT}(\chi_{\bj})  \,, \qquad  \cZ^{\go}_{\beta,mT}(\chi_{\bj}) :=\int_{(\bbR^d)^m}  \prod_{i=1}^M \bar \cZ^{\go}_{\beta}((i-1)T,x_{i-1} ; iT,x_{i}) \ind_{\{x_i\in \cC_{j_i}\}} \,\dd x_i.
$$
As a consequence, by subadditivity \eqref{eq:subbaditiv}, we have
\begin{equation*}
 \bbE\Bigl[ ( \cZ^{\go}_{mT})^{\frac12} \Bigr]\le \sum_{\bj\in (\bbZ^d)^{m}}
 \bbE\Bigl[  \cZ^{\go}_{mT}(\chi_{\bj})^{\frac12}\Bigr]. 
\end{equation*}
Our aim is  to obtain a good bound on 
$ \bbE[ ( \cZ^{\go}_{\beta,mT}(\chi_{\bj}) )^{1/2}]$. This is where we introduce a change-of-measure procedure.
We use the Cauchy--Schwarz inequality as follows:
\begin{equation}\label{CS}
  \bbE\Bigl[ ( \cZ^{\go}_{\beta,mT}(\chi_{\bj}))^{\frac12} \Bigr]^2\le \bbE\Bigl [ G_{\bj}(\go)^{-1}\Bigr ]\bbE\Bigl[ G_{\bj}(\go)\cZ^{\go}_{\beta,mT}(\chi_{\bj})\Bigr] \,,
\end{equation}
for some non-negative function $G_{\bj}(\go)$ that penalizes the $\go$'s that  contribute most to $\bbE[ \cZ^{\go}_{\beta,mT}(\chi_{\bj})]$.
 This procedure   is referred to as a change of measure since $\bbE [ G_{\bj}(\go)\cZ^{\go}_{\beta,mT}(\chi_{\bj}) ]$ corresponds to the expectation of  $\cZ^{\go}_{\beta,mT}(\chi_{\bj})$ under a new measure whose density with respect to the original one is given by $G_{\bj}(\go)$.
We now construct  $G_{\bj}$ with a product structure in order to gain a constant factor per coarse-grained step of the trajectory.
 For $K >0$, a fixed constant to be determined below, we define (with $j_0:=0$)
 \begin{equation}
 \label{def:ggg}
 \begin{split}
g(\go):=  \exp ( - K\, \ind_{\cA}(\go) ),\quad
g_{i,j}(\go):= g\circ \theta_{((i-1)T,j\sqrt{T})}( \go),   \quad
 G_{\bj}(\go):= \prod_{i=1}^m g_{i,j_{i-1}}(\go) ,
\end{split}
\end{equation}
where $\theta_{(t,x)}$ is the space-time shift operator acting on point collections in $\R\times \R^d\times (0,\infty)$, that is, $$\theta_{(t,x)}(\go) = \{ (s-t,y-x,z) \ : \  (s,y,z)\in \go\}.$$
%
Because $\cala$ has time range $T$, for given $\bj$,  the variables $(g_{i,j_{i-1}}(\go))_{i=1}^m$ are independent and identically distributed under $\bbP$, so that,  by  \eqref{H1},
$$\bbE[ G_{\bj}(\go)^{-1}]= (1+ (e^K-1) \bbP(\cA) )^m \le 2^m.$$
We then derive from \eqref{CS} that
\begin{equation}\label{detail1}
  \bbE\Bigl[ ( \cZ^{\go}_{\beta,mT})^{\frac12}\Bigr] \le 2^{\frac m2} \sum_{ \bj\in (\bbZ^d)^{m}}\sqrt{\bbE\Bigl[ G_{\bj}(\go) \cZ^{\go}_{\beta,mT}(\chi_{\bj})\Bigr]}.
\end{equation}
Using the size-biased representation of Lemma~\ref{lem:sizebias} and recalling the notation~\eqref{shorthat}, we obtain that 
\begin{equation*}
\bbE\Bigl[ G_{\bj}(\go) \cZ^{\go}_{\beta,mT}(\chi_{\bj})\Bigr]
=    \bbE \otimes \bbE'_{\beta} \otimes\bQ\Bigl[ \wh G_{\bj}(\go,\go',B) \chi_{\bj}(B) \Bigr].
\end{equation*}
Letting $B^{(i)}_t:=B_{(i-1)T+t}-j_{i-1}\sqrt{T}$, using the stationarity of $\go$ and $\go'$  and the fact that $\cA$ has time-range $T$, we have
\begin{align*} 
 \bbE \otimes \bbE'_{\beta} \otimes\bQ\Bigl[ \wh G_{\bj}(\go,\go',B) \chi_{\bj}(B) \Bigr] &=\bbE \otimes \bbE'_{\beta} \otimes\bQ\Biggl[ \prod_{i=1}^m \wh g_{i,j_{i-1}}(\om,\om',B^{(i)})\bone_{\{B^{(i)}_T\in \cC_{j_i-j_{i-1}}\}}\Biggr]  \\
 &=\bQ \Biggl[ \prod_{i=1}^m \bbE \otimes \bbE'_{\beta} \Bigl[ \wh g_{i,j_{i-1}}(\go,\go', B^{(i)}) \ind_{\{B^{(i)}_T\in \cC_{j_i-j_{i-1}} \}}\Bigr] \Biggr]\\
& = \bQ \Biggl[ \prod_{i=1}^m \bbE \otimes \bbE'_{\beta} \Bigl[ \wh g(\go,\go', B^{(i)}) \ind_{\{B^{(i)}_T\in \cC_{j_i-j_{i-1}} \}}\Bigr] \Biggr].
 \end{align*}
Conditionally on $(B_t)_{t\in[0,(i-1)T]}$, $(B^{(i)}_t)_{t\ge0}$ is a Brownian motion starting from $B_{(i-1)T} - j_{i-1}\sqrt{T}$, which belongs to $\cC_0$. Thus, 
 using the   Brownian motion's Markov property  iteratively (starting with $i=m$), we obtain that 
\begin{equation*} 
\bbE \otimes \bbE'_{\beta} \otimes\bQ\Bigl[ \wh G_{\bj}(\go,\go',B) \chi_{\bj}(B) \Bigr]
\le
 \prod_{i=1}^m \max_{x\in \cC_{0}} \bbE \otimes \bbE'_{\beta} \otimes\bQ_x\Bigl[  \wh g(\go,\go',B) \ind_{\{B_T\in \cC_{j_i-j_{i-1}} \}} \Bigr]\, .
\end{equation*}
Reindexing the sums yields
\begin{equation}\label{detail2}
\sum_{ \bj\in (\bbZ^d)^{m}}\sqrt{\bbE\Bigl[ G_{\bj}(\go) \cZ^{\go}_{\beta,mT}(\chi_{\bj})\Bigr]}  \le \Biggl( \sum_{j\in \bbZ^d}\max_{x\in \cC_{0}} \sqrt{ \bbE \otimes \bbE'_{\beta} \otimes\bQ_x\Bigl[  \wh g(\go,\go',B) \ind_{\{B_T\in \cC_{j} \}} \Bigr]} \Biggr)^m.
 \end{equation}
Hence, it remains to show that the sum on the right-hand side is small.
Given an integer $N$, since $g\le 1$, we have 
\begin{equation} \mathtoolsset{multlined-width=0.9\displaywidth}\begin{multlined}
\sum_{j\in \bbZ^d}\max_{x\in \cC_{0}} \sqrt{ \bbE \otimes \bbE'_{\beta} \otimes\bQ_x\Big[  \wh g(\go,\go',B) \ind_{\{B_T\in \cC_{j} \}} \Big]}
\\ \le  \sum_{\| j\|_{\infty} \ge N}\max_{x\in \cC_{0}} \sqrt{ \bQ_x[B_T\in \cC_j]} 
+  (2N)^{d} \max_{x\in \cC_{0}}\sqrt{ \bbE \otimes \bbE'_{\beta} \otimes\bQ_x\Big[  \wh g(\go,\go',B) \Big]}. \end{multlined}
 \end{equation}
We then choose $N$ to be the smallest integer such that the first sum
 is smaller than $(4\sqrt{2})^{-1}$; note it depends on $d$ but not  on $T$. For the second term, recalling the definition~\eqref{def:ggg} of $g$ and using   assumption \eqref{H2} for the last inequality,
we have
\begin{equation}\label{detail3}
\bbE \otimes \bbE'_{\beta} \otimes\bQ_x[  \wh g(\go,\go',B) ]
=e^{-K}+ (1-e^{-K})\P \otimes \P'_{\beta} \otimes\bQ_x( \wh \go \notin \cA )
\le 2 e^{-K}.
\end{equation}
Finally, combining \eqref{detail1}--\eqref{detail3},
 we obtain that
 \[
 \bbE\Bigl[ ( \cZ^{\go}_{\beta,mT})^{\frac12}\Bigr]  \le 2^{\frac m2}\Biggl( \frac{1}{4\sqrt{2}} +  (2N)^d  \sqrt{2 e^{-K}}  \Biggr)^m = 2^{-m} \, ,
 \]
where the last inequality holds by choosing
$K= 2 [ (3+d) \log 2 + d\log N ]$.
 \end{proof}

\subsection{One-body estimates}
\label{sec:onebody}

  We now present a  simple choice for $f$ that yields the desired bounds on the Lyapunov exponents both in dimension $d=1$ and in any dimension $d\geq 2$ for heavy-tailed environments. This gives the second half of Theorem~\ref{thm:thinup} \textit{(i)} and of Theorem~\ref{thm:heavyup}, that is, the lower bounds on $|\gamma_{\gb}(p)|$. More precisely, we prove the following.
 
 \begin{proposition}
	\label{prop:fractional} Assume that $\mu <\infty$.
	\begin{enumerate}
		\item[(i)]  In dimension $d=1$,  
		we have 
		\begin{equation}\label{onehalfcase}
		 \limsup_{\beta \to 0+} \beta^{-4} \bar \gamma_{\beta}( {\textstyle\frac12})<0.
		\end{equation}
		\item[(ii)]  In any dimension $d\geq 1$, if for some $\alpha < 1+\frac 2d$, 
		we have 
		$\liminf_{A\to \infty} A^{\alpha} \lambda([A,\infty)) >0,$
		 then 
		 	\begin{equation}\label{assumliminf}
		 \limsup_{\beta \to 0+} \beta^{-\frac{\alpha}{\nu_{\alpha}}} \bar \gamma_{\beta}({\textstyle\frac12})<0.
		\end{equation}
		If we only assume    
		$\limsup_{A\to \infty} A^{\alpha} \lambda([A,\infty)) >0$,
		then 
	\begin{equation}\label{assumlimsup}
		 \liminf_{\beta \to 0+} \beta^{-\frac{\alpha}{\nu_{\alpha}}} \bar \gamma_{\beta}({\textstyle\frac12})<0.
		\end{equation}
		
	\end{enumerate}
\end{proposition}

\begin{proof} 
The proofs of \textit{(i)} and \textit{(ii)} slightly differ but have the same starting point, that is, the same function $f$ (to which we will apply Corollary~\ref{cheapcoro}), up to a  choice of parameter. 

\smallskip

\noindent \textit{Preparing the setup.}
We start by fixing $R>1$ such that 
\begin{equation}\label{R32}
\Bigg((2\pi)^{-\frac12}\int_{[-R,R-1]}e^{-\frac{u^2}{2}}\, \dd u \Bigg)^{d} \ge 1-\frac{1}{64} \, e^{-K}.
		 \end{equation}   
where $K$ is the constant from Lemma~\ref{lem:changemeas}. Given $T>0$ and $0<a<b\leq \infty$,
we then consider the functional
\begin{equation}\label{classikf}
f(\go) := \int_{\X^{(1)}_T} \ind_{\{ \lVert x\rVert_{\infty}\le R\sqrt{T}, z\in [a,b)\}}\, \delta_\om(\dd t,\dd x,\dd z),
\end{equation}
which is Poisson distributed with
\begin{equation}
\label{expectvarf}
	\bbE[f(\go)]=\var( f(\go))= \gl([a,b))(2R\sqrt{T})^d T.
\end{equation}
Using the formalism of the previous subsection, we further have 
\begin{equation*}
 f'(\go',B) =  \int_{(0,T)\times (0,\infty)} \ind_{\{ \|B_t\|_{\infty}\le R\sqrt{T}, z\in [a,b)\}}\, \delta_{\om'}(\dd t,\dd z). 
\end{equation*}
To estimate the expectation and variance of $f'$,
we define 
\begin{equation*}
 \bar f(\go'):=\int_{(0,T)\times (0,\infty)}\ind_{\{z\in [a,b)\}}\, \delta_{\om'}(\dd t,\dd z).
\end{equation*}
Note that we have $f'(\go',B)\leq  \bar f(\go')$ and, recalling that $\om'$ has intensity measure $\beta z\,\dd t\,\la(\dd z)$, also
\begin{equation}\label{momentbarf}
\bbE'_{\beta}[\bar f(\go')]=       \beta \mu_{a,b}(1)T,\qquad \bbE'_{\beta}[\bar f(\go')^2]=       \beta \mu_{a,b}(1)T +  (  \beta \mu_{a,b}(1)T )^2.
\end{equation}
The condition \eqref{R32}   guarantees that for any $x\in [0,\sqrt{T})^d$ and $t\in(0,T]$
\begin{equation*}
 \bQ_x\Big( \lVert B_t\rVert_{\infty}\le R\sqrt{T}\Big)\ge 1-\frac{1}{64} \, e^{-K} \,,
\end{equation*}
so that we have
\begin{equation}\label{espefprim}
 \bbE'_{\beta}\otimes \bQ_x [ f'(\go',B)  ] \ge \bigg(1-\frac{1}{64} \, e^{-K}\bigg)\beta \mu_{a,b}(1)T 
 \geq \frac{1}{\sqrt{2}}\,  \beta \mu_{a,b}(1)T  \,.
\end{equation}
We also have, combining \eqref{momentbarf} and the first inequality of \eqref{espefprim},
\begin{equation}\label{varvar}
  \var_{\bbP'_{\beta}\otimes \bQ_x} (f'(\go',B) )  \le 
 \bbE'_{\beta}[\bar f(\go')^2]-  \bbE'_{\beta}\otimes \bQ_x [ f'(\go',B)  ]^2 \le 
  \beta \mu_{a,b}(1)T + \frac{e^{-K}}{32}  (  \beta \mu_{a,b}(1)T )^2.
\end{equation}
We are now ready to apply these estimates to specific cases.
%

\medskip

\noindent \textit{Proof of \textit{(i)}.}
We first  choose $a$ and $b$ such that
$\gl([a,b))>0$ (in a way that does not depend on $T$ or $\beta$)
and then we let 
\begin{equation*}
 T=T(\beta):= \frac{2^{10} R^2e^{2K} \gl([a,b))^2}{\beta^4 \mu_{a,b}(1)^4}.
\end{equation*}
With this choice and for $\beta$ sufficiently small, the second term on the right-hand side of  \eqref{varvar} dominates: using \eqref{espefprim} for the second inequality, 
we get 
\begin{equation*}
  \var_{\bbP'_{\beta}\otimes \bQ_x} (f'(\go',B) )\le
  \frac{e^{-K}}{16}  (  \beta \mu_{a,b}(1)T )^2 \le \frac {e^{-K}}{8} \bbE'_{\beta}\otimes \bQ_x [ f'(\go',B)  ]^2.
\end{equation*}
Now the reader can check that with our definition of $T$, we have (recall~\eqref{expectvarf})
\begin{equation*}
 \var( f(\go)) = 2R\gl([a,b)) T^{\frac32} = \frac{e^{-K}}{16}   (  \beta \mu_{a,b}(1)T )^2  \le \frac {e^{-K}}{8} \bbE'_{\beta}\otimes \bQ_x [ f'(\go',B)  ]^2.
\end{equation*}
Therefore, condition \ref{(C1)} in  Corollary \ref{cheapcoro} is indeed satisfied for small $\beta$, and we thus have 
\begin{equation*}
		 \limsup_{\beta \to 0+} \beta^{-4} \bar \gamma_{\beta}({\textstyle \frac12})\le -\frac{(\log 2) \mu_{a,b}(1)^4}{2^{10} R^2e^{2K} \gl([a,b))^2}.
\end{equation*}

\medskip

\noindent \textit{Proof of \textit{(ii)}.}
We let $\ell>0$ denote the limit superior/inferior of $A^{-\alpha} \gl([A,\infty))$. We choose 
\begin{equation}\label{deftheA}
A=A(\beta):= M \, \beta^{-\frac{d+2}{2-d(\alpha-1)}} \,,
\end{equation}
where $M=M(K,R,\ell)$ is a large constant that depends on $K$, $R$ and $\ell$.
For the remainder of the proof, we assume that $\beta$ is such that we have, 
for the value of $A$ specified above,
\begin{equation}\label{checkthelimit}
\gl([A,\infty))\ge \frac{\ell A^{-\alpha}}{2} \,.
\end{equation}
(We assume that $\beta$ is sufficiently small if the assumption is about the limit inferior, and we take $\beta$ along a well-chosen subsequence   tending to $0$ if the assumption is about the limit superior.)
We are going to use Corollary \ref{cheapcoro} for the function \eqref{classikf}, with the choices $a=A$, $b=\infty$ and 
\begin{equation}\label{eq:Tbound}
T:=  \Big[ e^{-K} (2R)^{-d} \gl([A,\infty))^{-1}\Big]^{\frac{2}{2+d}}.
\end{equation}
With our choice of $A$ and our tail assumption, we have 
\begin{equation*}
 T\le \bigg[ e^{-K} (2R)^{-d} \frac{2}{\ell} M^{\alpha}\bigg]^{\frac{2}{2+d}}\beta^{-\frac{\alpha}{\nu_{\alpha}}}\,.
\end{equation*}
Therefore, once we show that    \ref{(C2)}   is satisfied, Corollary~\ref{cheapcoro} yields the desired bound on $\bar \gamma_{\beta}(\frac12)$.

For the expectation of $f$, by~\eqref{expectvarf}, our choice for $T$  directly yields
$$\bbE[f(\go)]=(2R)^d\gl([A,\infty))T^{1+\frac{d}{2}}= e^{-K}.$$
To check the condition on the variance of $f'$, in view of \eqref{espefprim} and \eqref{varvar}, it is sufficient to have
\begin{equation*} 
 \beta \mu_{A,\infty}(1) T \ge 4 e^K.
 \end{equation*}
To this end, we bound $\mu_{A,\infty}(1) \ge A\gl([A,\infty))$ and use the definition~\eqref{eq:Tbound} of $T$
to get that $\beta \mu_{A,\infty}(1) T$ is bounded below by a constant (that depends on $K,R$) times $\gb A \lambda([A,\infty))^{d/(2+d)}$. By~\eqref{checkthelimit} and 
the definition of $A$ in \eqref{deftheA}, the latter can be made large by choosing $M=M(K,R,\ell)$ sufficiently large.
 \end{proof}

\subsection{Multi-body estimates}
\label{sec:multibod}

The aim of this section is to prove the following.
	\begin{proposition}\label{prop:dgeq2}
		Suppose that $d\geq 2$ and $\mu<\infty$. Then 
		\begin{equation}\label{eq:ub-loglog} \limsup_{\beta\to0}\frac{\log\lvert\log \lvert\bar \gamma_\beta(\frac12)\rvert\rvert}{\lvert\log \beta\rvert} \leq 1+\tfrac 2d. \end{equation}
	\end{proposition}

With Propositions~\ref{prop:fractional} and \ref{prop:dgeq2} at hand, we can complete the proof of Theorems \ref{thm:thinup} and
\ref{thm:heavyup}. Recall that the first half of the proof was given in Sections \ref{sub:fh} and \ref{sub:fattails}, respectively.

 \begin{proof}[Proof of  Theorems \ref{thm:thinup} and
	\ref{thm:heavyup} (second half)] 
	 By \eqref{convexinho2}, the lower bounds in \eqref{eq:d1} and \eqref{eq:1+2d}  directly follow  from \eqref{onehalfcase} and \eqref{eq:ub-loglog}. Next, \eqref{eq:lastone} follows from \eqref{assumlimsup} since 
	$\mu_{1,\infty}(q)=\infty$ implies that $\limsup_{A\to \infty} A^{\alpha} \lambda([A,\infty)) >0$ for every $\alpha>q$.
	Finally, the upper bound in \eqref{eq:stable}   follows from \eqref{eq:lastone} because \eqref{eq:alpha} implies $\mu_{1,\infty}(q)=\infty$ for all $q>\al$.
\end{proof}

\begin{proof}[Proof of Proposition \ref{prop:dgeq2}]
In  Section~\ref{sec:onebody}, counting the number of atoms  in $\om$  with size in a certain range was sufficient to determine the correct order of magnitude (in $\beta$) of $\bar \gamma_{\beta}(\frac12)$ (and hence of $\bar\gamma_\beta(p)$ for other values of $p$). If $d\geq2$ and if $\la$ is light-tailed, in order to obtain optimal estimates, we must consider certain \emph{clusters} of atoms. Furthermore, the size of these point configurations, which we denote by  $k$,  must be taken  large.
 We shall obtain our result \eqref{eq:ub-loglog} by proving that for any fixed $k\ge 3$, we have
		\begin{equation*}
		\limsup_{\beta\to0} \frac{\log\lvert\log \lvert \bar\gamma_\beta(\frac12)\rvert\rvert}{\lvert\log \beta\rvert} \leq (1+\tfrac2d) \frac{k}{k-2}.
		\end{equation*}
This will be achieved again by using Corollary~\ref{cheapcoro}. We need to introduce a few parameters for the functional $f$ used in this proof. First,  we   fix $0<a<b<\infty$ such that $\la([a,b))> 0$ as well as two parameters $0<\kappa<\kappa'$ satisfying
		\begin{equation}\label{condkap}
		 \kappa\in \bigg(\frac{1}{k}, \frac{1}{k-1}\bigg) \qquad 
		 \text{and}  \qquad (k-2)\kappa'+\kappa<1.
		\end{equation}
		  Furthermore, we  let $R=R(k,d)$   be  such that
		 \begin{equation}\label{RK}
		\Bigg((2\pi)^{-\frac12}\int_{[-\sqrt{2}R,\sqrt{2}(R-1)]}e^{-\frac{u^2}{2}} \,\dd u \Bigg)^{d}
		\Bigg((2\pi)^{-\frac12}\int_{[-R,R]}e^{-\frac{u^2}{2}}\, \dd u \Bigg)^{d(k-1)}\ge 1-\frac{e^{-K}}{8},
		 \end{equation}
	 where $K$ is the constant from Lemma~\ref{lem:changemeas}.
		 With this choice of $R$, we have, for any $T>0$ and $\bt\in \frakX_k(T)$ such that  $t_1\le \frac T2$,
		 \begin{equation}\label{thisisgood}
		  \inf_{x\in[0,\sqrt{T})^d} \bQ_x\Big( \lVert B_{t_1}\rVert_{\infty}\le R\sqrt{T}, \ \forall i\in \lint 2,k\rint: \lVert\Delta B_{t_i}\rVert_{\infty}\le R \sqrt{\Delta t_i} \Big)\ge 1- \frac{e^{-K}}{8}.
		  \end{equation}
Finally, for some small $\eps\in (0,1)$ (independent of $\beta$ and to be chosen later), we set
\begin{equation}\label{settingT}
 T:=\exp\bigg( \gep^{-\frac{4+d}{4(k-2)}} \beta^{-(1+\frac2d)\frac{k}{k-2}} \bigg).
\end{equation}
	        Using the notation $v_I := \prod_{i\in I} v_i$ for $I\subseteq \lint k\rint$,  we consider the multi-body functional
		\begin{equation}
		\label{eq:f}
		\mathtoolsset{multlined-width=0.9\displaywidth}\begin{multlined} f(\om)= \int_{\X^{(k)}_T} \bone_{\{\forall i \in \lint k \rint:\ z_i\in [a,b)\}} \bone_{\{ t_1\le \frac T2,\ \forall i\in\lint 2,k-1\rint:\ \Delta t_i \in  [T^{-\kappa'},T^{-\kappa}] , \  \Delta t_{\lint 2, k\rint}\le  \gep \beta^{ {2k}/{d}} T^{-1}\}}\\
				\times\bone_{\{  \lVert x_1\rVert_\infty \le R \sqrt{T}, \ \forall i\in\lint 2, k\rint:\ \lVert \Delta x_i\rVert_\infty\leq R\sqrt{\Delta t_i} \}} \prod_{i=1}^k\delta_\om(\dd t_i,\dd x_i,\dd z_i).
				 \end{multlined}
		\end{equation}	 
This function $f$ counts clusters of atoms that are so close to each other (in time) that the benefit of visiting all of them, which is $\prod_{i=2}^k z_i\rho(\Delta t_i, \Delta x_i)$ or  roughly $(\Delta t_{\lint 2, k\rint})^{-d/2}$, outweighs the cost of visiting the first atom of the group, which is of order $T^{-d/2}$ in the bulk of the box.
		Limiting the possible range for  $\Delta t_i$ is convenient in the computations and our particular choice for this limitation is largely \textit{ad hoc}.	
	If
		\begin{equation}\label{lacondition}
	  \gep \beta^{\frac{2k}{d}} T^{-1+(k-2) \kappa'}\le T^{-\kappa}, 
	\end{equation}	
	then the conditions on the $t_i$'s in the first line of \eqref{eq:f} imply $\Delta t_k \le T^{-\kappa}$.
	With our choice~\eqref{settingT}  for $T$ and \eqref{condkap}, the condition \eqref{lacondition} is satisfied for any $\beta<1$.
		
		\medskip
		
Let us now check that   condition \ref{(C2)} in Corollary \ref{cheapcoro} is satisfied. We start with the condition on  $\bbE[f(\go)]$.
Integrating over the $x_i$'s and $t_1$, making the change of variable $s_i=\Delta t_{i+1}$ and using that we have $\Delta t_k\leq T^{-\kappa}$, we get 
		\begin{equation}\label{calculexpect}
		 \bbE[f(\go)]=  \frac{T^{1+\frac{d}{2}}}{2}\Big[ (2R)^d \gl([a,b))\Big]^k \int_{[T^{-\kappa'},T^{-\kappa}]^{k-2}\times[0,T^{-\kappa}]}  \bone_{\{ s_{\lint k-1\rint}\le  \gep \beta^{ {2k}/{d}} T^{-1}\}} \prod_{i=1}^{k-1} s_i^{\frac d2}\, \dd s_i.
		\end{equation}
	Next, we   integrate with respect to $s_{k-1}$ and use \eqref{lacondition} to obtain
		\begin{equation}\label{calculexpect2}\mathtoolsset{multlined-width=0.9\displaywidth}\begin{multlined}
			T^{1+\frac{d}{2}} \int_{[T^{-\kappa'},T^{-\kappa}]^{k-2}\times[0,T^{-\kappa}]}  \bone_{\{  s_{k-1}\le  \gep \beta^{ {2k}/{d}} T^{-1} / s_{\lint k-2\rint}\}} \prod_{i=1}^{k-1} s_i^{\frac d2}\, \dd s_i\\
			 =  \frac{\gep^{1+\frac{d}{2}}\beta^{ (1+\frac{2}{d} )k}  }{1+\frac{d}{2}}\int_{[T^{-\kappa'},T^{-\kappa}]^{k-2}}\prod_{i=1}^{k-2} s_i^{-1} \,\dd s_i
			=\frac{\gep^{1+\frac{d}{2}}\beta^{ (1+\frac{2}{d} )k}
				[ (\kappa'-\kappa ) \log T]^{k-2}}{d+2}.\end{multlined}
	\end{equation}
		Therefore, by \eqref{settingT}, there exists a constant $C=C(R,\la,a,b,\kappa,\kappa',k,d)$ such that
		\begin{equation*}
		 	 \bbE[f(\go)]=  C \gep^{1+\frac{d}{2}} \beta^{ (1+\frac{2}{d} )k}(\log T)^{k-2}=C \gep^{\frac{d}{4}}.
		\end{equation*}
   Hence, choosing $\gep\le  (Ce^K )^{- {4}/{d}}$ fulfills the assumption on $\E[f(\om)]$ in \ref{(C2)}; let us stress that $\gep$ must  satisfy a second condition which will be specified later in the proof.

Let us now estimate the expectation and variance of $f'(\go',B)$ (recall the definition~\eqref{thefprim}). We consider an alternative function  
		\begin{equation}\label{expressbarf}
		 \bar f(\go'):=\int_{([0,T]\times[a,b))^{k}} \bone_{\{ t_1\le \frac T2,\ \forall i\in\lint 2,k-1\rint:\ \Delta t_i \in  [T^{-\kappa'},T^{-\kappa}],  \  \Delta t_{\lint 2, k\rint}\le  \gep \beta^{ {2k}/{d}} T^{-1}\}}\prod_{i=1}^k\delta_{\om'}(\dd t_i,\dd z_i)
		\end{equation}
 that does not include any restriction for $B$.
 Using \eqref{thisisgood}, we obtain, similarly to \eqref{espefprim} and \eqref{varvar}, that for any $x\in[0,\sqrt{T})^d$
 \begin{equation} \label{babarf}\begin{split}
   \bbE'_{\beta}\otimes \bQ_x [f'(\go',B) ]&\geq \Bigg(1-\frac{e^{-K}}{8}\Bigg)  \bbE'_{\beta} [\bar f(\go') ] \ge \frac{1}{\sqrt{2}} \bbE'_{\beta} [\bar f(\go') ],\\
   \var_{\bbP'_{\beta}\otimes \bQ_x} (f'(\go',B) )&\le 
   \var_{\bbP'_\beta} (\bar f(\go') )+\frac{e^{-K}}{4} \bbE'_{\beta} [\bar f(\go') ]^2\,.
   \end{split}
 \end{equation}
We can thus conclude that the variance bound in \ref{(C2)}  is satisfied if 
\begin{equation}\label{nuproof}
 \var_{\bbP'_{\beta}} (\bar f(\go') ) \le \frac{e^{-K}}{4}\bbE'_{\beta} [\bar f(\go') ]^2.
\end{equation}
 Proceeding as in \eqref{calculexpect}--\eqref{calculexpect2} and recalling that $\go'$ has intensity measure $\gb z \,\dd t\,\lambda (\dd z)$,
 we have
\beq
\label{barexp}
\begin{split}
 \bbE'_\beta [\bar f(\go') ]&=  (\beta \mu_{a,b}(1))^k \frac{T}{2} \int_{[T^{-\kappa'},T^{-\kappa}]^{k-2}\times[0,T^{-\kappa}]}\bone_{\{ s_{\lint k-1\rint}\le  \gep \beta^{ {2k}/{d}} T^{-1}\}} \prod_{i=1}^{k-1} \dd s_i \\
 &= \frac{\gep}{2}\mu_{a,b}(1)^k\beta^{ (1+\frac{2}{d} )k}    \int_{[T^{-\kappa'},T^{-\kappa}]^{k-2}}  \prod_{i=1}^{k-2} s_i^{-1} \,\dd s_i\\
 &= C' \gep \beta^{ (1+\frac{2}{d} )k} (\log T)^{k-2}=C' \gep^{-\frac{d}{4}}\,,
\end{split}
\eeq
with $C'=C'(\la,a,b,\kappa,\kappa',k,d)$, and where we have used the value~\eqref{settingT} of $T$ for the last identity.
In order to compute the second moment of $\bar f(\go')$, we expand the integral in \eqref{expressbarf} by writing 
\beq\label{eq:expand} \prod_{i=1}^k\delta_{\om'}(\dd t_i,\dd z_i)=\sum_{I\subseteq \lint  k\rint } \prod_{i\in I}\Big( \delta_{\om'}(\dd t_i,\dd z_i)-\beta z_i\,\dd t_i  \,\gl(\dd z_i) \Big)\prod_{j\in I^c}  \beta z_j \,\dd t_j\,\gl(\dd z_j),   \eeq
where $I^c:=\lint k\rint\setminus I$.
We let $ f_I$ denote the integral corresponding to the term $I$.
Since $f_{\emptyset}=\bbE'_{\beta}[\bar f(\go')]$, we have
\begin{equation}\label{eq:var}
 \var_{\P'_\beta}(\bar f(\go'))= \sum_{I,J\subseteq \lint k\rint:\ I,J\ne \emptyset}\bbE'_{\beta} [ f_I(\go')f_{J}(\go') ].
\end{equation}
The terms in the sum above are equal to zero if $I$ and $J$ do not have the same cardinality. If $\lvert I\rvert=\lvert J\rvert$, then   $\bbE'_{\beta} [ f_{I}(\go')f_{J}(\go') ]$ is obtained by matching the values of $(t_i)_{i\in I}$ to that of $(t_i)_{i\in J}$ in the integral before averaging. 
More precisely, 
we have
\begin{multline*}
 \bbE'_\beta [ f_I(\go')f_{J}(\go') ] \\ =  (\beta\mu_{a,b}(1) )^{2k-\lvert I\rvert}  \int_{[0,T]^{2k-\lvert I\rvert}} \bigg(   \ind_{\mathcal Q} \Big((t^{(1)}_i)_{i=1}^k \Big)  \prod_{j\in I^{\cc}} \dd t^{(1)}_j    \bigg)\bigg(   \ind_{\mathcal Q}\Big( (t^{(2)}_i)_{i=1}^k\Big) \prod_{j\in J^{\cc}} \dd t^{(2)}_j    \bigg)\prod_{i\in I} \dd t_i, 
\end{multline*}
where $\mathcal{Q}$ is the subset of $[0,T]^{k}$ induced by the  indicator function in \eqref{expressbarf}, and the vectors $\bt^{(1)}$ and $\bt^{(2)}$ are completed by setting
\begin{equation}\label{fusion}
t^{(1)}_i=t^{(2)}_{\sigma(i)}=t_i \qquad \text{ for } i\in I \,,
\end{equation}
with  $\sigma$  being the unique increasing bijection from $I$ to $J$.
In particular, we have
\begin{equation}\label{fullcase}
 \bbE'_\beta [ f_{\lint k \rint}(\go')^2 ]= \bbE'_{\beta} [\bar f(\go') ].
\end{equation}
When $\lvert I\rvert=\lvert J\rvert=\ell\in \lint k-1\rint $,
we are going to show that  
\begin{equation}\label{boundconstraint}
 \int_{[0,T]^{2k-\lvert I\rvert}} \bigg(   \ind_{\mathcal Q} \Big((t^{(1)}_i)_{i=1}^k \Big)  \prod_{j\in I^{\cc}} \dd t^{(1)}_j    \bigg)\bigg(   \ind_{\mathcal Q}\Big( (t^{(2)}_i)_{i=1}^k\Big) \prod_{j\in J^{\cc}} \dd t^{(2)}_j    \bigg)\prod_{i\in I} \dd t_i \le T^{1-(2k-\ell-1)\kappa}.
\end{equation}
So  combining \eqref{fullcase} with \eqref{boundconstraint}, we obtain that
\begin{equation*}
   \var_{\P'_\beta}(\bar f(\go'))\le \bbE'_{\beta}[\bar f(\go')]+ \sum_{\ell=1}^{k-1}  \binom{k}{\ell}^2  (\beta\mu_{a,b}(1) )^{2k-\ell} T^{1-(2k-\ell-1)\kappa}
   \le \bbE'_{\beta}[\bar f(\go')]+ 1,
\end{equation*}
where the last inequality is valid for sufficiently small $\beta$, because  $1-(2k-\ell-1)\kappa\le 0$ thanks to~ \eqref{condkap}.
For this reason, \eqref{nuproof} is satisfied if $\bbE'_{\beta}[\bar f(\go')]\ge 8 e^K$, which according to \eqref{barexp} holds true provided that $\gep$ is chosen small enough. 

\medskip

To complete the proof, let us show how  \eqref{boundconstraint} is obtained: the key is to integrate
in the correct order the different variables.
We construct a graph $G_{I,J}$ with $2k-\ell$ vertices where each vertex is identified with one of the variables $(t_i)_{i\in I}$, $(t^{(1)}_i)_{i\in I^{\cc}}$,  $(t^{(2)}_i)_{i\in J^{\cc}}$. We start with two (initially disconnected)  paths of $k$ vertices labeled $(t^{(1)}_i)_{i=1}^k$ and $(t^{(2)}_i)_{i=1}^k$, respectively,
the edges being $\{t^{(r)}_i, t^{(r)}_{i+1}\}$ for $i\in \lint k-1\rint$.
Then we glue the two paths together by identifying the vertices $(t^{(1)}_i)_{i\in I}$ with $(t^{(2)}_i)_{i\in J}$ according to \eqref{fusion}; any double edge that might have been created is replaced by a single edge.
As $I\ne \emptyset$, the graph we obtain is connected.
Now we consider an enumeration $v_1,v_2,v_3,\dots, v_{2k-\ell}$ of the vertices of $G_{I,J}$  such that for every $j\in \lint   2k-\ell-1\rint$, the subgraph $G^{(j)}_{I,J}$ of $G_{I,J}$   induced by $v_{j+1},\dots, v_{2k-\ell}$ is connected. One way of finding such an enumeration is, for example, to construct a spanning tree $T_{I,J}$ of $G_{I,J}$ and then, after $v_1,\dots, v_{j}$ have been determined for some $j\in\{0,\dots,2k-\ell-1\}$, to take any leaf of $T_{I,J}\setminus\{v_1,\dots,v_j\}$ as $v_{j+1}$.
We refer to Figure~\ref{fig:2} for one example of $G_{I,J}$ together with a permitted and a forbidden enumeration according to the rule we just introduced.
\begin{figure}[h!]
	\centering
		\begin{subfigure}[b]{\textwidth}
		\centering
		\begin{tikzpicture}[scale=0.6]
		   \draw (-2,-1) node[below] {(a)};
			\node[circle, inner sep = 1.5pt, fill = black] (12) at (0,-3) {}; 
			\node[circle, inner sep = 1.5pt, fill = black] (1122) at (2,-1.5) {};
			\node[circle, inner sep = 1.5pt, fill = black] (21) at (4,0) {};
			\node[circle, inner sep = 1.5pt, fill = black] (3) at (6,-1.5) {}; 
			\node[circle, inner sep = 1.5pt, fill = black] (41) at (8,0) {}; 
			\node[circle, inner sep = 1.5pt, fill = black] (42) at (8,-3) {};  
			\node[circle, inner sep = 1.5pt, fill = black] (51) at (10,0) {}; 
			\node[circle, inner sep = 1.5pt, fill = black] (52) at (10,-3) {}; 
			\node[circle, inner sep = 1.5pt, fill = black] (6) at (12,-1.5) {};
			\node[circle, inner sep = 1.5pt, fill = black] (7) at (14,-1.5) {};
			\node[circle, inner sep = 1.5pt, fill = black] (81) at (16,0) {}; 
			\node[circle, inner sep = 1.5pt, fill = black] (82) at (16,-3) {}; 
			\draw (0,-3) node[below] {\small $t^{(2)}_1$};
			\draw (2,-1.5) node[above] {\small $t^{(1)}_1$};
			\draw (2,-1.5) node[below] {\small $t^{(2)}_2$};
			\draw (4,0) node[above] {\small $t^{(1)}_2$};
			\draw (6,-1.5) node[above] {\small $t^{(1)}_3$};
			\draw (6,-1.5) node[below] {\small $t^{(2)}_3$};
			\draw (8,0) node[above] {\small $t^{(1)}_4$};
			\draw (8,-3) node[below] {\small $t^{(2)}_4$};
			\draw (10,0) node[above] {\small $t^{(1)}_5$};
			\draw (10,-3) node[below] {\small $t^{(2)}_5$};
			\draw (12,-1.5) node[above] {\small $t^{(1)}_6$};
			\draw (12,-1.5) node[below] {\small $t^{(2)}_6$};
			\draw (14,-1.5) node[above] {\small $t^{(1)}_7$};
			\draw (14,-1.5) node[below] {\small $t^{(2)}_7$};
			\draw (16,0) node[above] {\small $t^{(1)}_8$};
			\draw (16,-3) node[below] {\small $t^{(2)}_8$};
			\draw[ultra thick] (12) -- (1122) -- (21) -- (3) -- (41) -- (51) -- (6) -- (7) -- (81);
			\draw (3) -- (1122);
			\draw[ultra thick] (3) -- (42) -- (52);
			\draw (6) -- (52);
			\draw[ultra thick] (7) -- (82);
			%
		\end{tikzpicture}
	\end{subfigure}	
	\vspace{1\baselineskip}
	\begin{subfigure}[b]{\textwidth}
		\centering
		\begin{tikzpicture}[scale=0.6]
				   \draw (-2,-1) node[below] {(b)};
			\node[circle, inner sep = 1.5pt, fill = black] (12) at (0,-3) {}; 
			\node[circle, inner sep = 1.5pt, fill = black] (1122) at (2,-1.5) {};
			\node[circle, inner sep = 1.5pt, fill = black] (21) at (4,0) {};
			\node[circle, inner sep = 1.5pt, fill = black] (3) at (6,-1.5) {}; 
			\node[circle, inner sep = 1.5pt, fill = black] (41) at (8,0) {}; 
			\node[circle, inner sep = 1.5pt, fill = black] (42) at (8,-3) {};  
			\node[circle, inner sep = 1.5pt, fill = black] (51) at (10,0) {}; 
			\node[circle, inner sep = 1.5pt, fill = black] (52) at (10,-3) {}; 
			\node[circle, inner sep = 1.5pt, fill = black] (6) at (12,-1.5) {};
			\node[circle, inner sep = 1.5pt, fill = black] (7) at (14,-1.5) {};
			\node[circle, inner sep = 1.5pt, fill = black] (81) at (16,0) {}; 
			\node[circle, inner sep = 1.5pt, fill = black] (82) at (16,-3) {}; 
			\draw (0,-3) node[above] {\small{$v_1$}};
			\draw (2,-1.5) node[above] {\small{$v_2$}};
			\draw (4,0) node[above] {\small{$v_3$}};
			\draw (6,-1.5) node[above] {\small{$v_6$}};
			\draw (8,0) node[above] {\small{$v_7$}};
			\draw (8,-3) node[above] {\small{$v_5$}};
			\draw (10,0) node[above] {\small{$v_{8}$}};
			\draw (10,-3) node[above] {\small{$v_{4}$}};
			\draw (12,-1.5) node[above] {\small{$v_{9}$}};
			\draw (14,-1.5) node[above] {\small{$v_{12}$}};
			\draw (16,0) node[above] {\small{$v_{10}$}};
			\draw (16,-3) node[above] {\small{$v_{11}$}};
			\draw[ultra thick] (12) -- (1122) -- (21) -- (3) -- (41) -- (51) -- (6) -- (7) -- (81);
			\draw (3) -- (1122);
			\draw[ultra thick] (3) -- (42) -- (52);
			\draw (6) -- (52);
			\draw[ultra thick] (7) -- (82);
		\end{tikzpicture}
	\end{subfigure}
	\vspace{1\baselineskip}
\begin{subfigure}[b]{\textwidth}
	\centering
	\begin{tikzpicture}[scale=0.6]
		\draw (-2,-1) node[below] {(c)};
		\node[circle, inner sep = 1.5pt, fill = black] (12) at (0,-3) {}; 
		\node[circle, inner sep = 1.5pt, fill = black] (1122) at (2,-1.5) {};
		\node[circle, inner sep = 1.5pt, fill = black] (21) at (4,0) {};
		\node[circle, inner sep = 1.5pt, fill = black] (3) at (6,-1.5) {}; 
		\node[circle, inner sep = 1.5pt, fill = black] (41) at (8,0) {}; 
		\node[circle, inner sep = 1.5pt, fill = black] (42) at (8,-3) {};  
		\node[circle, inner sep = 1.5pt, fill = black] (51) at (10,0) {}; 
		\node[circle, inner sep = 1.5pt, fill = black] (52) at (10,-3) {}; 
		\node[circle, inner sep = 1.5pt, fill = black] (6) at (12,-1.5) {};
		\node[circle, inner sep = 1.5pt, fill = black] (7) at (14,-1.5) {};
		\node[circle, inner sep = 1.5pt, fill = black] (81) at (16,0) {}; 
		\node[circle, inner sep = 1.5pt, fill = black] (82) at (16,-3) {}; 
		\draw (0,-3) node[above] {\small {$v_1$}};
		\draw (2,-1.5) node[above] {\small {$v_2$}};
		\draw (4,0) node[above] {\small{$v_3$}};
		\draw (6,-1.5) node[above] {\small{$v_4$}};
		\draw (8,0) node[above] {\small{$v_7$}};
		\draw (8,-3) node[above] {\small{$v_5$}};
		\draw (10,0) node[above] {\small{$v_{8}$}};
		\draw (10,-3) node[above] {\small{$v_{6}$}};
		\draw (12,-1.5) node[above] {\small{$v_{9}$}};
		\draw (14,-1.5) node[above] {\small{$v_{10}$}};
		\draw (16,0) node[above] {\small{$v_{11}$}};
		\draw (16,-3) node[above] {\small{$v_{12}$}};
		\draw[ultra thick] (12) -- (1122) -- (21) -- (3) -- (41) -- (51) -- (6) -- (7) -- (81);
		\draw (3) -- (1122);
		\draw[ultra thick] (3) -- (42) -- (52);
		\draw (6) -- (52);
		\draw[ultra thick] (7) -- (82);
	\end{tikzpicture}
\end{subfigure}
	\captionsetup{width=.9\linewidth}
	\caption{\footnotesize (a) The graph $G_{I,J}$ with $k=8$, $I=\{1,3,6,7\}$ and $J=\{2,3,6,7\}$ and a spanning tree $T_{I,J}$ (thick edges); (b) a permitted enumeration of $G_{I,J}$  based on~$T_{I,J}$; (c)   a forbidden enumeration of $G_{I,J}$  based on~$T_{I,J}$. Integration in the forbidden  order only yields a suboptimal bound: 
 once $v_1,\dots,v_4$ have been integrated (\textit{i.e.}, removed from the graph), we obtain two disconnected components, so if  integration is continued, $v_6$ becomes isolated and contributes a factor $T$ when integrated
(similarly, both $v_{11}$ and $v_{12}$ become isolated when $v_1,\ldots,v_{10}$ have been integrated: each of them contributes a factor of $T$.
}
	\label{fig:2}
\end{figure}
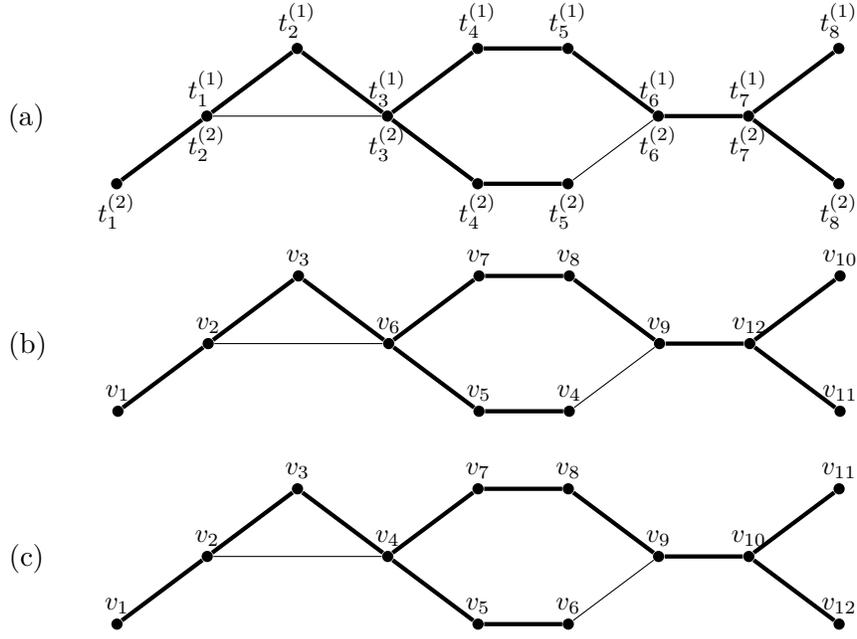

Having fixed such an enumeration, in order to
obtain \eqref{boundconstraint}, we integrate $v_1, \dots, v_{2k-\ell}$ in this order. 
Note that if $v_j= t^{(r)}_i$ for some $j\in\lint 2k-\ell-1\rint$, $r\in\{1,2\}$ and $i\in I^c$, then, by construction, 
\[
\{ t^{(r)}_{i-1}, t^{(r)}_{i+1}\} \cap\{ v_{j+1},\dots, v_{2k-\ell}\}\ne \emptyset ;
\]
indeed, if both $t_{i-1}^{(r)}$ and $t_{i+1}^{(r)}$ had been picked before, then $v_j=t_i^{(r)}$ would be an isolated vertex in $G_{I,J}^{(j-1)}$ because $i\in I^c$. 
In the same manner, if $v_j= t_i$ for some $j\in \lint 2k-\ell-1\rint$ and $i\in I$, then 
\[
\{ t^{(1)}_{i-1}, t^{(1)}_{i+1},t^{(2)}_{\sigma(i)-1},t^{(2)}_{\sigma(i)+1}\} \cap\{ v_{j+1},\dots, v_{2k-\ell}\}\ne \emptyset.
\]
As a result, for each $j\in\lint 2k-\ell-1\rint$, the range when we integrate with respect to  the variable $v_j$ and the values of the variables  $v_{j+1},\dots, v_{2k-\ell}$ are fixed is included in an interval of length~$T^{-\kappa}$:
Indeed,
by \eqref{lacondition},
the indicator function~$\ind_{\mathcal Q}$ in the integral implies   
that $\Delta t^{(r)}_i\le T^{-\kappa}$ for all $i\in \lint 2,k\rint$   
and $r\in \{ 1,2\}$.
We no longer have constraints only for the last variable $v_{2k-\ell}$, which yields a factor of~$T$.
\end{proof}

\subsection{Proof of Proposition~\ref{prop:waitingforbetter}}
\label{sec:waitingforbetter}

First, notice that it suffices to consider the free-end partition function.  Since  $(\cZ^{\go,a}_{\beta}(t,x))_{a\in(0,1]}\stackrel{(d)}{=} (e^{- {\|x\|^2}/({2t})}\cZ^{\go,a}_{\beta}(t,0))_{a\in(0,1]}$, the a.s.\ positivity of the limit when $a\to 0$ does not depend on $x$. Hence the result in the point-to-point case follows from the free-end case by Fatou's Lemma.
 Although this is not essential to the reasoning, it is practical for our computations to assume that $\int_{(0,1)}z^{1+ {2}/{d}}\, \gl(\dd z)<\infty$.
If this assumption is violated, then the result has already been proved in  \cite[Proposition~2.10]{BL20_cont}.
We also assume that $\mu <\infty$: the general case follows from a truncation procedure, see the end of Section~5 in \cite{BL20_cont}. For notational simplicity, we actually assume that $\mu=0$ (\textit{i.e.}, that $\la$ has support in $(0,1)$).

\smallskip

While the proof does not rely on Lemma \ref{lem:changemeas}, it partially builds on the same idea, which is to show that the size-biased measure transforms an atypical event under the original probability measure into a typical one. Let us summarize this in a lemma that is similar in spirit to Corollary~\ref{cheapcoro}. Recall \eqref{filtraG} and the notation $\cZ^{\go,a}_{\beta,t}:= \bar \cZ^{\go,a}_{\beta}(t,\ast)$.

\begin{lemma}\label{lem:help}
Let us assume that there exists a sequence of integer-valued functions $(f_n)_{n\in\N}$   of the form \eqref{padrao}  (with $\Upsilon$ integer-valued and with $T$ replaced by $t$) with the following properties:
\begin{itemize}
 \item [(i)] For every $n$, $f_n$ is $\cG_{a}$-measurable for some $a=a_n>0$.
 \item [(ii)] Using the notation \eqref{thefprim}, we have that
  \begin{equation}\label{limiitz}
  \lim_{n\to \infty}\bbE [f_n(\go)  ]=0\qquad \text{ and } \qquad 
  \lim_{n\to \infty} \frac{\var_{\bbP'_{\beta} \otimes \bQ}(f'_n(\go',B))}{\bbE'_{\beta}\otimes \bQ[f'_n(\go',B)]^2}=0. 
\end{equation}
\end{itemize}
Then we have 
\begin{equation*}
 \lim_{a\to 0} \bbE\Big[  (\cZ^{\go,a}_{\beta,t} )^{\frac12}\Big]=0.
\end{equation*}

\end{lemma}

\begin{proof}
Consider $\gep>0$ and let $n$ be sufficiently large such that both quantities in \eqref{limiitz} are smaller than $\gep$. Furthermore,  let $a=a_n$ be such that 
$f_n$ is $\cG_a$-measurable.
Defining $K:= \log \gep $ and
$G(\go):=\exp(-K\ind_{\{f_n(\go)\ge 1\}})$, we
have by the Cauchy--Schwarz inequality,
\begin{equation*}
 \bbE\Big[  (\cZ^{\go,a}_{\beta,t} )^{\frac12}\Big]
 \le  \bbE\Big[G(\go)^{-1} \Big]^{\frac12}\bbE\Big[ G(\go)\cZ^{\go,a}_{\beta,t}\Big]^{\frac12} \,.
\end{equation*}
Clearly, by Markov's inequality,
\begin{equation*}
 \bbE[G(\go)^{-1}]\le 1 + e^{K} \bbP(f_n(\go)\ge 1)\le 2\,.
\end{equation*}
Using Lemma \ref{lem:sizebias} and the fact that $f_n$ is $\cG_a$-measurable, 
we further have (recall the notation \eqref{shorthat})
\begin{equation*}
 \bbE[G(\go)\cZ^{\go,a}_{\beta,t}]=  \bbE\otimes\bbE'_{\beta}\otimes \bQ\Big[\widehat G_{a,t}(\go,\go',B)\Big].
\end{equation*}
As $f_n$ is integer-valued, the inequality \eqref{magic} implies that 
\begin{equation*}
 \widehat G_{a,t}(\go,\go',B)\le e^{-K}+ \ind_{\{\wh f_n(\go,\go',B)=0\}}\leq e^{-K} + \ind_{\{f'(\om',B)=0\}},
\end{equation*}
which, by an application of the Chebychev inequality, results in
  \[
  \bbE\otimes\bbE'_{\beta}\otimes \bQ\Big[\widehat G_{a,t}(\go,\go',B)\Big]\le e^{-K}+ \bbP'_{\beta} \otimes \bQ ( f'_n(\go',B)\leq 0  ) \le e^{-K}+\gep=2\gep.
  \]
Because $a\mapsto (\calz^{\om,a}_{\beta,t})^{1/2}$ is a supermartingale, we obtain that 
\begin{equation*}
 \limsup_{a'\to 0} \bbE\Big[  (\cZ^{\go,a'}_{\beta,t} )^{\frac12}\Big]\le \bbE\Big[  (\cZ^{\go,a}_{\beta,t} )^{\frac12}\Big]\le 2 \sqrt{\gep},
\end{equation*}
which finishes the proof since $\gep$ is arbitrary.
\end{proof}

Before we specify the sequence of functions $f_n$, we need to introduce a few parameters.
We let $k_n:=1+ \lfloor n^{1/3} \rfloor$ and we set for $n\geq 1$
\[
M_n := \prod_{i=1}^{n}  \frac{k_{i-1}(k_{i-1}+\frac12)}{(k_{i-1}+1)(k_{i-1}-\frac12)}  \qquad \text{and}
\qquad b_n := e^{ -M_n} \,.
\]
Let us stress that since we have 
$$ \frac{M_{n+1}}{M_n}=1+ \frac{1}{2 k^2_{n}}+O(k^{-3}_{n}) $$ 
we have $\log M_n \sim  \frac32 \, n^{1/3} \sim \frac32 k_n $ as $n\to\infty$.
 Now, as the quantity $ {|\log z|}/{(\log \lvert \log z\rvert)^{5+4/d +\gep}}$ is of order $M_n (\log M_n)^{-(5+4/d +\gep)}$ on the interval $[b_{n+1},b_n)$, the assumption~\eqref{eq:log1-eps}  is equivalent to
\begin{equation}
\sum_{n= 1}^\infty \frac{M_n}{(\log M_n)^{5+\frac4d +\gep}} \int_{[b_{n+1},b_n)} z^{1+\frac{2}{d}} \,\gl(\dd z)=\infty \,.
\end{equation}
Then, using that $\log M_n \sim c \, n^{1/3} $ and in particular that $M_n^{1/k_n} \leq C$ for some constant $C>0$,
we get that
\begin{equation}\label{trucdelimsup} 
 \limsup_{n\to \infty}  k_n^{-2(1+\frac2d) -\frac{\gep}{2}}   M_n^{1-\frac{1}{k_n}} \int_{[b_{n+1},b_n)} z^{1+\frac{2}{d}} \,\gl(\dd z)=\infty \,,
\end{equation}
because $k_n^{-3-\frac{\gep}{2}}$ is summable.
Considering a subsequence if necessary, we assume from now on, without loss of generality, that~\eqref{trucdelimsup} holds with a limit instead of a $\limsup$. 
We set
\begin{equation}
\label{def:Vn}
V_n := M_n^{1-\frac{1}{k_n}} \int_{[b_{n+1},b_n)} z^{1+\frac{2}{d}} \,\gl(\dd z) \,,
\end{equation}
so that $\lim_{n\to+\infty} k_n^{-2(1+\frac2d) -\frac{\gep}{2} }  V_n = \infty $. We let   $\delta =\delta_{\gep} $
be sufficiently small so that
\begin{equation}
\label{def:delta}
\frac{d-2\delta}{d+2}  \left( 2 (1+\tfrac2d) + \tfrac{\gep}{2} \right) >2 \,,
\end{equation}
 and we set
\begin{equation}\label{Reta}
	 R_n := V_n^{ \delta /d }  \qquad \text{and} \qquad  \eta_n:= 1- \Bigg[(2\pi)^{-\frac12}\int_{[-R_n,R_n]}e^{-\frac{u^2}{2}}\, \dd u \Bigg]^{dk_n} .
\end{equation}
One can easily check that $\lim_{n\to\infty} R_n =   \infty$ and also  $\lim_{n\to \infty} \eta_n=0$ (using that $V_n \geq k_n^{2(1+2/d)}$).

\smallskip
We are now ready to introduce $f_n$. As in the proof of Proposition \ref{prop:dgeq2}, we want to count clusters of $k_n$ atoms that are beneficial to visit; note that here we have $k_n \to \infty$. More precisely, assuming without loss of generality that $t=1$, we define
\begin{equation} \label{anotherf}
\mathtoolsset{multlined-width=0.9\displaywidth}\begin{multlined} f_n(\om):=  \int_{ \X^{(k_n)}_1}   \bone_{\{ (\Delta t_{\lint k_n\rint}  )^{- {d}/{2}}z_{\lint k_n\rint} \ge A_n,\   \forall i\in\lint 2,k_n\rint: \  \Delta t_i \in [b_n^{ {2\theta'_n}/{d}}, b_n^{ {2\theta_n}/{d}}]   \}} \\
	\times\bone_{\{\forall i\in\lint k_n\rint:\ \lVert \Delta x_i\rVert_\infty \leq R_n \sqrt{\Delta t_i},\ z_i \in [b_{n+1},b_n)\}}\prod_{i=1}^{k_n} \delta_\om(\dd t_i,\dd x_i,\dd z_i)\,,
	\end{multlined}
\end{equation}
where we have used the shorthand notation
\begin{equation}
\label{thetathetaprim}
A_n:= V_n^{ \frac{d }{d+2} (1+\delta) k_n }\,,\qquad 
 \theta_n := \frac{2k_n+1}{2k_n-1} 
\qquad \text{ and } \quad \theta_n' := \frac{k_n}{k_n-1}\,.
\end{equation}
Note that $f_n$ is $\cG_{b_{n+1}}$-measurable.
The restriction on the range for $x_i$ and the requirement that
$ (\Delta t_{\lint k_n\rint}  )^{- {d}/{2}}z_{\lint k_n\rint}$  be large are   important features of $f$. The remaining constraints are \textit{ad hoc} and mainly serve the purpose to lighten the computations.
	If $z_i$ and $t_i$ satisfy the constraints above, then recalling the definition of $\theta_n'$, we have
	\begin{equation}\label{home}
t^{\frac{d}{2}}_1\le A_n^{-1}( \Delta t_{\lint 2, k_n\rint} )^{-\frac{d}{2}} z_{\lint k_n\rint } \le A_n^{-1} \exp (  M_n  [ (k_n-1)\theta_n'-k_n ]  ) = A_n^{-1}  \,.
\end{equation}
Since $\lim_{n\to\infty} A_n =\infty$ (recall that $\lim_{n\to\infty} V_n =\infty$)
and $k_n b_n^{ {2\theta_n}/{d}} \leq k_n \exp(- \frac2d M_n)$ goes to~$0$, 
we have 
$A_n^{-1} + (k_n-1) b_n^{2\theta_n/d}\le 1$ for  sufficiently large $n$,
and hence
\begin{equation*}
\left\{\big(\Delta t_{\lint k_n\rint} \big)^{-\frac{d}{2}}z_{\lint k_n\rint} \ge A_n, \    \forall i\in\lint 2,k_n\rint: \  \Delta t_i \in [b_n^{ {2\theta'_n}/{d}}, b_n^{ {2\theta_n}/{d}}]   \right\}
 \Longrightarrow  t_{k_n}  \le 1. 
\end{equation*}
Consequently, integrating first with respect to the $x_i$'s,  making a change of variables $s_i=\Delta t_i$ and then integrating only with respect  to $s_{k_n}$,  we have
\begin{equation}
\label{calculEf}
\begin{split}
 \bbE  [f_n(\go) ]&=  (2R_n)^{dk_n}\int_{ [b_n^{ {2\theta'_n}/{d}}, b_n^{ {2\theta_n}/{d}}]^{k_n-1}
 \times [0,1] \times [b_{n+1},b_n)^{k_n} } \ind_{\{ (s_{\lint k_n\rint})^{ {d}/{2}} \le A_n^{-1}z_{\lint k_n\rint} \}}\prod_{i=1}^{k_n} s^{\frac{d}{2}}_{i} \,  \dd s_i\,\la( \dd z_i)
 \\ &=\frac{2}{d+2} (2R_n )^{dk_n}A_n^{- (1+\frac{2}{d} )} \int_{ [b_n^{ {2\theta'_n}/{d}}, b_n^{ {2\theta_n}/{d}}]^{k_n-1}
  \times [b_{n+1},b_n)^{k_n} }  \prod_{j=1}^{k_n-1} s^{-1}_{j}\, \dd s_j \prod_{i=1}^{k_n} z^{1+\frac{2}{d}}_i \,\la(\dd z_i)
 \\ &= \frac{2}{d+2} 2^{dk_n} V_n^{-k_n} \bigg(\frac{2(\theta'_n-\theta_n)}{d} M_n\bigg)^{k_n-1} \bigg(\int_{[b_{n+1},b_n)}  z^{1+\frac{2}{d}}\,\la(\dd z) \bigg)^{k_n} \,,
\end{split}\!
\end{equation}
where we have used the definition of $R_n$ and $A_n$ in the last line.
Using the definition of $V_n$, together with the fact that $\theta'_n-\theta_n \leq \frac{1}{2 (k_n-1)^2}$,
we get that
\[
 \bbE  [f_n(\go) ] \leq \frac{2^{1+d}}{d+2}  \bigg(  \frac{2^d}{d (k_n-1)^2}\bigg)^{k_n-1} \,.
\]
Since $k_n \to\infty$,
this shows the first condition  in~\eqref{limiitz}.

\smallskip 

 To check the second condition in~\eqref{limiitz}, we  compute the expectation and variance of 
\begin{equation}\label{eq:barf}
 \bar  f_n(\om'):=  \int_{ \mathfrak{X}_{k_n}(1) \times [b_{n+1},b_n)^{k_n}}  \bone_{\{(\Delta t_{\lint k_n\rint}  )^{- {d}/{2}}z_{\lint k_n\rint} \ge A, \    \forall i\in\lint 2,k_n\rint: \  \Delta t_i \in [b_n^{ {2\theta'_n}/{d}}, b_n^{ {2\theta_n}/{d}}]   \}} \prod_{i=1}^{k_n} \delta_{\om'}(\dd t_i,\dd z_i). \!
\end{equation}
Recalling the notation~\eqref{Reta} of $\eta_n$, we have
$
 \bQ ( \forall i\in\lint k_n\rint: \lVert\Delta B_{t_i}\rVert_{\infty}\le R\sqrt{ \Delta t_i}  )= 1-\eta_n\,.
$
Hence, repeating the computations  that led to \eqref{babarf} in the previous section, we obtain 
\begin{equation*}\begin{split}
 \bbE'_{\beta }\otimes \bQ [  f'_n(\om',B) ] &= (1-\eta_n) \bbE'_{\beta}[\bar  f_n(\om')],\\ 
  \var_{\bbP'_{\beta}\otimes \bQ} [  f_n'(\om',B) ] &\le  \var_{\P'_\beta}(\bar f_n (\go'))+ 2\eta_n \bbE'_{\beta} [ \bar f_n(\go') ]^2.
\end{split}\end{equation*}
Since $\eta_n$ tends to zero, the reader can check that the second condition in \eqref{limiitz} follows from the following claims: 
\begin{equation} \label{finalclaimz}
\begin{split}
 \lim_{n\to \infty} \bbE'_\beta [ \bar f_n(\go') ] &=\infty,\\
 \var_{\bbP'_{\beta}}(\bar f_n(\go'))&\le 2 \bbE'_\beta [ \bar  f_n(\om') ] \qquad  \text{ for  sufficiently large $n$.}
\end{split}
\end{equation}

The first property in \eqref{finalclaimz} follows from
direct calculation. We have
\begin{equation}
\label{calculEprime}
\begin{split}
\bbE'_\beta [ \bar  f_n(\om') ] & =
\int_{ [b_n^{ {2\theta'_n}/{d}}, b^{ {2\theta_n}/{d}}]^{k_n-1}
 \times [0,1] \times [b_{n+1},b_n)^{k_n} } \ind_{\{ (s_{\lint k_n\rint})^{ {d}/{2}} \le A_n^{-1}z_{\lint k_n\rint} \}}\prod_{i=1}^{k_n}   \dd s_i\, \gb z_i \,\la( \dd z_i)  \\
&=  
\beta^{k_n} A_n^{-\frac{2}{d}} \bigg(\frac{2(\theta'_n-\theta_n)M_n}{d}\bigg)^{k_n-1}\bigg(\int_{[b_{n+1},b_n)}  z^{1+\frac{2}{d}}\,\la(\dd z) \bigg)^{k_n} \\
&=
 \frac{d}{2(\theta'_n -\theta_n)} \Big( \frac{2 \gb}{d} (\theta'_n-\theta_n)  V_n^{\frac{d-2\delta}{d+2}} \Big)^{k_n}  \,,
\end{split}
\end{equation}
where the computation is similar to~\eqref{calculEf}.
Now, since $\theta'_n-\theta_n \sim \frac12 k_n^{-2}$  and $ k_n^{-2(1+2/d) -  {\gep}/{2}} V_n \to\infty$, we have from our choice~\eqref{def:delta} for $\delta$
 that $  (\theta'_n-\theta_n)  V_n^{ ({d-2\delta})/({d+2})} \to \infty$,
which leads to the first part of~\eqref{finalclaimz}.

As in the previous section, to compute the second moment of $\bar f_n(\go')$, we can use \eqref{eq:expand} to write $\var_{\P'_\beta}(\bar f_n(\om'))$ in the form \eqref{eq:var}. As seen before, only sets with $\lvert I\rvert=\lvert J\rvert$ contribute,  and the term corresponding to  $I=J=\lint k_n\rint$ is simply $\bbE'_\beta[ \bar  f_n(\om')]$.
If $\lvert I\rvert=\lvert J\rvert=\ell \in \lint k_n-1\rint$, then
\begin{multline*}
 \bbE'_\beta [ f_I(\om')f_{J}(\go') ] = \beta^{2k_n-\ell} \int_{((0,1)\times[b_{n+1},b_n))^{2k_n-\ell}} \bigg(   \ind_{\mathcal{R}}\Big( (t^{(1)}_i)_{i=1}^{k_n},(z^{(1)}_i)^{k_n}_{i=1} \Big) \prod_{j\in I^{\cc}} z^{(1)}_i\,\dd t^{(1)}_j \,\gl (\dd z^{(1)}_i)   \bigg)\\ 
 \times\bigg(   \ind_{\mathcal{R}}\Big( (t^{(2)}_i)_{i=1}^{k_n},(z^{(2)}_i)^{k_n}_{i=1} \Big) \prod_{j\in J^{\cc}} z^{(2)}_i \,\dd t^{(2)}_j  \,\gl (\dd z^{(2)}_i)\Bigg)\prod_{i\in I}z_i\,\dd t_i\, \gl(\dd z_i),
\end{multline*}
where $\mathcal{R}$ is the subset of $((0,1)\times[b_{n+1},b_n))^{k_n}$ induced by the indicator function in \eqref{eq:barf} and for  $i\in I$, we used the notation
$$ t^{(1)}_i=t^{(2)}_{\sigma(i)}=t_i\qquad\text{and}\qquad z^{(1)}_i=z^{(2)}_{\sigma(i)}=z_i  \,.$$
Next, we relax the constraints imposed by $\mathcal{R}$
by only keeping 
$\Delta t^{(r)}_i \le b_n^{ {2\theta_n}/{d}}$.
 Then we proceed as in   the proof of \eqref{boundconstraint}:
by choosing an optimal order, we obtain  that  integration over each time variable produces a factor of at most $b_n^{ {2\theta_n}/{d}}$, except for the last integration, which produces a factor of one. Thus,
\begin{align*}
 \bbE'_\beta[ f_I(\om')f_{J}(\go') ] &\le  \beta^{2k_n-\ell} \, b_n^{\frac{2\theta_n}{d}(2k_n-\ell-1)}\,  \bigg( \int_{[b_{n+1},b_n)} z \,\gl (\dd z) \bigg)^{2k_n-\ell}\\ 
 &\le   \beta^{2k_n-\ell}    \exp\bigg( \frac2d \Big[  (2k_n-\ell)M_{n+1}-(2k_n-\ell-1)\theta_n M_n\Big] \bigg)\mu_{0,1}(1+\tfrac{2}{d})^{2k_n-\ell},
\end{align*}
where in the second line we simply used 
the fact that  
\begin{equation*}
   \int_{[b_{n+1},b_n)} z\, \gl (\dd z)  \le (b_{n+1})^{-\frac{2}{d} } \mu_{0,1}(1+\tfrac{2}{d}) \,,
\end{equation*}
together with the definition of $b_n$ and $b_{n+1}$. 
Recalling the definition of $\theta_n$ and of $M_n$,
we have $M_{n+1} = \frac{k_n}{k_n+1} M_n \theta_n$, so that
\[
 (2k_n-\ell) M_{n+1}  -(2k_n-\ell-1)\theta_n M_n = M_n \theta_n  k_n \left[ \frac{2k_n-\ell}{k_n+1} - \frac{2k_n- \ell-1}{k_n}\right] \leq 0 
\]
for any $\ell \leq k_n-1$. 
We therefore end up with 
\[
\begin{split}
 \var_{\P'_\beta}(\bar f_n(\go')) \leq \bbE'_\beta[ \bar  f_n(\om')] + \sum_{\ell=1}^{k_n-1} \binom{k_n}{\ell}^2  \beta^{2k_n-\ell}   \mu_{0,1}(1+\tfrac{2}{d})^{2k_n -\ell}   \leq  \bbE'_\beta[ \bar  f_n(\om')] + C'(d,\lambda,\gb)^{k_n}\,.
 \end{split}
 \]
Recalling~\eqref{calculEprime}, we have that for any constant $C>0$
the expectation $\bbE'_\beta[ \bar  f_n(\om')]$ is larger than $ C^{k_n}$ for sufficiently large  $n$.
 This proves the second claim in \eqref{finalclaimz} and concludes the proof.
\qed

\appendix

\section{Technical results} \label{app:A}

\subsection{Proof of Theorem~\ref{thm:dec}}

The result is contained in \cite{Kallenberg17}, but only implicitly, so we give a short proof. Let us first check that  \eqref{eq:fcond}
ensures that all integrals in \eqref{eq:dec-2} are well defined and finite.
Applying the BDG inequality and using the subadditivity~\eqref{eq:subbaditiv} of the function $x\mapsto x^{p/2}$ for $x>0$ iteratively, we deduce that
\beq\label{eq:BDG}\begin{split}
	&\E\Biggl[ \biggl\lvert\int_{ \X^N} f(\bw)\,M^\om(\dd w_1)\cdots\,M^\om(\dd w_N) \biggr\rvert^p \Biggr] \\
	&\qquad\leq C_p\E\Biggl[\Biggr(\int_\X \Biggl(\int_{\X^{N-1}} f(\bw)\,M^\om(\dd w_1)\cdots\,M^\om(\dd w_{N-1})\Biggr)^2\,\delta_\om(\dd w_N)\Biggr)^{p/2} \Biggr]\\
	&\qquad\leq C_p\E\Biggl[ \int_\X \biggl\lvert\int_{\X^{N-1}} f(\bw)\,M^\om(\dd w_1)\cdots\,M^\om(\dd w_{N-1})\biggr\rvert^p\,\delta_\om(\dd w_N)  \Biggr]\\
	&\qquad =C_p\int_\X \E\Biggl[ \biggl\lvert\int_{\X^{N-1}} f(\bw)\,M^\om(\dd w_1)\cdots\,M^\om(\dd w_{N-1})\biggr\rvert^p  \Biggr]\,\nu(\dd w_N)\\
	&\qquad\leq \cdots\leq C_p^N \int_{ \X^N} \lvert f(\bw)\rvert^p \,\nu(\dd w_{1})\cdots\nu(\dd w_N)<\infty.
\end{split}\eeq
These inequalities remain   unchanged if $\om$ is replaced by $\om_1,\dots,\om_N$ and $\E$ is replaced by $\E^{\otimes N}$.

\smallskip

Now we move to the proof of \eqref{eq:dec-2}, for which we shall prove the second inequality.  To obtain the reverse inequality it is sufficient follow the same proof and observe  that all estimates are two-sided (\textit{i.e.}, one can always substitute ``$\geq$'' for ``$\leq$'' if one also replaces $1/C$ by $C$). 
We now consider a  state space on which $\go$ is jointly defined with our i.i.d.\ copies $\go_1,\dots,\go_N$ (and $\go$ is independent of $(\go_1,\dots,\go_N)$).  We let $\bar \bbP:= \bbP\otimes \bbP^{\otimes N}$ denote the associated probability.  
In analogy with \eqref{filtraF}, we consider a filtration $(\bar \cF_t)_{t\ge 0}$ on this state space  defined by
\begin{equation*} 
\bar \cF_t:= \cF_t \otimes \cF^{(1)}_t \otimes \cdots \otimes \cF^{(N)}_t,\quad \cF^{(i)}_t:= \sigma\Big( \go \cap ([0,t]\cap \bbR^d \times (0,\infty))\Big),\quad i=1,\dots,N.
\end{equation*}
Our result follows from applying the following inequality, valid for $i=1,\dots, N$, iteratively:
\begin{equation}\label{compachain}\mathtoolsset{multlined-width=0.9\displaywidth}\begin{multlined}
 \bar \bbE\bigg[ \int_{ \X^N} f(\bw)\,M^\om(\dd w_1)\cdots M^{\om}(\dd w_i)\, M^{\om_{i+1}}(\dd w_{i+1})\dots M^{\om_N}(\dd w_N)\bigg]\\
 \le  \frac{C}{1-p} \bar \bbE\bigg[ \int_{ \X^N} f(\bw)\,M^\om(\dd w_1)\cdots
 M^{\om}(\dd w_{i-1})\,M^{\om_i}(\dd w_i)\cdots M^{\om_N}(\dd w_N)\bigg].\end{multlined}
\end{equation}
Let us first spend some time on the first step $i=N$.
By elementary properties of It\^o integrals, the processes 
 $(X_t)_{t\ge 0}$ and $(Y_t)_{t\ge 0}$ defined by
 \begin{equation*}\begin{split}
               X_t&\ceq  \int_{ \X^N} f(\bw)\bone_{(0,t]}(t_N)\,M^\om(\dd w_1)\cdots M^\om(\dd w_N),\\
                     Y_t&\ceq \int_{ \X^N} f(\bw)\bone_{(0,t]}(t_N)\,M^{\om}(\dd w_1)\cdots M^{\om}(\dd w_{N-1})\, M^{\om_N}(\dd w_N)
                 \end{split}
                 \end{equation*}
are martingales with respect to $(\bar \calf_t)_{t\geq0}$, whose quadratic variation processes  are given by
\begin{align*}
	[X]_t&= \int_\X\Biggl(\int_{ \X^{N-1}} f(\bw)\bone_{(0,t]}(t_N)\,M^\om(\dd w_1)\cdots M^\om(\dd w_{N-1})\Biggr)^2\,\delta_{\om}(\dd w_N),\\
	[Y]_t&= \int_\X\Biggl(\int_{ \X^{N-1}} f(\bw)\bone_{(0,t]}(t_N)\,M^\om(\dd w_1)\cdots M^\om(\dd w_{N-1})\Biggr)^2\,\delta_{\om_{N}}(\dd w_N),
\end{align*}
respectively. Since $\om$ and $\om_N$ are Poisson random measures with the same $(\bar\calf_t)_{t\geq0}$-intensity measures (namely $\dd t\otimes \dd x\otimes \la(\dd z)$), the jump measures associated to $[X]$ and $[Y]$  have the same predictable compensator in $(\bar\calf_t)_{t\geq0}$. As a result, $X$ and $Y$
are  \emph{weakly tangential} martingales 
in the sense of \cite[Section~3]{Kallenberg17}. 
Thus, by \cite[Thm. 4.1]{Kallenberg17} and Doob's inequality, there is $C_p>0$ such that
\begin{equation} \label{fdsb}
\bar \E [ \lvert X_{\infty}\rvert^p ]\leq \bar \E\bigg[  \sup_{t\ge 0} {\lvert X_t\rvert^p} \bigg] \leq C_p  \bar\E\bigg[  \sup_{t\ge 0} {\lvert Y_t\rvert^p} \bigg] \leq C_p\bigg(\frac{ p}{p-1}\bigg)^p \bar\E[|Y_\infty|^p].
\end{equation}
The reader can check that for  the implicit constant $C_p$ from  \cite[Thm. 4.2]{Kallenberg17},  one may take $
 C_p:=(24p^2)^p [ 2\times 3^{p/2} ( 2^{p/2+1}+14\times 3^{p/2}(28\times 3^{p/2})^{p/2} ) ]$.
The expression in brackets 
comes from the last equation in  \cite[p.\ 38]{Kallenberg17} (with $\vp(x)=\lvert x\rvert^{p/2}$ and $c=(14\times 3^{p/2})^{-1}$), while the computation in \cite[p.\ 39]{Kallenberg17}, combined with the BDG inequality as in \cite[Ch.\ VII, Thm.\ 92]{Del82}, entails an additional factor of $(24p^2)^p$.  
Most importantly, $C_{p}$ is bounded uniformly in $p\in (1,2]$, so that there exists a universal constant $C$ such that for $p\in (1,2]$,
\begin{equation*}
  \bar\E [ \lvert X_{\infty}\rvert^p ]\ \le \bigg(\frac{C}{p-1}\bigg)\bar\E [ \lvert Y_{\infty}\rvert^p ],
\end{equation*}
which concludes the proof of \eqref{compachain} for $i=N$.

\medskip

In order to iterate and  prove \eqref{compachain} for $i \le N-1$, we 
wish to interchange $M^\om(\dd w_1)\cdots M^\om(\dd w_{i})$ with $M^{\om_{i+1}}(\dd w_{i+1})\cdots M^{\om_{N}}(\dd w_{N})$, that is, we want to write
\begin{equation}\label{eq:fubini}  \mathtoolsset{multlined-width=0.9\displaywidth}\begin{multlined} \int_{ \X^N} f(\bw)\,M^\om(\dd w_1)\cdots M^{\om}(\dd w_i)\, M^{\om_{i+1}}(\dd w_{i+1})\cdots M^{\om_N}(\dd w_N)\\
	=\int_{\X^{i}} \Biggl(\int_{\X^{N-i}} f(\bw)\,M^{\om_{i+1}}(\dd w_{i+1})\cdots M^{\om_N}(\dd w_N)\Biggr) \,M^{\om}(\dd w_1)\cdots M^\om(\dd w_{i}).\end{multlined} \end{equation}
Even though the  integral on the right-hand side is anticipative when considering the filtration $(\bar \cF_t)_{t\geq0}$,
we can recover an integral in It\^o's sense by constructing the inner integrals 
$$\int_{\X^{N-i}} f(\bw)\,M^{\om_{i+1}}(\dd w_{i+1})\cdots M^{\om_N}(\dd w_N)$$
using the filtration $(\bar \cF_t)_{t\geq0}$ and the outer integrals using the filtration $\bar \cF^{(i)}$ defined 
by 
\begin{equation*}
 \bar \cF^{(i)}_t:= \cF_t \otimes \cF^{(1)}_t\otimes \cdots \otimes \cF^{(i)}_t \otimes \cF^{(i+1)}_\infty\otimes \cdots   \otimes \cF^{(N)}_{\infty}.
\end{equation*}
Note that the inner integrals are $\bar\calf^{(i)}_0$-measurable.
With this convention, we can justify \eqref{eq:fubini} as follows: It  certainly holds if $f$ is a step function, that is, if $f$ only assumes finitely many values (in that case, the  integrals are simply finite sums). 
For general $f$, take a sequence of step functions $(f_n)_{n\in\N}$ such that $\lvert f_n\rvert\leq \lvert f\rvert$ for all $n$ and $f_n\to f$ pointwise as $n\to\infty$. Equation \eqref{eq:fubini} holds for $f_n$, and arguing similarly to \eqref{eq:BDG} and using  dominated convergence, we have that
\begin{align*} 
\lim_{n\to \infty}\bar \bbE \Biggl[\Biggl\lvert \int_{ \X^N} (f-f_n)(\bw)\,M^\om(\dd w_1)\cdots M^{\om}(\dd w_i) \,M^{\om_{i+1}}(\dd w_{i+1})\cdots M^{\om_N}(\dd w_N)\Biggr\rvert^p\Biggr] &=0 ,\\
\lim_{n\to \infty} \bar\E\Biggl[\Biggl| \int_{\X^{i}} \Biggl(\int_{\X^{N-i}} (f-f_n)(\bw)\,M^{\om_{i+1}}(\dd w_{i+1})\cdots M^{\om_N}(\dd w_N)\Biggr) \,M^{\om}(\dd w_1)\cdots M^\om(\dd w_{i})\Biggr|^p\Biggr] &=0, 
\end{align*}
proving \eqref{eq:fubini}. The fact that we are able to interpret the latter integral in It\^o's sense is crucial for the BDG inequality, which was needed in \eqref{eq:BDG}, to apply.
Once \eqref{eq:fubini} is established, we can prove \eqref{compachain}  by considering the weakly tangential martingales (for the filtration $(\bar \cF^{(i)}_t)_{t\geq0}$)
\begin{align*} 
           X^{(i)}_t&\ceq \int_{\X^{i}} \Biggl(\int_{\X^{N-i}} f(\bw)\ind_{(0,t]}(t_i)\,M^{\om_{i+1}}(\dd w_{i+1})\cdots M^{\om_N}(\dd w_N)\Biggr) \,M^{\om}(\dd w_1)\cdots M^\om(\dd w_{i}),\\
                     Y^{(i)}_t&\ceq \int_{\X^{i}} \Biggl(\int_{\X^{N-i}} f(\bw)\ind_{(0,t]}(t_i)\,M^{\om_{i+1}}(\dd w_{i+1})\cdots M^{\om_N}(\dd w_N)\Biggr) \,M^{\om}(\dd w_1)\cdots M^{\om}(\dd w_{i-1})\,M^{\om_i}(\dd w_{i}),
\end{align*}
applying the same estimate as in \eqref{fdsb} and re-arranging the integrals similarly to \eqref{eq:fubini}.

\subsection{Proof of Lemma~\ref{lem:integrab}}
 For $A, u,\eps>0$ and $H\in \cals$ that satisfies $\lvert H\rvert\leq u\lvert K\rvert$, we have that
\begin{multline*}
	\P\Biggl( \Biggl\lvert \int_\X H(\om,t,x)\,\xi_{\om_<}(\dd t,\dd x) \Biggr\rvert>\eps\Biggr) \leq \P_\geq \Bigl( A u \lVert K\rVert_{\xi_{\om_<},p; \P_{<}}>\eps\Bigr)\\ +  \P\Biggl( \Biggl\lvert \int_\X H(\om,t,x)\,\xi_{\om_<}(\dd t,\dd x) \Biggr\rvert>uA\lVert K\rVert_{\xi_{\om_<},p; \P_{<}}\Biggr). 
\end{multline*}
The second probability is bounded by
$$
\E_{\geq}\Biggl[ \frac{1}{u^pA^p\lVert K\rVert^p_{\xi_{\om_<},p; \P_{<}}}  \E_<\left[ \left\lvert \int_\X H(\om,t,x)\,\xi_{\om_<}(\dd t,\dd x) \right\rvert^p \right] \Biggr]\leq A^{-p}.
$$
Therefore, using dominated convergence for the first term, we get
\begin{equation*}
	\lim_{u\to0}\sup_{H\in\cals,\lvert H\rvert\leq u\lvert K\rvert}\P\Biggl( \Biggl\lvert \int_\X H(\om,t,x)\,\xi_{\om_<}(\dd t,\dd x) \Biggr\rvert>\eps\Biggr) \leq \lim_{u\to0}  \E_\geq [\bone_{\{ A u \lVert K\rVert_{\xi_{\om_<},p; \P_{<}}>\eps\}}]+A^{-p}\leq A^{-p}. 
\end{equation*}
Sending $A\to\infty$ shows that the left-hand side is $0$, which is obviously equivalent to $\lVert uK\rVert_{\xi_<,0}\to0$ as $u\to0$. Therefore, $K$ is integrable with respect to $\xi_{\om_<}$ by \eqref{eq:integ}.
\qed

\bibliographystyle{abbrv}
\bibliography{biblio.bib}

\end{document}